\documentclass[11pt]{article}
\usepackage[textsize=footnotesize]{todonotes}

\usepackage{mathrsfs}
\usepackage{graphicx}
\usepackage{amsmath}        
\usepackage{amssymb}
\usepackage{amsthm}
\usepackage[countmax]{subfloat}  
\usepackage{fourier}
\usepackage{xcolor}
\usepackage{mathtools}
\usepackage{enumerate}

\usepackage{chngcntr}
\counterwithin{figure}{section}
\usepackage{url}
\usepackage{verbatim}
\usepackage{booktabs}
\usepackage{placeins}
\usepackage{multirow}
\usepackage{caption}
 \usepackage[numbers]{natbib} 
\usepackage[title]{appendix}
\usepackage{subcaption}
\usepackage{hyperref}
\hypersetup
{
    pdfauthor={Anne van Delft},
    pdfsubject={A note on quadratic forms of stationary functional time series under mild conditions},
    pdftitle={A note on quadratic forms of stationary functional time series under mild conditions},
    pdfkeywords={functional data, time series, spectral analysis,martingales}
}

\newtheorem{prop}{Proposition}[section]     
\newtheorem{thm}{Theorem}[section]
\newtheorem{lemma}{Lemma}[section]

\newtheorem{Remark}{Remark}[section]
\newtheorem{Corollary}{Corollary}[section]
\newtheorem{assumption}{Assumption}[section]

\newtheorem*{lemmas}{Lemma 3.0$^\star$}

\theoremstyle{definition}

\newcommand{\im}{\mathrm{i}}

\newcommand{\mynegspace}{\hspace{-0.12em}}

\newcommand{\bigsnorm}[1]{\Big\rvert\mynegspace\Big\rvert\mynegspace\Big\rvert\mynegspace {#1} \Big\rvert\mynegspace \Big\rvert\mynegspace \Big\rvert}

\newcommand{\snorm}[1]{\rvert\mynegspace\rvert\mynegspace\rvert\mynegspace {#1} \rvert\mynegspace \rvert\mynegspace \rvert}

\newcommand{\norm}[1]{{\|}{#1}{\|}}
\newcommand{\bignorm}[1]{\Big{\|}{#1}\Big{\|}}
\newcommand{\Bignorm}[1]{\Bigg{\|}{#1}\Bigg{\|}}
\newcommand{\inprod}[2]{\langle #1, #2 \rangle}
\newcommand{\biginprod}[2]{\Big\langle #1, #2 \Big\rangle}
\newcommand\tageq{\addtocounter{equation}{1}\tag{\theequation}}
\DeclareMathOperator{\Tr}{Tr}

\newcommand{\lamx}[2]{\lambda^{(#1)}_{X,#2}}

\newcommand{\E}{\mathbb{E}}
\newcommand{\lpr}{{l^{\prime}}}
\newcommand{\F}{\mathcal{F}}
\newcommand{\G}{\mathcal{G}}

\newcommand{\op}{\mathcal{L}}
\newcommand{\cov}{\mathrm{Cov}}
\newcommand{\var}{\mathrm{Var}}
\newcommand{\znum}{\mathbb{Z}}

\newcommand{\rnum}{\mathbb{R}}
\newcommand{\cnum}{\mathbb{C}}
\newcommand{\hi}{\mathbb{H}}
\newcommand{\ldx}[1]{\mathcal{V}^{{#1}\lambda}_{T}}
\newcommand{\ldd}[1]{\mathcal{V}^{{#1}(m),\lambda}_{T}}
\newcommand{\ldm}[1]{\mathcal{M}^{{#1}(\lambda)}_{T,m}}

\newcommand{\ldmi}[2]{\mathcal{M}^{{#1}(\lambda_{#2})}_{T,m}}

\newcommand{\dm}[1]{D^{(\lambda)}_{m,#1}}
\newcommand{\dmi}[2]{D^{(\lambda_{#2})}_{m,#1}}

\newcommand{\blu}[1]{\textcolor{black}{#1}}
\allowdisplaybreaks
\usepackage{authblk}
\providecommand{\keywords}[1]{\textbf{\textit{keywords: }} #1}
\providecommand{\AMS}[1]{\textbf{\textit{AMS subject classification: }} #1}
\providecommand{\acknowledgements}[1]{\textbf{\textit{Acknowledgements}} #1}
\begin{document}

\title{A note on quadratic forms of stationary functional time series under mild conditions}

\author{Anne van Delft\thanks{anne.vandelft@rub.de}}
\affil{Ruhr-Universit\"at Bochum\\ Fakult\"at f\"ur Mathematik\\ 44780 Bochum, Germany}

\date{November 25, 2022}
\maketitle
\vspace*{-5pt}
\begin{abstract}
We study distributional properties of a quadratic form of a stationary functional time series under mild moment conditions. As an important application, we obtain consistency rates of estimators of spectral density operators and prove joint weak convergence to a vector of complex Gaussian random operators. Weak convergence is established based on an approximation of the form via transforms of Hilbert-valued martingale difference sequences.  As a side-result, the distributional properties of the long-run covariance operator are established.
\end{abstract}


\smallskip

 \AMS{Primary: 62HG99, 60G10, 62M15, Secondary 62M10.} 


\section{Introduction}

The subject of this paper \blu{is} quadratic forms of a stationary time series $\{X_t : t \in \znum\}$ with paths in some function space $H$. From a technical perspective, we shall adhere to existing literature and assume $H$ is a separable Hilbert space. Each realization is therefore a function. Such \textit{Functional time series} are of growing interest due to the fact that many processes are almost continuously measured on their domain of definition.
While quadratic forms of Euclidean-valued random variables have received considerable attention and have been studied under various dependence conditions \citep[see \blu{e.g.},][and references therein]{deJong87,Mikosch91,KokTaq1997,BhGiKok07,WuSh07,LeeSR17},
this is not so much the case for quadratic forms of \blu{function-valued} random variables. Yet, they do arise naturally in a variety of inference problems. A quadratic form statistic of a functional time series can be given by 
 \begin{align*}
 \hat{\mathcal{Q}}_{T}=\sum_{s,t =1}^{T}  \Phi_{T,s,t}( X_s \otimes X_t) \tageq \label{eq:Q1}
\end{align*}
where $\{\Phi_{T,t,s} \}_{t,s \in \{1, \ldots, T\}}$ \blu{is} a sequence of bounded linear operators, which will vary depending on the application. Important applications in which statistics of the form \eqref{eq:Q1} arise are those that concern the consistent estimation of the second order characteristics of the process. This is especially relevant \blu{for} functional data because the smoothness properties of the random functions are encoded in the second order structure and are key in obtaining optimal finite-dimensional representations. 
For example, if we denote $I_{H \otimes H}$ the identity operator on the tensor product space $H \otimes H$, then the specification $\Phi_{T,t,s}= \frac{1}{T}\mathrm{1}_{s=t} I_{H \otimes H}$ yields the sample covariance operator. In the case of $i.i.d.$ functional data, this object captures the full second order structure and its eigen decomposition plays a central role in the \blu{reduction} to finite dimension of the process's properties, e.g., via the Karhunen-Lo{\`e}ve representation if $H = L^2$. Not surprisingly, the sample covariance operator received considerable attention in the corresponding line of literature \citep[e.g.,][]{Grenander81,DauPoRo82,RicRam91}, but also in case of linear processes \citep[see among others][and references therein]{Bosq02,DehlShar05,HorvKok12}. However, when there is serial correlation between observations the covariance operator clearly does not capture the full dynamics. For dependent functional data, a more meaningful object is therefore the spectral density operator
\begin{align}
\F^{(\lambda)} =\frac{1}{2\pi}\sum_{h \in \znum} C_h e^{-\im \lambda h}, \quad \lambda \in (-\pi,\pi], \label{eq:sdo}
\end{align}
where $C_h$ is the $h$-lag covariance operator of the process $X \blu{:=\{X_t \colon t\in \znum\}}$. As an estimator for $\F^{(\lambda)}$ of a process $X$ with mean function $\mu$, one can consider 
\begin{align*}
\hat{\F}^{(\lambda)} =\frac{1}{2 \pi T}\sum_{s,t=1}^{T}\underbrace{w(b_T(t-s)) e^{\im \lambda (t-s)}}_{\phi_{T,t-s}^{(\lambda)}} \big( (X_s-\mu) \otimes (X_t-\mu) \big) ,  \quad \lambda \in (-\pi,\pi],
\end{align*}
which simply corresponds to the quadratic form in \eqref{eq:Q1} with $\Phi^{(\lambda)}_{T,t,s}=2\pi T \phi_{T,t-s}^{(\lambda)} I_{H \otimes H}$. Here $w(\cdot)$ is an even, bounded function on $\rnum$ that is continuous at zero and $b_T$ is a bandwidth parameter converging to zero at a rate such that $b_T T \to \infty$ as the sample size $T$ tends to infinity. The properties of this estimator and its relation to the smoothed periodogram operator are discussed in detail in  \autoref{sec:sec4}. For $\lambda=0$, $2\pi \hat{\F}^{(0)}$ is an estimator of the long-run covariance operator. Because it arises as the limiting covariance operator of the sample mean function, properties of the long-run covariance operator have been studied in several contexts within the framework of $L^p_m$-approximability \blu{\citep[see e.g.,][]{HorKokRe13,HorKok10,BerHorRi17}.}

Frequency domain analysis of functional time series, i.e., the case $\lambda \ne 0$, has received considerably less attention than time domain analysis. \blu{Yet, not only does frequency domain analysis (and hence the spectral density operator)
arise naturally in various applications, 
 it allows in particular to capture the full second order dynamics of dependent functional data.} It can therefore be seen to take on a similar role for dependent functional data as the covariance operator takes on in the case of $i.i.d.$ functional data. In fact, it allows to \blu{reduce} the uncountably infinite variation to a countably infinite space in an optimal manner via a \textit{dynamic} Karhunen-Lo{\`e}ve representation provided the function space is sufficiently smooth. Moreover, frequency domain based inference methods enable powerful nonparametric tools for hypothesis testing. Because of its relevance for dependent functional data, estimators of $\F^{(\lambda)}$ in the context of $L^p_m$-dependence as well as under functional cumulant-mixing conditions were introduced earlier this decade. Under $L^p_m$ approximability, \cite{HorHalL15} considered dynamic principal components for stationary functional time series and obtained a consistency result for a lag window estimator. Under cumulant-mixing conditions, \cite{pt13} derived consistency and asymptotic normality of a smoothed periodogram operator estimator. Estimation and distributional properties of an estimator for a time-varying spectral density operator were derived in \cite{vde16}, who introduced a framework for locally stationary functional time series. Note that all of the aforementioned estimators can be written in the form \eqref{eq:Q1}. It is worth mentioning that these works have paved the way for frequency domain-based inference of functional time series, leading to an upsurge in the available literature in the past few years \blu{\citep[see e.g.,][and references therein]{Zhang16,HorKokNis17,LPapSap2018,PhPan2018,vde18,KokJou2019}}. 

Cumulant tensors and spectral cumulant tensors can be shown to form Fourier pairs, provided appropriate summability conditions are satisfied. The consideration of functional cumulant mixing conditions as in \cite{pt13} can therefore to some extent be seen to provide a natural framework for the derivation of sampling properties. Yet, the central limit theorem and consistency result as derived in \cite{pt13} rely on existence of all moments and summability conditions of the cumulant tensors. In certain applications such required summability conditions might be too strong and worthwhile to be relaxed. To the best of the author's knowledge, there is currently no CLT  available under $L^p_m$-dependence and the consistency rate available in this setting \citep{HorHalL15} is sub-optimal compared to the one derived under cumulant mixing conditions in \cite{pt13}.

Broadly speaking, the goal of this paper is therefore twofold. We aim to derive a general central limit theorem for quadratic forms of stationary functional time series under sharp moment conditions. At the same time, we aim to obtain the best possible convergence and consistency rates for the aforementioned applications. It is worth mentioning that our conditions on the dependence structure are also weaker than those considered within the $L^p_m$-dependence framework. 
Underlying our approach is an approximation of the quadratic form with a Hilbertian-valued martingale process. To construct this process, we shall use a martingale approximation of the  quadratic form. The idea to approximate a normalized partial sum process via a related martingale process was first put forward by \cite{Gor69}. \cite{Wu05a} introduced this approach to derive distributional properties of the Discrete Fourier transform (DFT) of a Euclidean-valued ergodic time series. The latter has since then been applied in a variety of problems \citep[see e.g.,][]{PeWu10,WuSh07,LiuWu10}. In \cite{CevHor17}, the result of \cite{Wu05a} and \cite{PeWu10} was generalized to a CLT of the Discrete Fourier transform of a Hilbertian-valued time series.

The structure of this note is as follows. In \autoref{sec:sec2}, we introduce necessary notation and conditions. In \autoref{sec:QF}, we explain the approach in more detail and provide a joint central limit theorem for a set of quadratic forms as in \eqref{eq:Q1}. In \autoref{sec:sec4}, we focus on the estimation of the spectral density operator and long-run covariance as particular applications. More specifically, consistency rates and distributional properties are established. \textcolor{black}{In \autoref{sec:sec5}, we relate the mild assumptions made in this paper to functional cumulant mixing conditions.} Various technical results and proofs are relegated to the Appendix.

\section{Framework}\label{sec:sec2}

\blu{Throughout this paper, we will focus on random variables taking values in some separable Hilbert space, say  $H$. For elements of $H$, we shall denote the inner product by $\inprod{\cdot}{\cdot}$ and the induced  norm by $\|\cdot\|_H$.} We let $H_1 \otimes H_2$ denote the Hilbert tensor product of the Hilbert spaces $(H_j, \inprod{\cdot}{\cdot}_{H_j})_{j=1,2}$. This Hilbert space can be constructed from the algebraic tensor product $H_1 \otimes_{\mathrm{alg}} H_2$ together with a bilinear map $\psi: H_1 \times H_2 \to H_1 \otimes_{\mathrm{alg}} H_2$ that satisfies $\inprod{\psi(x_1,x_2)}{\psi(y_1,y_2)} = \inprod{x_1}{y_1}_{H_1}\inprod{x_2}{y_2}_{H_2}$ for $x_1, y_1 \in H_1$ and $x_2,y_2 \in H_2$ and then taking the completion with respect to the induced norm \cite[see e.g.,][for details]{KadRing97}. \blu{By the associative law, this can be extended to construct a Hilbert space, $\otimes_{i=1}^n H_i$, from the algebraic $n$-fold tensor product. 
Next, we denote by $S_{\infty}(H_1, H_2)$ the Banach space of bounded linear operators $A\colon H_1 \to H_2$ equipped with the operator norm $\snorm{A}_\infty$ $=\sup_{\|g\|_{H_1} \le1}\|Ag\|_{H_2}, g \in H_1$.  An operator $A \in S_\infty(H):=S_\infty(H,H)$ is called non-negative definite if $\inprod{Ag}{g}\ge 0$ for all $g \in H$. It is called self-adjoint if $\inprod{Af}{g} = \inprod{f}{A^{\dagger}g}=\inprod{f}{A g}$ for all $f,g \in H$, where $A^\dagger$ denotes the adjoint of $A$. The conjugate operator of $A$ can be defined as $\overline{A}g=\overline{(A\overline{g})}$, where $\overline{g}$ denotes the complex conjugate of $g \in H$. 
For $A, B, C \in S_{\infty}(H)$ we define the Kronecker tensor product $(A \widetilde{\otimes} B)C = ACB^{\dagger}$, while the transpose Kronecker tensor product is given by $(A \widetilde{\otimes}_{\top} B)C = (A \widetilde{\otimes} \overline{B})\overline{C}^{\dagger}$. A \blu{bounded linear} operator $A: H_1 \to H_2$ belongs to the class of Hilbert-Schmidt operators, denoted by $S_2(H_1,H_2)$, if it has finite Hilbert-Schmidt norm $\snorm{A}_2: = (\sum_{i=1}^{\infty}\|A(\chi_i)\|_{H_2}^2)^{1/2}$, \blu{where $\{\chi_i\}_{i \ge 1}$ is an arbitrary orthonormal basis of $H_1$}. We say $A$ is trace class and denote $A \in S_1(H)$ if $A$ has finite trace class norm $\snorm{A}_1 = \sum_{i=1}^{\infty}\inprod{ (AA^\dagger)^{1/2}(\chi_i)}{\chi_i}$. 
The space $(S_2(H_1,H_2),\snorm{\cdot}_2)$ is a Hilbert space with inner product $\inprod{A}{B}_S =\Tr(AB^\dagger) = \sum_{i=1}^{\infty} \inprod{A(\chi_i)}{B(\chi_i)}_{H_2}$, $A, B \in S_2(H_1,H_2)$.}
 For $f,g, v\in H$, define the tensor product $f \otimes g\colon H  \to H$ as the bounded linear operator $(f \otimes g)v =\inprod{v}{g}f$. The mapping  $\mathcal{T}\colon H \otimes H \to S_2(H)$ defined by the linear extension of $\mathcal{T}(f \otimes g) = f \otimes \overline{g}$ is an isometric isomorphism.

\blu{A $H$-valued random element $X$ over a probability space $(\Omega,\mathcal{A},\mathbb{P})$ is a strongly measurable function $X: (\Omega,\mathcal{A},\mathbb{P}) \to (H, \mathcal{B})$, where $\mathcal{B}$ denotes the $\sigma$-algebra on $H$. We denote $X \in \op^p_H$ if $\|X\|_{\mathbb{H},p}:=(\E\|X\|_{H}^p)^{1/p} <\infty$.  Observe that $\op^p_H$ is a Banach space w.r.t. the norm $\|X\|_{\mathbb{H},p}$ and for $p=2$ it is a Hilbert space when equipped with the inner product $\E\inprod{\cdot}{\cdot}$. }
For $X \in \op^1_H(\Omega,\mathcal{A},\mathbb{P})$ and $\mathcal{A}_o$ a sub-algebra of $\mathcal{A}$, we define the conditional expectation $\E[\cdot|\mathcal{A}_o]: \op^1_H(\Omega,\mathcal{A},\mathbb{P}) \to  \op^1_H(\Omega,\mathcal{A}_o,\mathbb{P}) $ to be the mapping such that
\[
\int_{A} \E[X|\mathcal{A}_o] d\mathbb{P}=\int_{A} X d\mathbb{P} \quad A \in \mathcal{A}_o,
\]
where the expectations should all be understood in the sense of a Bochner integral. Note that classical properties of conditional expectations remain valid in the context of separable Hilbert spaces. The cross-covariance operator between two zero-mean elements $X, Y \in \op^2_H$ is given by $\cov(X,Y)=\E(X \otimes Y)$ and belongs to $S_1(H)$. We note in particular that $\|X\|^2_{\hi,2}=\Tr(\text{Var}(X \otimes X))$, \blu{where $\var(X) = \cov(X,X)$}. For a filtration $\{\mathcal{G}_k\}$ of sub $\sigma$-algebras of $\mathcal{A}$, we shall make extensive use of projection operators defined by
\[
P_k = \E[\cdot|\G_k]-\E[\cdot|\G_{k-1}], \quad k \in \mathbb{Z}
\]
which are linear operators on $\op^1_H$ and are strongly orthogonal elements in $\op^2_H$, i.e., 
\[
\cov(P_k(X_1), P_j(X_2)) = O_H \quad \forall X_1, X_2 \in \op^2_H \text{ and } k \ne j \in \mathbb{Z}.
\]
Finally, we let $\Rightarrow_N$ indicate convergence in distribution as $N \to \infty$, where $N \in \mathbb{N}$.

\section{Main result}\label{sec:QF}

Throughout this article, we are interested in weakly stationary functional time series $\{X_t\colon t\in\mathbb{Z}\}$ taking values in $\op^2_H$. In particular this means that the mean $\E X_t=\mu$ and the $h$-lag covariance operator $C_h$ are invariant under translations in time, i.e, $C_h =\E (X_h-\mu) \otimes (X_0-\mu)$. Without loss of generality, we shall assume that the data are centered. When the mean is unknown one can consider centering the data by subtracting the sample mean function (see Remark \autoref{rem:mean}). Furthermore, we assume the process admits a representation of the form
\[X_t = g(\epsilon_t, \epsilon_{t-1},\ldots,)\]
where $\{\epsilon_t: t\in\mathbb{Z}\}$ is an i.i.d. sequence of elements in some \blu{measurable space $S$ and where  $g: S^{\infty} \to H$ is a measurable function}. Functional processes with such representation are widely applicable and allow for example for nonlinear dynamics\citep[e.g.,][]{HorKok10}. 
It is clear from this representation that $X$ is stationary and ergodic and we can consider the filtration $\G_t = \sigma(\epsilon_t, \epsilon_{t-1},\ldots)$.  Moreover, it is straightforward to show that a stationary ergodic process can be written as 
\[
X_h =\sum_{j=-\infty}^{h} P_j(X_h) \tageq \label{eq:ProjX}
\]
where the equality holds in $\op^2_H$\blu{\citep[see e.g.,][]{Pourahmadi01}}. Note that $\{P_j(X)\}$ \blu{is} a martingale difference sequence with respect to the backward filtration of $\sigma$-algebras $\{\G_{-j}:j>0\}$. In order to formulate conditions on the dependence structure we consider a generalized version of the physical dependence measure of \cite{Wu05b}. More specifically, let $\{\epsilon^{\prime}_t: t\in\mathbb{Z}\}$ be an independent copy of $\{\epsilon_t: t\in\mathbb{Z}\}$ defined on $(\Omega,\mathcal{A},\mathbb{P})$. For a set $I \subset \mathbb{Z}$, let $\G_{t,I}=\sigma(\epsilon_{t,I},\epsilon_{t-1,I},\ldots)$ where $\epsilon_{t,I} =\epsilon^\prime_t$ if $t \in I$ and $\epsilon_{t,I} =\epsilon_t$ if $t \not\in I$\blu{. As} a measure of dependence define \blu{the coefficients}
\[
\nu_{\hi,p}(X_t)=\|X_t -\E[X_t|\G_{t,\{0\}}] \|_{\hi,p}. \tageq \label{eq:depstruc3}
\]
Additionally, define the $m$-dependent version
\begin{align*}
X^{(m)}_{t}= P_{t,t-m}X_{t}=\E[X_t| \sigma(\epsilon_t, \epsilon_{t-1},\ldots,\epsilon_{t-m})].
\end{align*}
 The following summarizes the assumption on the dependence structure made throughout this paper.
\begin{assumption}
\label{as:depstruc}
$\{X_t\colon t\in\mathbb{Z}\}$ is a centered stationary functional time series in $\op^p_H$ such that
 \[
\sum_{j=0}^{\infty}\nu_{\hi,p}(X_j)< \infty \tageq \label{eq:depstruc2}
 \]
with $p=4$.
\end{assumption}
Observe that \autoref{as:depstruc} is weaker than $L^p_m$-dependence as introduced  in \cite{HorKok10} and than physical dependence as given in \cite{Wu05b}. 
 \begin{lemmas}[\textbf{Assumption \eqref{eq:depstruc2} versus $L^p_m$- and physical dependence}]\label{lpmim31} 
 Define the random functions $\tilde{X}^{(m)}_t=g(\epsilon_t, \ldots,  \epsilon_{t-m+1}, \epsilon^\prime_{t-m}, \epsilon^\prime_{t-m-1}, \ldots)$ and $\tilde{X}^{\prime}_t=g(\epsilon_t, \ldots, \epsilon_{1}, \epsilon^\prime_0, \epsilon_{-1}, \ldots)$, respectively. Then
\begin{enumerate}
\item[(i)]  $\sum_{t=0}^{\infty}\|X_t-\tilde{X}^{(t)}_t\|_{\hi,p}< \infty\,$ implies $\,\sum_{t=0}^{\infty}\nu_{\hi,p}(X_t)< \infty$.
\item[(ii)]  $\sum_{t=0}^{\infty}\|X_t-\tilde{X}^{\prime}_t\|_{\hi,p}< \infty\,$ implies  $\,\sum_{t=0}^{\infty}\nu_{\hi,p}(X_t)< \infty$.
\end{enumerate}
\end{lemmas}
\begin{proof}[Proof]
To prove (i), it suffices to observe that
\begin{align*}
\nu_{\hi,p}(X_t)
&\le  \bignorm{X_t -P_{t,1}\big(\E[X_t|\G_{t,\{0\}}]\big)}_{\hi,p} +\bignorm{P_{t,1}\big(\E[X_t|\G_{t,\{0\}}]\big)- \E[X_t|\G_{t,\{0\}}] }_{\hi,p}
\\& \le 2\|X_t -P_{t,1}(X_t)\|_{\hi,p}  
 = 2\|\E[X_t|\G_t]-\E[\tilde{X}^{(t)}_t|\G_t]\|_{\hi,p} 
\\&  \le  2\|X_t-\tilde{X}^{(t)}_t\|_{\hi,p}.
\end{align*}
Similarly, $\nu_{\hi,p}(X_t) 
\le  \|X_t -X^\prime_t\|_{\hi,p} + \|\E[X^\prime_t|\G_{t,\{0\}}]- \E[X_t|\G_{t,\{0\}}] \|_{\hi,p}  \le  2\|X_t-X^\prime_t\|_{\hi,p}$, from which (ii) follows. 
\end{proof}

 Additionally, note that by Jensen's inequality and the contraction property of the conditional expectation 
\[\|P_0(X_t)\|_{\hi,p}
= \| \E[ X_t-\E [X_t|\G_{t,\{0\}}] | \G_0 ]\|_{\hi,p} 
\le \nu_{\hi,p}(X_t).\] Hence, under condition \eqref{eq:depstruc2} we have $\sum_{j=0}^{\infty}\|P_0(X_j)\|_{\hi,p} < \infty$. \blu{It is moreover interesting to relate \autoref{as:depstruc} to summability of functional cumulant mixing conditions. In particular, the assumption of \textit{at least} summability of the fourth order cumulant operator in either $\snorm{\,\cdot}_2$ or $\snorm{\,\cdot}_1$ is often made in available literature \citep[see e.g.][]{pt13,Zhang16,PhPan2018}. These type of conditions are generally stronger than those considered in this paper and  generalizations of the coefficients in \eqref{eq:depstruc3} that can control for the interdependencies are required for such mixing conditions to hold. We postpone the discussion of these conditions to Section \ref{sec:sec5}.}

The assumption $\{X_t\colon t\in\mathbb{Z}\} \in \op^4_H$ ensures a finite second order structure of a random element of the form $X_t \otimes X_s$. Note that the latter can be viewed as a random element of $S_2(H)$, i.e., it is a measurable mapping from $(\Omega,\mathcal{A})$ into $(S_2(H), \mathcal{B})$ and thus $X_t \otimes X_s \in \op^2_{S_2(H)}$. Existence of a limit of the quadratic form in \eqref{eq:Q1} requires conditions on both the weight sequence as well as on the dependence structure. To elaborate on the latter, condition \eqref{eq:depstruc2} has two implications (Proposition \autoref{prop:sumCh} and Proposition \autoref{prop:propertiesDM}, resp.), which we shall make use of in order to derive distributional properties of the quadratic form. Denote the functional Discrete Fourier Transform (fDFT) of the stationary process $X$ by
\[\mathcal{D}^{(\lambda)}_T=\frac{1}{\sqrt{2\pi T}} \sum_{t=1}^{T} X_t e^{-\im \lambda t}. \tageq \label{eq:fDFT}\]
\blu{Provided the dependence structure of the process decays fast enough, the limiting variance of \eqref{eq:fDFT} is given by the spectral density operator in \eqref{eq:sdo}. As the next statement shows, this is the case for processes that satisfy Assumption \ref{as:depstruc} with $p=2$.}
\begin{prop}\label{prop:sumCh}
\blu{Suppose Assumption \ref{as:depstruc} with $p=2$ is satisfied. Then $\sum_{h \in \znum}\snorm{C_h}_2 < \infty $ and $
\F^{(\lambda)}$ exists as a non-negative definite Hermitian element of $S_2(H)$ for all $\lambda \in (-\pi,\pi]$. Furthermore,}
\begin{align*}
\blu{ \lim_{T \to \infty}\var(\mathcal{D}^{(\lambda)}_T )
= \F^{(\lambda)}.}
\end{align*}
\end{prop}
\begin{proof}[Proof of Proposition \autoref{prop:sumCh}]
We obtain by orthogonality of the projections, stationarity and Jensen's inequality
\begin{align*}
\sum_{h=0}^{\infty}\snorm{C_h}_{2} &
 \le \sum_{h = 0}^{\infty}\sum_{j=-\infty}^{0}\E \bigsnorm{\,P_0(X_{h-j})\otimes P_0(X_{-j})}_{2}
 \le\sum_{h =0}^{\infty}\sum_{j=-\infty}^{0}\sqrt{\E \|P_0(X_{h-j})\|_H^2 \E \|P_0(X_{-j})\|_H^2} 
  \\& \le \big (\sum_{j=0}^{\infty} \|P_0(X_{j})\|_{\hi,2}\big)^2 < (\sum_{j=0}^{\infty} \nu_{\hi,2}(X_j) \big)^2 <\infty
\end{align*}
and similarly for $h <0$. Hence, ${\F}_{\lambda} =\frac{1}{2\pi}\sum_{h \in \mathbb{Z}} C_h e^{-\im h \lambda}$ converges in norm $\snorm{\cdot}_2$ for all $\lambda \in (-\pi, \pi]$. It follows that $\F^{(\lambda)}$ is a non-negative definite,  Hermitian $S_2(H)$-valued density function over frequencies that satisfies $C_h=\int_{-\pi}^{\pi} \F^{(\lambda)}e^{\im h \lambda}d\lambda$ \citep[e.g.,][Thm 3.7]{vde18}. Moreover, from the dominated convergence theorem one obtains
\[\lim_{T \to \infty}\var(\mathcal{D}^{(\lambda)}_T )=\lim_{T \to \infty}\sum_{h \le T} (1-\frac{|h|}{T}) \E(X_h \otimes X_0) e^{-i \lambda h} = \F^{(\lambda)}   \quad \lambda \in (-\pi, \pi]. \tageq \label{eq:Iunb} \]
\end{proof}
Hence, \blu{$\F^{(\lambda)}$ in \eqref{eq:sdo}} exists as a limit of C{\'e}saro averages of $\{C_h e^{-\im h \lambda}\colon h \in \znum\}$ in $S_2(H)$. 
Without stronger assumptions, such as summability conditions, \blu{derivations} of several distributional properties of the quadratic form in \eqref{eq:Q1} do not appear obvious. Yet, \autoref{as:depstruc} allows to proceed via an approximating $S_2(H)$-valued random process. Underlying this approximation is the following process
\[
D^{(\lambda)}_{m,k,T} :=  \frac{1}{\sqrt{2\pi}} \sum_{t=0}^{T} P_{k}(X^{(m)}_{t+k})e^{-\im t \lambda}.
\tageq \label{eq:Dm}
\]
The second order structure of \eqref{eq:Dm} is closely related to that of $\mathcal{D}^{(\lambda)}_T$, but moreover has several useful properties that we shall make extensive use of.
\begin{prop} \label{prop:propertiesDM}
Under the conditions of \autoref{as:depstruc} with $p=4$, we have for all $\lambda \in (-\pi,\pi]$,
\begin{enumerate}[(i)]\itemsep-0.35ex
\item The process $D^{(\lambda)}_{m,k}:=D^{(\lambda)}_{m,k,\infty}$ \blu{exists} and forms \blu{an} $m$-dependent stationary martingale difference sequence with respect to the filtration $\{\G_k\}$ in $\op^4_H$. 
\item The process $\{D^{(\lambda)}_{\infty,0,T}\}_{T \ge 1}$ is Cauchy in $\op^{4}_H$
with limit 
\[D^{(\lambda)}_0:=\sum_{t=0}^{\infty} P_{0}(X_{t})e^{-\im t \lambda}\]
and the process $\{D^{(\lambda)}_{\infty,0,T} \otimes D^{(\lambda)}_{\infty,0,T}\}_{\{T \ge 1\}}$ is Cauchy in $\op^{2}_{S_2(H)}$ with limit $D^{(\lambda)}_{0} \otimes D^{(\lambda)}_{0}$.
\item
$\lim_{m \to \infty} \lim_{T \to \infty} \Tr\Big(\Pi_{ijkl}\big(\var(D^{(\lambda)}_{m,0,T}) \widetilde{\otimes} \var(D^{(\lambda)}_{m,0,T})\big)\Big) =\lim_{m \to \infty}\Tr\Big(\Pi_{ijkl} \big(\F_m^{(\lambda)} \widetilde{\otimes}\F_m^{(\lambda)}\big) \Big)$\\$
=\Tr\Big(\Pi_{ijkl} \big(\F^{(\lambda)} \widetilde{\otimes}\F^{(\lambda)}\big) \Big)< \infty$.
\end{enumerate}
\end{prop}
Here, $\Pi_{ijkl}$ denote the permutation operator on $\otimes_{i=1}^4 H$ that permutes the components of a tensor product of simple tensors according to the permutation $(1,2,3,4)\mapsto(i,j,k,l)$, that is, $\Pi_{ijkl}(x_1\otimes\cdots\otimes x_4)=(x_i\otimes\cdots\otimes x_l)$. The details of the proof can be found in \autoref{sec:Ap1}. For fixed $m$ and $p=2$, the first statement is almost immediate from the properties of the projection operators which form martingale difference sequences with respect to $\{\G_k\}$ and the fact that the process $\{X^{(m)}_{t}\}$ is $m$-dependent. For $p=4$, the proof of the above statements requires extensions of inequalities such as Burkholder's inequality for linear transforms of Hilbert-valued martingales; see \autoref{sec:Ap1}. The Cauchy property will be necessary to verify several aspects of the distributional properties, including verification of tightness on the function space of the quadratic form. 
Proposition \autoref{prop:propertiesDM}(iii) shows in particular that the iterated limit in $T$ and $m$, respectively, of a certain functional of the variance operators of the family of martingale processes  $\{D^{(\lambda)}_{m,0,T}\}_{T \ge 1}, m \ge 1$ converge to that of the corresponding functional of $\F^{(\lambda)}$, i.e. of the limiting variance operators of the fDFT, and that this functional is finite. 

Next, we require the following conditions on the sequence of weight operators. We assume that we have a representation $\Phi_{T,s,t}=(\phi_{T,s,t} \widetilde{\otimes} I_H)$, where $\phi_{T,s,t} \in S_{\infty}(H)$ such that $\Phi_{T,s,t}=\Phi_{T,s,t}^{\dagger}$. Observe that this is an operator in $S_\infty(S_2(H))$ with the property
\[
(\phi_{T,s,t} \widetilde{\otimes} I_H) (X_s \otimes X_t) = \phi_{T,s,t}(X_s) \otimes I_H(X_t) =  (\phi_{T,s,t} \widetilde{\otimes} I_H)^{\dagger} (X_s \otimes X_t) 
=I_H(X_s) \otimes \phi_{T,s,t}^{\dagger} (X_t). \tageq \label{eq:Phiprop}
\]
Note that the identity operator can be replaced with any arbitrary bounded linear operator $B_T \in S_{\infty}(H)$. 
Additionally, we require a few technical conditions to ensure that the weights are ``well-behaved'', i.e., the quadratic form exists as a well-defined random element of $S_2(H)$ for which no degenerate (non-Gaussian) limiting distributions can arise.
\begin{assumption}[Conditions on $\phi: \znum \times \znum \to S_{\infty}(H)$] \label{as:phi}
Let $T \in \mathbb{N}$ and $\lambda \in (-\pi,\pi]$. Let $A_{T,(\cdot)}:\znum \to S_{\infty}(H)$ be a continuous mapping such that $A_{T,t} \equiv A_{T,-t}, \forall t \in \znum$ and set $\phi^{(\lambda)}_{T,t}=A_{T,t}e^{\im \lambda t}$. Denote 
\[ \snorm{\Phi_T}^2_{ F}:=\sum_{t=1}^{T}\sum_{s=1}^{T} \snorm{\,\phi_{T,{t-s}}}_{\infty}^2 \text{ and } \varrho^2_T:=\sum_{t=1}^{T} \snorm{\,\phi_{T,t}}_{{\infty}}^2.\] We assume,
\begin{enumerate}[(i)]
\item $T\varrho_T^2=O(\snorm{\,\Phi_T}^2_{F});$
\item $\max_{1 \le t \le T}\snorm{\, \phi_{T,t}}_{\infty}^2=\max_{1 \le t \le T} \snorm{A_{T,t}}^2_{\infty} = o(\varrho_T^2);$
\item $\sum_{t=1}^{T} \snorm{A_{T,t}-A_{T,t-1} }_{\infty}^2 =o(\varrho_T^2);$
\item $\sum_{j=1}^{T-1} \sum_{s =1}^{j-1}\snorm{\sum_{t=j+1}^{T}  \phi_{T,s-t}\widetilde{\otimes}\phi_{T,j-t} }_{\infty}^2=o(\|\Phi_T\|^4_{F})$.
\end{enumerate}
\end{assumption}
Note that the first condition simply ensures a balance in order, i.e., the left-hand side is of the same order as the total sum of weights operator when the latter is viewed as a function-valued operator on $\znum \times \znum$. Together with the second, this means the norm of none of the individual weight contributions dominates the order of the variance. The third condition ensures a ``smooth'' contribution of each component ${\Phi}_{T,s,t}(X_s \otimes X_T)$ to the total mass of the quadratic form. The fourth condition is required to ensure that, as the overlap of the two bivariate operator-valued functions over $\znum \times \znum$ gets smaller, the contribution to the total mass must become negligible. Observe that for the examples mentioned in the introduction where $\phi^{(\lambda)}_{T,t}$ are scalar-valued, the norms $\snorm{\cdot}_{\infty}$ can be replaced by $|\cdot|$. Condition (iv) on the kernel then simply means a bandwidth parameter $b_T << 1$ must ensure a local smoothing occurs. As will become clear in the next section, it predictably excludes that the periodogram operator without smoothing can provide an asymptotically Gaussian consistent estimator of the spectral density operator. Many different weight functions used for the consistent estimation of $\F^{\lambda}$ will satisfy the above conditions, including the common choice of a bounded piecewise continuous lag window function with compact support, provided the bandwidth parameter ensures condition (iv) holds true (see \autoref{sec:sec4}). \\

In order to derive the properties of the quadratic form, a natural and common approach is to decompose $ \hat{\mathcal{Q}}_{T}$ into off-diagonal elements and diagonal elements as follows
 \begin{align}
 \hat{\mathcal{Q}}_{T}=\sum_{t=2}^T  \sum_{s=1}^{t-1} \Phi^{(\lambda)}_{T,s,t}( X_s \otimes X_t)+\Big(\sum_{t=2}^T  \sum_{s=1}^{t-1} \Phi^{(\lambda)}_{T,s,t}( X_s \otimes X_t) \Big)^{\dagger}+\sum_{1 \le t \le T}  \Phi_{T,t,t} (X_t \otimes X_t).  \tageq \label{eq:Qdecom}
\end{align}
The main ingredient to the proof is to use that the off-diagonal elements, after centering around their mean, can be approximated by the process 
 \begin{align}
\ldm{} = \sum_{t=2}^{T} \sum_{s=1}^{t-1} \Phi^{(\lambda)}_{T,t,s} \Big(\dm{t} \otimes \dm{s}\Big), \tageq \label{eq:ldm}
\end{align}
where the functionals $\dm{t}$ are defined via \eqref{eq:Dm} in Proposition \autoref{prop:propertiesDM}(i). The intuition is therefore similar in spirit to the strategy applied in the Euclidean setting \citep[see e.g.,][]{LiuWu10,WuSh07}. We emphasize that the aim of this paper is not the same nor can the weak convergence result in our paper be seen as a trivial extension of these works. We aim to derive consistency rates and joint distributional convergence of a set of operators where the quadratic form is very general, consisting of operator-valued weight operators of a Hilbertian-valued stochastic process. The derivation of the operator approximations and of the distributional properties, including the verification of tightness on the function space, are therefore far more involved. 
The convenient properties of \eqref{eq:ldm} are given in the next statement.
\begin{prop} \label{prop:Msq}
Let $\ldm{}$ as defined in \eqref{eq:ldm}. Under \autoref{as:depstruc} with $p=4$ and fixed $m$, the process 
\[
\Big\{\snorm{\Phi_T}^{-1}_{ F}{\ldm{}}\Big\}_{T \ge 1}
\]
is a martingale process in $\op^{2}_{S_2(H)}$ with respect to the filtration $\{\G_T\}$ for all fixed $\lambda \in (-\pi,\pi]$.
\end{prop}
\begin{proof}[Proof of Proposition \autoref{prop:Msq}]
It is immediate that $\ldm{}$ is adapted to the filtation $\G_T$. Secondly, from the properties of the operators $\{\Phi_{T,s,t}\}$ we can write
\begin{align*}
\ldm{} =  \sum_{t=2}^{T} \sum_{s=1}^{t-1} \Phi^{(\lambda)}_{T,s,t} \Big(\dm{t} \otimes \dm{s}\Big) =
\sum_{t=2}^{T}  \dm{t} \otimes \Big(\sum_{s=1}^{t-1} \phi^{(\lambda)}_{T,s,t} \dm{s}\Big).
\end{align*} 
From Proposition \autoref{prop:propertiesDM}(i), $\dm{t}$ forms a stationary martingale difference sequence in $\op^4_H$ with respect to $\{\G_t\}$. Hence, using orthogonality of the increments and \autoref{lem:Burkh} yields
\[
\E \snorm{\ldm{}}^2_2 \le  \sum_{t=2}^{T}  \E\bigsnorm{\dm{t} \otimes \Big(\sum_{s=1}^{t-1} \phi^{(\lambda)}_{T,s,t} \dm{s}\Big)}^2_2 \le   \|\dm{0}\|^2_{\hi,4}  \|\dm{0}\|^2_{\hi,4}   \sum_{t=2}^{T} \sum_{s=1}^{t-1}\snorm{\phi^{(\lambda)}_{T,s,t}}^2_\infty. 
\]
Noting that $\sum_{t=2}^{T}  \sum_{s=1}^{t-1}\snorm{\phi^{(\lambda)}_{T,s,t}}^2_\infty \approx 1/2 \snorm{\Phi_T}^{-1}_{ F}$, we obtain $\snorm{\Phi_T}^{-1}_{F} \E \snorm{\ldm{} }^2_2 < \infty$. Finally,  observe that 
\begin{align*}
 \E \Big[ \dm{t} \otimes \Big(\sum_{s=1}^{t-1} \phi^{(\lambda)}_{T,s,t} \dm{s}\Big) \big\vert \G_{t-1}\Big] = \E \Big[  \dm{t} \big\vert \G_{t-1}\Big]\otimes \Big(\sum_{s=1}^{t-1} \phi^{(\lambda)}_{T,s,t} \dm{s}\Big)=O_H\end{align*}
 where we used that $\sum_{s=1}^{t-1} \phi^{(\lambda)}_{T,s,t} \dm{s}$ is $\G_{t-1}$-measurable and that $\dm{t}$ is a $H$-valued martingale with respect to $\G_t$. The result now follows. 
\end{proof}
The following theorem states the distributional properties of the quadratic form. 
\begin{thm}[asymptotic normality of $\hat{\mathcal{Q}}^{\lambda}_T$] \label{thm:funcdist}
Let $\{X_t\}$ be a random sequence with paths in a separable Hilbert space $H$ for which assumption \eqref{as:depstruc} holds with $p=4$ and suppose that the sequence $\{{\Phi}^{\lambda}_{T}\}$ satisfies \autoref{as:phi}. Then the quadratic form in \eqref{eq:Q1} satisfies
\[
{(\snorm{\,\Phi_T}^2_{F})}^{-1/2}
\big(\hat{\mathcal{Q}}^{\lambda_j}_T-\E (\hat{\mathcal{Q}}_T^{\lambda_j})\big)_{j=1,\ldots,d} \Rightarrow \big( \breve{\mathcal{Q}}^{\lambda_j}\big)_{j=1,\ldots, d} 
\]
where, $\breve{\mathcal{Q}}^{\lambda_j}, j = 1, \ldots, d$ are jointly complex Gaussian elements of $S_2(H)$
\[
\begin{pmatrix}
\Re(\breve{\mathcal{Q}^{\lambda_j}}) \\ \Im(\breve{\mathcal{Q}}^{\lambda_j}) 
\end{pmatrix}_{j=1,\ldots, d}
\sim
\mathcal{N}_{(S_2(H))^d \times (S_2(H))^d}
\begin{pmatrix}
\begin{pmatrix}O_H \\ O_H
\end{pmatrix},
\frac{1}{2}
\begin{pmatrix}
\Re(\Gamma+ \Sigma) & \Im(-\Gamma+ \Sigma) \\
\Im(\Gamma+ \Sigma) & \Re(\Gamma- \Sigma)
\end{pmatrix}
\end{pmatrix}.
\]
The $(i,j)$-th element of the covariance operator is given by 
\[\Gamma_{i,j}=\eta(\lambda_i \pm \lambda_j) 4\pi^2\Big( \F^{(\lambda_j)}\widetilde{\otimes}\F^{(\lambda_j)} +\mathrm{1}_{\{0,\pi\}}\F^{(\lambda_j)} \widetilde{\otimes}_{\top} \F^{(\lambda_j)}\Big)\] and of the pseudocovariance operator by
\[\Sigma_{i,j} =\eta(\lambda_i \pm \lambda_j) 4\pi^2\Big(\mathrm{1}_{\{0,\pi\}}\F^{(\lambda_j)}\widetilde{\otimes}\F^{(\lambda_j)} +\F^{(\lambda_j)} \widetilde{\otimes}_{\top} \F^{(\lambda_j)}\Big),\]
and where 
$\eta(x) =1$ for $x = 2\pi z, z \in \znum$ and zero otherwise.
\end{thm}

In particular, for distinct frequencies $\lambda_1, \ldots, \lambda_d \in [0,\pi]$, $\Gamma$ and $\Sigma$ are $d \times d$ diagonal matrices with $S_1(H \otimes H)$-valued components and hence for such choice of frequencies, the components of $\big(\breve{\mathcal{Q}}^{\lambda_j}\big)_{j=1,\ldots, d}$ are asymptotically independent.
\begin{proof}[Proof of \autoref{thm:funcdist}]
We consider the sequence of processes $\{{\xi}^{\lambda}_{T} :T \in \mathbb{N}\}$ where
\[
{\xi}^{\lambda}_{T}:={(\snorm{\Phi^{\lambda}_T}^2_{F})}^{-1/2} (\hat{\mathcal{Q}}^{\lambda}_T-\E \hat{\mathcal{Q}}^{\lambda}_T). 
\]
Observe that $\{{\xi}^{\lambda}_{T} :T \in \mathbb{N}\}$ is a measurable stochastic processes with sample paths in the Hilbert space $S_2(H)$. We shall \blu{verify the two conditions in \autoref{lem:cremka}} to show weak convergence \blu{in} $S_2(H)$ \citep[see e.g.,][]{bil68}.
\begin{lemma}[weak convergence]\label{lem:cremka}
 Let $\{{\xi}_{T} :T \in \mathbb{N}\}$ be a stochastic process with sample paths in a separable Hilbert space. If the following two conditions are satisfied
\vspace{-5pt}
\begin{enumerate}\itemsep-0.3ex
\item[$\mathrm{i})$]The finite dimensional distributions of ${\xi}_{T}$ converge to those of ${\xi}$ a.e.;
\item[$\mathrm{ii})$] The family 
of laws $\mathcal{P}:=(\mathbb{P}_T)_{T \in \mathbb{N}}$ of $\{{\xi}_{T} :T \in \mathbb{N}\}$ is tight.
\end{enumerate}
then,
 ${\xi}_{T} \Rightarrow_T {\xi}$. 
\end{lemma}
First we derive that, for all $m \ge 1$, \blu{${\xi}^{\lambda}_{T,m} \Rightarrow_{T} {\xi}^{\lambda}_{m}$, where $ {\xi}^{\lambda}_{m}$} defines a zero-mean Gaussian element of $S_2(H)$ and where the double indexed process is given by
\[
{\xi}^{\lambda}_{T,m}:={(\snorm{\Phi_T}^2_{F})}^{-1/2} ({\ldm{}} +{\ldm{\dagger}})
\]
with $\ldm{}$ as in \eqref{eq:ldm}. By Proposition \autoref{prop:Msq}, for every fixed $m$, $\ldm{}$ is a martingale process in $\op^2_{S_2(H)}(\Omega, \mathcal{A},\mathbb{P})$ with respect to the filtration $\{\G_T\}$. Note the same holds for ${\ldm{\dagger}}$. Let $(\chi_l)_{l \ge 1}$ be an orthonormal basis of $H$. Then $(\chi_{l,\lpr})_{l,\lpr}:=(\chi_{l} \otimes \chi_\lpr)_{l,\lpr}$ defines an orthonormal basis of $H \otimes H$ and we shall denote
\[
{\xi}^{\lambda}_{T,m}(\chi) = \langle {\xi}^{\lambda}_{T,m}, \chi \rangle.
\]
for any $\chi \in H \otimes H$. The result below shows that, for a finite set of frequencies, the finite-dimensional distributions of ${\xi}^{\lambda}_{T,m}$ for fixed $m$ converge jointly to those of $\xi^{\lambda}_{m}$ as $T \to \infty$, where these are asymptotically independent at distinct frequencies.
\begin{thm}\label{thm:fddsmain} Suppose the conditions of \autoref{thm:funcdist} hold.  Then, for a finite set of distinct frequencies $\lambda_1,\ldots, \lambda_d \in [0,\pi]$, for all $\forall m \ge 1$ and any $\chi_{l_j l_j^{\prime}} \in H \otimes H$, we have
\[
\{{\xi}^{\lambda_j}_{T,m}(\chi_{l_j l_j^{\prime}})\}_{j=1,\ldots,d} \Rightarrow_T \{{\xi}^{\lambda_j}_{m}(\chi_{l_j l_j^{\prime}})\}_{j=1,\ldots, d} \sim \mathcal{N}_{\cnum^d}\Big(0,\mathrm{diag}\big(\Gamma^{\lambda_j}_m(\chi_{l_j l_j^{\prime}})\big),\mathrm{diag}\big(\Sigma^{\lambda_j}_m(\chi_{l_j l_j^{\prime}})\big)\Big), \]
 where 
\[\Gamma^{\lambda_j} _m(\chi_{l_j l^{\prime}_j})= { 4\pi^2} \Big( \F^{(\lambda_j)}_m(\chi_{l_jl_j}) \overline{\F^{(\lambda_j)}_m(\chi_{l^{\prime}_j l^{\prime}_j})} +\mathrm{1}_{\{0,\pi\}}\big(\F^{(\lambda_j)}_m(\chi_{l_j l^{\prime}_j})  \F^{(\lambda_j)}_m(\chi_{l_j l^{\prime}_j})\big) \Big) \tageq \label{eq:Gamfdds}\] 
and 
\[\Sigma^{\lambda_j}_m(\chi_{l_j l^\prime_j}) = { 4\pi^2} \Big(\mathrm{1}_{\{0,\pi\}}\big(\F^{(\lambda_j)}_m(\chi_{l^{\prime}_j l^{\prime}_j})\F^{(\lambda_j)}_m(\chi_{l_jl_j}) \big)+ \F^{(\lambda_j)}_m(\chi_{l_j l^{\prime}_j}) \F^{(-\lambda_j)}_m(\chi_{l_j l^{\prime}_j})  \Big),\]
where $\F^{(\lambda_j)}_m(\chi_{l_j l^{\prime}_j})  = \big(\E(\dmi{0}{j} \otimes \dmi{0}{j})\big)(\chi_{l_j l^{\prime}_j})$.
\end{thm}
The proof is tedious and relegated to \autoref{sec:Ap_fdds}. Next, we show that $\forall m \ge 1, \, \{{\xi}^{\lambda}_{T,m}, T\ge 1\}$ is tight. In order to verify tightness we shall use the following result, which is a particular case of \citep[][Theorem 3]{Suquet1996} who considers tightness criteria for more general Schauder decomposable Banach spaces.
\begin{lemma}[tightness on a separable Hilbert space]\label{lem:tight}
Let $(\chi_{l \lpr})$ be an orthonormal basis of $H \otimes H$. A family of probability measures $\mathcal{P}:=(\mathbb{P}_T)_{T \in \mathbb{N}}$ 
on $S_2(H)$ is tight if and only if
\begin{enumerate}\itemsep-0.3ex
\item[$\mathrm{i})$] $\forall k \ge 1: \quad \lim_{h \to \infty} \sup_{T} \mathbb{P}_T\Big(\Big\{x \in S_2(H): \sum_{l ,\lpr< k}|\inprod{x}{\chi_{l \lpr}}|^2 > h \big\}\Big)=0$;
\item[$\mathrm{ii})$] $\forall \epsilon>0: \quad \lim_{k \to \infty} \sup_{T} \mathbb{P}_T\Big(\big\{x \in S_2(H): \sum_{l, \lpr: l+\lpr > k}|\inprod{x}{\chi_{l \lpr}}|^2 > \epsilon \big\}\Big)=0$.
\end{enumerate}
\end{lemma}
 In order to verify the first condition, note that, since $k$ is fixed
\begin{align*}
\lim_{h \to \infty}& \sup_{T} \mathbb{P}\Big(\sum_{l,\lpr< k}|\inprod{{\xi}^{\lambda}_{T,m}}{\chi_{l \lpr}}|^2 > h \Big)
\le \sum_{l, \lpr< k} \lim_{h \to \infty} \sup_{T} \mathbb{P}\Big(|\inprod{{\xi}^{\lambda}_{T,m}}{\chi_{l \lpr}}|^2 > h\Big),
\end{align*}
and hence the first condition is implied by
\[ \forall l, \lpr \ge 1: \quad \lim_{h \to \infty} \sup_{T}  \mathbb{P}\Big(|\inprod{{\xi}^{\lambda}_{T,m}}{\chi_{l \lpr}}|^2 > h\Big)=0, \tageq \label{eq:first}\]
for which we moreover have
\[
  \mathbb{P}\Big(|\inprod{{\xi}^{\lambda}_{T,m}}{\chi_{l \lpr}}|^2 > h\Big) 
 \le  \mathbb{P}\Big( \Re(\inprod{{\xi}^{\lambda}_{T,m}}{\chi_{l \lpr}})^2  > h/2\Big)+\mathbb{P}\Big( \Im(\inprod{{\xi}^{\lambda}_{T,m}}{\chi_{l \lpr}})^2  > h/2\Big).
 \]
Since the real and imaginary part of the random variables $\inprod{{\xi}^{\lambda}_{T,m}}{\chi_{l \lpr}}$ converge to real-valued random variables by \autoref{thm:fddsmain}, the corresponding sequence of probability measures is tight on $(\mathbb{R},\mathcal{B})$. \eqref{eq:first} therefore follows from  the continuous mapping theorem. In order to verify the second condition of \autoref{lem:tight}, note that by Markov's inequality it suffices to prove that 
\[
\lim_{k \to \infty}\sup_T \sum_{\substack{l, \lpr: l+\lpr \ge k}} \E |{\xi}^{\lambda}_{m,T}(\psi_{l \lpr})|^2 = 0. \tageq \label{eq:targettail}
\]

Firstly, observe that $\E |{\xi}^{\lambda}_{m,T}(\psi_{l \lpr})|^2 \ge 0$ and $\inprod{\var({\xi}^{\lambda}_{m})(\psi_{l \lpr})}{(\psi_{l \lpr})}\ge0$. Note then from \eqref{eq:Gamfdds} that for any $\chi_{l \lpr} \in H \otimes H$
\[\lim_{T\to \infty} \E|{\xi}^{\lambda}_{m,T}(\chi_{l \lpr})|^2 = \Gamma^{\lambda} _m(\chi_{l \lpr})=\var({\xi}^{\lambda}_{m}(\chi_{l \lpr}))< \infty. \tageq \label{eq:pointwise}\]
Together with Parseval's identity the monotone convergence and by definition of the (transpose) Kronecker tensor product, \autoref{thm:fddsmain} implies
\begin{align*}
\limsup_{T\to \infty}\E\big[\snorm{{\xi}^{\lambda}_{T,m}}_2^2\big]
&\le \sum_{l,l^\prime=1}^\infty \lim_{T\to \infty} \E\big[\big|{\xi}^{\lambda}_{T,m}(\psi_{ll^\prime})\big|^2\big]
\\&= \sum_{l,l^\prime=1}^\infty { 4\pi^2}\Big( \F^{(\lambda)}_m(\psi_{ll}) \overline{\F^{(\lambda)}_m(\psi_{\lpr \lpr})} +\mathrm{1}_{\{0,\pi\}}\big(\F^{(\lambda)}_m(\psi_{ll^\prime})  \F^{(\lambda)}_m(\psi_{ll^\prime})\big) \Big)
\\&  =  { 4\pi^2} \Tr\Big(\F^{(\lambda)}_m \widetilde{\otimes} \F^{(\lambda)}_m)\Big) +\Tr\Big(\mathrm{1}_{\{0,\pi\}}\big( \F^{(\lambda)}_m\widetilde{\otimes}_{\top} {\F^{(\lambda)}_m} \big) \Big)= \E\snorm{{\xi}^{\lambda}_{m}}_2^2. \tageq \label{eq:sumwise}
\end{align*}
From Proposition \autoref{prop:propertiesDM}(iii.) we find immediately that
\[\E\snorm{{\xi}^{\lambda}_{m}}_2^2=\Tr(\var({\xi}^{\lambda}_{m})) < \infty.
\]
Consequently, we can choose an $\epsilon>0$ such that for all $k \ge k_0$
\[
|\Tr(\var({\xi}^{\lambda}_{m})) -  \sum_{l + \lpr \le k_0} \inprod{ (\var({\xi}^{\lambda}_{m}))(\psi_{l \lpr})}{(\psi_{l \lpr})}| < \epsilon.
\]
From the pointwise convergence \eqref{eq:pointwise} and from the sequence convergence in \eqref{eq:sumwise}, we obtain
\begin{align*}
\lim_{T \to \infty} \sum_{l,\lpr: l+\lpr \ge k_0}^{\infty} \E|{\xi}^{\lambda}_{m,T}(\psi_{l \lpr})|^2&= 
\lim_{T \to \infty}\Big( \sum_{l,\lpr =1}^{\infty} \E|{\xi}^{\lambda}_{m,T}(\psi_{l \lpr})|^2 -  \sum_{l + \lpr < k_0} \E|{\xi}^{\lambda}_{m,T}(\psi_{l \lpr})|^2\Big)
\\& \le \Tr(\var({\xi}^{\lambda}_{m})) -  \sum_{l + \lpr < k_0} \inprod{ (\var({\xi}^{\lambda}_{m}))(\psi_{l \lpr})}{(\psi_{l \lpr})}.
\end{align*}
In other words, there must exist a $T_0$ such that for all $T \ge T_0$ and $k \ge k_0$
 \[
|\E\snorm{{\xi}^{\lambda}_{m,T}}^2_2 -  \sum_{l + \lpr < k} \E|{\xi}^{\lambda}_{m,T}(\psi_{l \lpr})|^2| < \epsilon.
\]
Moreover, we can choose a $\tilde{k} \ge k_0$ such that for all $1\le  T <T_0$,
 \[
\E\snorm{{\xi}^{\lambda}_{m,T}}^2_2 -  \sum_{l + \lpr < \tilde{k}} \E|{\xi}^{\lambda}_{m,T}(\psi_{l \lpr})|_2^2 < \epsilon.
\]
\eqref{eq:targettail} now follows by taking $\max(k,\tilde{k})$. Therefore, we have established both conditions of \autoref{lem:cremka}, and thus ${\xi}^{\lambda}_{T,m} \Rightarrow_T {\xi}^{\lambda}_{m}$. Next, we will show that ${\xi}^{\lambda}_{m} \Rightarrow_m \breve{\mathcal{Q}}^{\lambda}$ where $\breve{\mathcal{Q}}^{\lambda}$ denotes the limiting process given in \autoref{thm:funcdist}. We again verify the conditions of  \autoref{lem:cremka}. From \autoref{thm:fddsmain} and Proposition \autoref{prop:propertiesDM}(iii.), we find
\begin{align*}
\lim_{m \to \infty}\E\snorm{\xi^{\lambda}_m}_2^2
&=
\lim_{m \to \infty} { 4\pi^2} \Tr\Big(\F^{(\lambda)}_m \widetilde{\otimes} \F^{(\lambda)}_m +\mathrm{1}_{\{0,\pi\}}\big( \F^{(\lambda)}_m\widetilde{\otimes}_{\top} {\F^{(\lambda)}_m} \big) \Big) 
\\&= { 4\pi^2} \Tr\Big(\F^{(\lambda)} \widetilde{\otimes} \F^{(\lambda)} +\mathrm{1}_{\{0,\pi\}}\big( \F^{(\lambda)}\widetilde{\otimes}_{\top} {\F^{(\lambda)}} \big) \Big) =\E\snorm{\breve{\mathcal{Q}}^{\lambda}}_2^2< \infty. \tageq \label{eq:varlimm}
\end{align*}
Recall then that \autoref{thm:fddsmain} shows that for fixed $m$ and for any $\chi \in H \otimes H$, $\xi^{\lambda}_m(\chi)$  is a zero mean complex-valued Gaussian random variable. Hence $\xi^{\lambda}_m(\chi) \Rightarrow_m \breve{\mathcal{Q}}^{\lambda}(\chi)$ if we can show that the covariance structure satisfies
\begin{align*}
\lim_{m \to \infty}\Gamma_m(\chi_{l\lpr}) &= \Gamma(\chi_{l\lpr}) \tageq \label{eq:varlimfddsm}
\\
\lim_{m \to \infty}\Sigma_m(\chi_{l\lpr}) &= \Sigma(\chi_{l\lpr}),
\end{align*}
where $\Gamma $ and $\Sigma$ are the covariance and pseudocovariance operator given in \autoref{thm:funcdist}. This however follows immediately from \eqref{eq:varlimm}.  Hence, $\xi^{\lambda}_m(\chi_{l\lpr}) \Rightarrow_m \breve{\mathcal{Q}}^{\lambda}(\chi_{l\lpr})$ showing the finite-dimensional distributions converge. Similar to \eqref{eq:first} this implies that \[ \forall l, \lpr \ge 1: \quad \lim_{h \to \infty} \sup_{m}  \mathbb{P}\Big(|\inprod{{\xi}^{\lambda}_{m}}{\chi_{l \lpr}}|^2 > h\Big)=0. \tageq \label{eq:first2}\] 
Hence, condition \autoref{lem:tight}(i) is satisfied. The tightness condition \autoref{lem:tight}(ii) is satisfied if 
\[
\lim_{k \to \infty}\sup_m \sum_{l+ \lpr \ge k} \E |{\xi}^{\lambda}_{m}(\psi_{l \lpr})|^2 = 0. \tageq \label{eq:targettailm}
\]
From the pointwise convergence \eqref{eq:varlimfddsm} and the convergence of \eqref{eq:varlimm} as $m \to \infty$, this now however follows similarly to the proof of \eqref{eq:targettailm}. Altogether, this establishes ${\xi}^{\lambda}_{T,m} \Rightarrow_T {\xi}^{\lambda}_m \Rightarrow_m \breve{\mathcal{Q}}^{\lambda}$.
Finally, it remains to show  ${\xi}^{\lambda}_{T} \Rightarrow_T \breve{\mathcal{Q}}^{\lambda}$, for which we make use of the next lemma.
\begin{lemma}\label{lem:bilaprox}
Under the conditions of \autoref{thm:funcdist}
\begin{align}
\lim_{m \to \infty} \limsup_{T \to \infty } \frac{1}{\snorm{\Phi_T}_{F}} \bignorm{\hat{\mathcal{Q}}^{\lambda}_{T}-\E\hat{\mathcal{Q}}^{\lambda}_T -\ldm{}-\ldm{\dagger} }_{S_2,2} = 0. \label{eq:bilcheck}
\end{align}
\end{lemma}
The proof can be found in \autoref{sec:A_opAprox}. Since $S_2(H)$ is a complete metric space, let $F$ be a closed set of $S_2(H)$ and fix $\epsilon >0$. Then
 \begin{align*}
 \mathbb{P}({\xi}^{\lambda}_{T} \in F) & \le \mathbb{P}( \snorm{{\xi}^{\lambda}_{T,m}-{\xi}^{\lambda}_{T} }_2 \ge \epsilon)  + \mathbb{P}\big( {\xi}^{\lambda}_{T,m}\in \{x : \snorm{x- y}_2 \le \epsilon, y \in F\}\big) 
\end{align*}
and since by the weak convergence of ${\xi}^{\lambda}_{T,m} \Rightarrow_T {\xi}^{\lambda}_m \Rightarrow_m \breve{\mathcal{Q}}^{\lambda}$, we have 
\begin{align*}
\lim_{m \to \infty}  \limsup_{T \to \infty} \mathbb{P}\big( {\xi}^{\lambda}_{T,m}\in \{x : \snorm{x- y}_2 \le \epsilon, y \in F\}\big) \le  \mathbb{P}\big( \breve{\mathcal{Q}}^{\lambda}\in \{x : \snorm{x- y}_2 \le \epsilon, y \in F\}\big).
\end{align*}
Using then \autoref{lem:bilaprox}, Markov's inequality yields
 \begin{align*}
\limsup_{T \to \infty}\mathbb{P}({\xi}^{\lambda}_{T} \in F) \le \mathbb{P}\big( \breve{\mathcal{Q}}^{\lambda}\in \{x : \snorm{x- y}_2 \le \epsilon, y \in F\}\big),  
\end{align*}
so that taking $\epsilon \to 0$, completes the proof.
\end{proof}

\section{Estimation of the spectral density operator} \label{sec:sec4}

In this section, we focus on the application of the above theorem to estimate the spectral density operator
\begin{align*}
\F^{(\lambda)} =\frac{1}{2\pi}\sum_{h \in \znum} C_h e^{-\im \lambda h}.
\end{align*}
Proofs of the statements in this section are postponed to \autoref{sec:ApF}. It is well-known that under various conditions \citep[see e.g.,][]{{HorHalL15},pt13} an asymptotically unbiased estimator is given by the periodogram operator
\[\mathcal{I}^{\lambda}_T : = \mathcal{D}_{T}^{\lambda}\otimes \mathcal{D}_{T}^{\lambda} \]
where $ \mathcal{D}_{T}^{\lambda}$ are the fDFT of $X$ given in \autoref{sec:sec2}. Note that by construction, this operator is hermitian, non-negative definite and $\lambda \mapsto \mathcal{I}^{\lambda}_T$ is $2\pi$-periodic. From \eqref{eq:Iunb}, we can immediately conclude that, under the stated conditions, the periodogram operator is indeed an asymptotically unbiased estimator of $\F^{(\lambda)}$. It can however never be consistent because it is based upon one frequency observation. A consistent estimator of the spectral density operator can be obtained via smoothing the operator-valued function $\lambda \mapsto \mathcal{I}^{\lambda}_T$ over neighboring frequency ordinates, i.e., via convolving the periodogram operator with a window function $K$. 
 For example, it is very common to consider an estimator of the form
\begin{align}
\hat{\F}^{\omega}& =\frac{1}{b_T}\int^{\infty}_{-\infty} K\big(\frac{\omega-\lambda}{b_T}\big) \mathcal{D}_{T}^{\lambda}\otimes \mathcal{D}_{T}^{\lambda} d\lambda,  \label{eq:smoFQ}
\end{align}
where $K: \rnum \to \mathbb{R}_{+}$ is assumed to be an even, non-negative weight function that is integrable. Under \autoref{as:depstruc} with $p=4$, it is immediate from an application of the Cauchy-Schwarz inequality and  \autoref{lem:lineq}(i) that $\sup_{\lambda}\norm{\mathcal{I}^{\lambda}_T}_{S_2,2}=O(1)$ uniformly in $T$. By Holder's inequality, \eqref{eq:smoFQ} therefore exists as an element of $\norm{\cdot}_{S_2,2}$. In order to exploit the results from the previous section, we however require the estimator can be formulated in terms of a quadratic form. As remarked in the introduction, we consider
\[\hat{\F}^{\omega}= \frac{1}{2\pi T}\sum_{s,t=1}^{T} \big( (X_s-\mu) \otimes (X_t-\mu) \big)  w(b_T(t-s)) e^{\im \omega (t-s)}.  \tageq \label{eq:relFQ}\]
Note that $\hat{\F}^{\omega} =\frac{1}{2\pi  T} \hat{\mathcal{Q}}^{\omega}_{T}$ with ${\Phi}_{T,t,s} = \phi_{T,(t-s)}^{\omega} I_{H \otimes H}=w(b_T(t-s)) e^{\im \omega (t-s)}I_{H \otimes H}$ thus yields the representation in terms of the quadratic form introduced in the previous section. Provided $w(\cdot)$ and $K(\cdot)$ form Fourier pairs, there is a clear connection between \eqref{eq:smoFQ} and \eqref{eq:relFQ}. Namely, a change of variables gives
\begin{align*}
\hat{\F}^{\omega}& =\frac{1}{b_T}\int^{\infty}_{-\infty} K\big(\frac{\omega-\lambda}{b_T}\big) \mathcal{D}_{T}^{\lambda}\otimes \mathcal{D}_{T}^{\lambda} d\lambda  = \int^{\infty}_{-\infty} K(x) \mathcal{I}_{T}^{\omega + x b_T}d x
\\& = \int^{\infty}_{-\infty} K(x) \frac{1}{2\pi T}\sum_{s,t=1}^{T} e^{-\im (\omega +x b_T) (s-t)} (X_s \otimes X_t)d x
\\& =\frac{1}{2\pi T}\sum_{s,t=1}^{T} (X_s \otimes X_t) e^{-\im \omega(s-t)}\int^{\infty}_{-\infty} K(x) e^{\im x b_T (t-s)}d x
\\& =\frac{1}{2\pi T}\sum_{s,t=1}^{T} (X_s \otimes X_t)  w(b_T(t-s)) e^{\im \omega (t-s)}, 
\end{align*}
where the equality is with respect to $\norm{\cdot}_{S_2,2}$. 
In order to verify consistency and asymptotic normality, we require the following assumptions on the weight function $w(\cdot)$ in \eqref{eq:relFQ}.
\begin{assumption}\label{as:Weights}
Let $w$ be an even, bounded function on $\mathbb{R}$ with $\lim_{x \to 0} w(x) =1$ that is continuous except at a finite number of points. 
 Furthermore, suppose that $\lim_{b \to 0}b \sum_{h \in \znum} w^2(b h) = \kappa$ where $\kappa:=\int^{\infty}_{-\infty} w^2(x) dx < \infty$ such that $\sup_{0 \le b \le 1} b \sum_{|h| \ge M/b} w^2(b h) \to 0$ as $M \to \infty$.
\end{assumption}
Observe that these are rather mild conditions for window functions and includes a wide range of common choices \citep[see e.g.][]{BrockDav91}. Under these conditions we can obtain consistency in mean square of the spectral density operator.
\begin{thm}\label{thm:consF}
Suppose \autoref{as:depstruc} with $p\ge 2$ and \autoref{as:Weights} are satisfied. Then, 
\begin{enumerate}[(i)]
\item $\sup_{\lambda \in [0,\pi]}\norm{\hat{\F}_T^{\lambda} -{\F}^{\lambda}}^2_{S_2,p} \to 0$ if  $b_T \to 0$ as $T \to \infty$ such that $b_T T \to \infty$.
\item If, in addition, $\sum_{h \in \znum} h \|P_0(X_h)\|_{\hi,2} < \infty$ and $w(x) - 1=O(x)$ as $x\to 0$, then
\begin{align*}
\norm{\hat{\F}_T^{\lambda} -{\F}^{\lambda}}^2_{S_2,p} = O(\frac{1}{b_T T}+b^2_T)
\end{align*}
uniformly in $\lambda \in [0,\pi]$.
\end{enumerate} 
\end{thm}
Note that \autoref{thm:consF} does not rely on a martingale approximation to exist but relies on the ergodicity properties of the underlying process. Without loss of generality, we can restrict to the interval $[0,\pi]$ since the mappings $\lambda \mapsto \hat{\F}_T^{\lambda}$ and $\lambda \mapsto {\F}_T^{\lambda}$ are even and $2\pi$-periodic. Under \autoref{as:depstruc} with $p=4$ and \autoref{as:Weights} we in fact obtain that $\sup_{\lambda \in [0,\pi]} \E \snorm{\hat{\F}_T^{\lambda} -\E{\F}^{\lambda}}^2_2 =O(\frac{1}{b_T T})$. It is however often of importance to specify the rate of consistency and hence to control the order of the bias in norm. As given in the second part of the statement, this requires mild additional conditions on the smoothness of the process as well as a smoothness condition of the weight function around $0$.

\begin{Remark}[\textbf{If the function $\mu$ is unknown}]\label{rem:mean}
\textup{In case the mean function $\mu$ is unknown, we can instead consider the estimator
\begin{align}
\hat{\ddot{\F}}^{(\lambda)}& =\frac{1}{2\pi T}\sum_{s,t=1}^{T}\big( (X_s-\hat{\mu}) \otimes (X_t-\hat{\mu}) \big)  w(b_T(t-s)) e^{\im \lambda (t-s)}, 
\end{align}
where $\hat{\mu} = \frac{1}{T}\sum_{j=1}^{T}X_T$ denotes the sample mean function and which defines a random element of $H$. We obtain the following error bound with the estimator in \eqref{eq:relFQ}, which shows the results in this section are not affected by centering the data using the sample mean.
\begin{lemma}\label{lem:sampmean}
Suppose \autoref{as:depstruc} with $p=4$ is satisfied and \autoref{as:Weights} holds. Then
\[ \sup_{\lambda \in [0,\pi]}\E\snorm{\hat{\ddot{\F}}^{(\lambda)}- \hat{{\F}}^{(\lambda)}}^2_{2} =O( (b_T T)^{-2}).\]
More generally, if $X \in \op^{2p}_H$, then $\norm{\hat{\ddot{\F}}^{(\lambda)}- \hat{{\F}}^{(\lambda)}}_{S_2,p} =O( (b_T T)^{-1}) $, $p \ge 1$. 
\end{lemma}}
\end{Remark}

\blu{A consequence of \autoref{thm:consF} is that dynamic functional principal component analysis and hence an optimal dimension reduction of functional time series can be derived under weaker conditions than currently available \citep{pt13b,HorHalL15}. 
A central role for such techniques is played by the empirical eigenelements of an estimator $\hat{\F}_T^{\lambda}$ of $\F^{\lambda}$.  
More specifically, since $\F^{\lambda}$ and $\hat{\F}_T^{\lambda}$ are compact self-adjoint non-negative definite operators, they admit representations of the form
\[
{\F}_T^{\lambda} = \sum_{j=1}^{\infty}{\beta}_j^{\lambda}\, \Pi^{\lambda}_j
 \quad \text{ and } \quad \hat{\F}_T^{\lambda} = \sum_{j=1}^{\infty} \hat{\beta}_j^{\lambda}\, \hat{\Pi}^{\lambda}_j, \]
where the sequence of eigenvalues $\{{\beta}_j^{\lambda} \}_{j \ge 1}$ is a non-increasing sequence of positive real numbers tending to zero and where  $\Pi^{\lambda}_j ={\phi}_j^{\lambda}\otimes {\phi}_j^{\lambda}$ and $\hat{\Pi}^{\lambda}_j = \hat{\phi}_j^{\lambda}\otimes \hat{\phi}_j^{\lambda}$ are the $j$-th eigenprojectors of ${\F}^{\lambda}$ and $\hat{\F}_T^{\lambda}$, respectively.  
Using the result of \autoref{thm:consF}{(i)}, we obtain consistency of the empirical eigenprojectors and eigenvalues for their population counterparts. 
\begin{prop}\label{prop:eigcon}
Suppose \autoref{as:depstruc} with $p=4$ and \autoref{as:Weights} are satisfied. Then for each $j \ge 1$ and each $\lambda \in [0,\pi]$, and for bandwidths satisfying $b_T \to 0$ as $T \to \infty$ such that $b_T T \to \infty$,
\begin{enumerate}\itemsep-0.3ex
\item[(i)] $| {\beta}_j^{\lambda}- \hat{\beta}_j^{\lambda}|\overset{p}{\to} 0$;
\item[(ii)] If in addition $\inf_{l\ne j}\big\vert \beta^{\lambda}_j-\beta^{\lambda}_l\big\vert>0$, we have $\snorm{\,\hat{\Pi}^{\lambda}_j-{\Pi}^{\lambda}_j \,}_2 \overset{p}{\to}  0$,\\
\end{enumerate}
\vspace{-20pt}
\end{prop}
The proof of $(i)$ and $(ii)$ follow (almost) immediately from results of \cite{GGK90} and \cite{MM03}, respectively. Details are given in the Appendix. Note that the consistency is given for the projectors since the empirical eigenfunctions can only be identified up to rotation on the unit circle, that is, we obtain a version $c(\lambda)\hat{\phi}_j^{\lambda}$, for some $c(\lambda) \in \mathbb{C}$ with $|c(\lambda)|^2=1$. In contrast, the projectors are invariant to the rotation. 
}

The next result is the joint distributional convergence of a set of estimators at distinct frequencies to uncorrelated Gaussian elements of $S_2(H)$.
\begin{thm}\label{thm:ANF}
Suppose \autoref{as:depstruc} with $p=4$ and \autoref{as:Weights} are satisfied. Let $\lambda_1,\ldots, \lambda_d \in [0,\pi]$ be distinct. Then, for \blu{bandwidths satisfying} $b_T \to 0$ such that $b_T T \to \infty$ as $T \to \infty$ 
\[
\sqrt{b_T T}\big(\hat{\F}^{\lambda_j} -\E\hat{\F}^{\lambda_j}\big)_{j=1,\ldots, d} \Rightarrow_T \big(\mathfrak{F}^{\lambda_j}\big)_{j=1,\ldots,d}
\]
where $\mathfrak{F}^{\lambda_j}, j=1,\ldots, d$ are zero-mean jointly independent complex Gaussian elements of $S_2(H)$, with covariance operator
 \[\cov\big(\mathfrak{F}^{\lambda_j},\mathfrak{F}^{\lambda_j}\big)=   \kappa^2 \Big( \F^{(\lambda_j)}\widetilde{\otimes}\F^{(\lambda_j)} +\mathrm{1}_{\{0,\pi\}}\F^{(\lambda_j)} \widetilde{\otimes}_{\top} \F^{(\lambda_j)}\Big)\] and with
 pseudocovariance operator 
\[\cov\big(\mathfrak{F}^{\lambda_j},\overline{\mathfrak{F}^{\lambda_j}}\big) =\kappa^2 \Big(\mathrm{1}_{\{0,\pi\}}\F^{(\lambda_j)}\widetilde{\otimes}\F^{(\lambda_j)} +\F^{(\lambda_j)} \widetilde{\otimes}_{\top} \F^{(-\lambda_j)} \Big).\]
If the conditions of \autoref{thm:consF}$\mathrm{(ii)}$ are also satisfied, then 
\[\sqrt{b_T T}\big(\hat{\F}^{\lambda_j} -{\F}^{\lambda_j}\big)_{j=1,\ldots, d} \Rightarrow_T \big(\mathfrak{F}^{\lambda_j}\big)_{j=1,\ldots,d}~\]
for bandwidths satisfying $b_T =o(T^{-1/3})$.
\end{thm}
Observe that if $\lambda_j \in \{0, \pi\}$, then $\mathfrak{F}^{\lambda_j}$ is real Gaussian. Finally, we obtain the following corollary on the distributional properties of the estimator of the long run covariance operator, \blu{which can be seen to improve upon the results in \cite{BerHorRi17} and \cite{pt13}}.
\begin{Corollary} \label{cor:lonrun}
Under the conditions of \autoref{thm:ANF}, 
\[
\sqrt{b_T T}2\pi \big(\hat{\F}^{(0)} -{\F}^{(0)}\big) \Rightarrow_T \mathcal{N}_{S_2(H)}(0, 4 \pi^2\Gamma^{(0)}) 
\]
where $\Gamma^{(0)} =  \kappa^2 \Big( \F^{(0)}\widetilde{\otimes}\F^{(0)} +\F^{(0)} \widetilde{\otimes}_{\top} \F^{(0)}\Big)$.
\end{Corollary} 

\section{\blu{Relation to functional cumulant mixing conditions}} \label{sec:sec5}

Unlike the majority of existing literature \citep[see e.g.,][]{pt13,Zhang16,PhPan2018,vde16}, the results of this paper do not require higher order functional cumulant mixing conditions of the form
\[
\sum_{t_1,\ldots, t_{n-1} \in \znum}\snorm{\, \mathrm{cum}(X_{t_1},\ldots, X_{t_{n-1}},X_{0})}_{2} < \infty. \tageq \label{eq:cumk1}
\]
Here, the $n$-th order joint cumulant tensor of $X_{t_1}, \ldots, X_{t_n} \in \mathcal{L}_H^n$, which can be viewed as an element of $ S_2({\otimes^{\lfloor \frac{n+1}{2} \rfloor}H, \otimes^{\lfloor \frac{n}{2}\rfloor} H})$, is defined by
\[
\mathrm{cum}(X_{t_1},\ldots, X_{t_n}) = \sum_{\boldsymbol{v}=(v_1,\ldots v_{\rho})} (|\rho|-1)! (-1)^{|\rho|-1} \Pi_{\boldsymbol{v}}\Big( \otimes_{r=1}^{\rho}\E \big[\otimes_{i \in v_r} X_{t_i}\big]\Big), \tageq \label{eq:cumtens}
\]
where the summation extends over all unordered partitions $\boldsymbol{v}$ of $\{1,\ldots,n\}$ and $\Pi_{\boldsymbol{v}}$ denotes the permutation operator that maps the components of the tensor back into the original order, that is, $\Pi_{\boldsymbol{v}}( \otimes^\rho_{r=1} \otimes_{{t_i} \in v_r} X_{t_i}) = X_{t_1} \otimes \cdots \otimes X_{t_n}$ \citep[see e.g.,][]{vde16}. The cumulant tensor in \eqref{eq:cumtens}  measures the joint statistical dependence of order $n$ and condition \eqref{eq:cumk1} ensures the span of dependence is small enough for the existence -- in a Hilbert-Schmidt sense -- of $n$-th order spectral cumulant tensors (functional polyspectra), which can be seen to capture nonlinear dynamics of the process for $n \ge 3$. 
When not assumed explicitly, it can in general be difficult to verify if conditions of the form \eqref{eq:cumk1} are satisfied. Because of their natural relation to higher order functional polyspectra and their frequent usage as an underlying assumption in existing literature on statistics that are of the form \eqref{eq:Q1}, it is however of interest to understand how these relate to the cumulative measure in \autoref{as:depstruc}, which is more easily tractable. In order to do so, we make a few simple observations. Firstly,  cumulant operators can be viewed as generalizations of covariance operators, which is easily seen from rewriting  \eqref{eq:cumtens} as
\[
\mathrm{cum}(X_{t_1},\ldots, X_{t_n}) =\E[X_{t_1} \otimes \cdots \otimes X_{t_n}] - \sum_{\boldsymbol{v}; |\rho| \ne 1}\Pi_{\boldsymbol{v}}\Big( \otimes_{r=1}^{\rho} \mathrm{cum}(X_{t_i}; i \in v_r )\Big). \tageq \label{eq:cumcov}
\]
Hence, the $n$-th order cumulant tensor can be interpreted as a measure of the $n$-th order joint dependence  corrected for by all lower order joint dependencies, i.e., it 
captures the interaction between the $n$ variables that is not captured by any subset of the $n$ variables.  
Secondly, the cumulant of order $n$ of a stationary process is translation invariant and is a function of $n-1$ time differences w.r.t. a chosen base time, e.g., $\mathrm{cum}(X_{t_1},\ldots, X_{t_n}) = \mathrm{cum}(X_{t_1-t_n},\ldots, X_{t_{n-1}-t_n}, X_0)$. 

For $p=n=2$, a direct relation is then in fact given in Proposition \autoref{prop:sumCh}, which shows that $\sum_{t=0}^{\infty}\|X_0 - \E[X_0|\G_{0,\{-t\}}]\|_{\hi,2} < \infty$  implies $\sum_{t\in \znum} \snorm{C_t}_2=\sum_{t\in \znum} \snorm{\,\mathrm{cum}(X_{t},X_0)}_{2} < \infty.$
Note that this is intuitive since \autoref{as:depstruc} provides a cumulative measure of the dependence of $X_0$ on $\epsilon_{-t}$, i.e., the dependence on the element of the data-generating mechanism $t$ lags apart. This captures all second order dynamics; there is only one direction from the base time in which the dependence span needs to be controlled for and \autoref{as:depstruc} with $p=n=2$ suffices. 

This measure appears however inadequate to capture the higher order dynamics that come into play for $n \ge 3$.  More specifically, the coefficients of dependence \eqref{eq:depstruc3} cannot directly control for the interactions in the various directions since the lags between the $n-1$ indices cannot be exploited. Tedious calculations in the appendix indicate that in order to capture the magnitude of these dynamics, we require a  generalization of \eqref{eq:depstruc3} to
 \[\nu_{X,\hi,p}(j_{1},\ldots,j_{n-1}): = \bignorm{X_0+\sum_{i=1}^{n-1}(-1)^{i} \sum_{1\le l_1 < \ldots <l_{i}\le n-1} \E[X_0 |\G_{0,\{-j_{l_1},\ldots,-j_{l_i}\}}]}_{\hi,p} \tageq \label{eq:sufcon2}\]
for $j_1, \ldots, j_{n-1} \ge 0$. Observe that  \eqref{eq:sufcon2} accounts for the dependence in all $n-1$ directions from the base time simultaneously, corrected for by  the ``over-counted'' interactions (in spirit of the exclusion-inclusion principle). It therefore fully encapsulates the dynamics of a $n$-th order joint cumulant tensor. We remark that a slightly weaker (yet less intuitive) condition was derived in the calculations in terms of composite projection operators (see \eqref{eq:tightbound}). 
Using \eqref{eq:sufcon2}, we obtain the following sufficient condition to ensure summability of cumulant tensors.
\begin{thm} \label{thm:cumrel4}\vspace{-5pt}
Suppose $\{X_t\colon t\in\mathbb{Z}\}$ is a centered stationary functional time series in $\op^p_H$ with $p=n \ge 2$.
Suppose moreover that
\begin{align*}
\sum_{j_{1}, \dots, j_{k}=0}^{\infty}\nu_{X,\hi,p}(j_{1},\ldots,j_{k}) < \infty \quad \text{ for } k= n-1. \tageq \label{eq:sufcon}
\end{align*}
Then, for all $i \le p$, \[\sum_{t_1,\ldots, t_{i-1} \in \znum}\snorm{\, \mathrm{cum}(X_{t_1},\ldots, X_{t_{i-1}}, X_{0})}_{2} < \infty.\] 
\end{thm}
For $k\ge 2$, \eqref{eq:sufcon} can be seen to capture the nonlinear dynamics in the process, while it coincides with the standard measure in \eqref{eq:depstruc2} for $k=1$. Condition \eqref{eq:sufcon} is weak in the sense that for $k \ge 2$ it is in fact identically zero for linear processes, in which case \eqref{eq:sufcon} does not impose additional constraints. Neither summability of cumulants nor  \autoref{as:depstruc} are however required for linear processes \citep[see e.g.,][]{pt13,vde16}. 
The following provides an upper bound on \eqref{eq:sufcon} in terms of a weighted version of the standard coefficients of dependence.
\begin{prop} \label{prop:sufcon}
For $p \ge 2$,
\[
\sum_{j_{1}, \dots, j_{k}=0}^{\infty}\nu_{X,\hi,p}(j_{1},\ldots,j_{k}) \le 2^{k-1} \sum_{j=0}^{\infty} j^{k-1}\nu_{\hi,p}(X_j) \quad \text{ for all } 1 \le k \le p-1. \tageq \label{eq:lv}
\]
\end{prop}
Observe that for $k=1$ this is an equality.  Additionally, note that the exponential factor in the upper bound on the right hand side of \eqref{eq:lv} increases in $k$. As this provides an upperbound for the $k+1$-th order cumulant tensor, this can be seen to control the interaction in the various directions (and hence the dependence span) in  a very rough way. It is worth mentioning that our findings corroborate with those in \cite{ZhangWu18} who obtain a similar bound  to ensure summability of the third order cumulant of a univariate time series.
 
\medskip
\acknowledgements{This work has been supported by the Collaborative Research Center ``Statistical modeling of nonlinear dynamic processes'' (SFB 823, Teilprojekt A1, C1) of the German Research Foundation
(DFG). The author kindly thanks Holger Dette and Davide Giraudo for some useful discussions. Furthermore, the author sincerely thanks the Editor and two anonymous referees for their constructive comments that helped to produce an improved version of the original paper.}

\appendix

\section{Inequalities for $H$-valued martingales and linear transforms} \label{sec:Ap1}

Let $H$ be a Hilbert space. For a probability space $(\Omega,\mathcal{A},\mathcal{G}_{\infty},\mathbb{P})$ and $\G=\{\G_t\}_{t \ge 0}$ a non-decreasing sequence of sub-$\sigma$-fields of $\G_{\infty}$, let $\{M_t\} \in \op^p_H$ be a martingale with respect to $\G$ and note that we can write $M_n=\sum_{k=0}^{n}D_k$, where $\{D_k\}$ denotes its difference sequence. Additionally denote the variable
\[V_n(M)=(\sum_{k} \|D_k\|^2_H )^{1/2}\]
which we call the square function of $M$. It was shown \cite[][theorem 3.1]{Burk88} that for $H$-valued martingales, we have  for $1<p<\infty$
 \begin{align}
 (p^\star-1)^{-1}(\E |V(M)|^p)^{1/p} \le   (\E \|M\|_H^p)^{1/p}\le  (p^\star-1)(\E |V(M)|^p)^{1/p} \label{eq:burkH}
 \end{align}
 where $p^\star=\max(p,\frac{p}{p-1})$. As a consequence we have the following lemma, which extends lemma 1 of \cite{WuSh07}. 
\begin{lemma}\label{lem:Burkh}
Let ${\{M_k\}}_{k=1,\ldots,n} \in \op^p_H, p>1$, be a martingale with respect to $\G$ with $\{D_k\}$ denoting its difference sequence and let $\{A_k\}_{k=1,\ldots,n} \in S_\infty(H)$. Then, for $q=\min(2,p)$,
\[
\Bignorm{\sum_{k=1}^{n} A_k (D_k)}^q_{\hi,p} \le K^q_{p} \sum_{k=1}^{n} \snorm{A_k}_{\infty}^q \|D_k\|^q_{\hi,p}
\]
where $K^q_{p}= (p^\star-1)^q$ with $p^\star=\max(p,\frac{p}{p-1})$.
\end{lemma}
\begin{proof}[Proof of \autoref{lem:Burkh}]
By Burkholder's inequality \eqref{eq:burkH}
\begin{align*}
\Bignorm{\sum_{k=1}^{n}A_k (D_k)}^q_{\hi,p} =\Big(\E\Bignorm{\sum_{k=1}^{n} A_k (D_k)}^p_H\Big)^{q/p}  \le 
(p^\star-1)^q \Big(\E  \Big| \Big(\sum_{k=1}^{n}\|A_k (D_k)\|^2_H\Big)^{1/2}\Big|^p\Big)^{q/p}.
\end{align*}
Let $p<2$. Then, applying the inequality $|\sum_k x_k|^{r} \le \sum_k |x_k|^r$ for $r<1$ to $x_k = \|A_k(D_k)\|^2_{H}$, we obtain
\begin{align*}
(p^\star-1)^q \Big(\E \Big| \Big(\sum_{k=1}^{n}\|A_k (D_k)\|^2_H\Big)^{1/2}\Big|^{p}\Big)^{q/p} 
& \le (p^\star-1)^q \Big( \E\sum_{k=1}^{n}  \|A_k (D_k)\|^{p}_H\Big)^{q/p}  
\\& \le (p^\star-1)^q \Big( \sum_{k=1}^{n}  \snorm{A_k}_{\infty}^p \E\| D_k\|^{p}_H\Big)^{q/p} 
\\&  \le (p^\star-1)^q \sum_{k=1}^{n}  \snorm{A_k}^q _{\infty}\|D_k\|^q_{\hi,p},
\end{align*}
where the one before last inequality follows from Holder's inequality for operators and from the fact that $p=q$. 
For $p\ge2$, $q=2$. Therefore, an application of Minkowski's inequality to $\|\cdot\|_{\mathbb{C},{p/q}}$ and Holder's inequality yield in this case
\begin{align*}
(p^\star-1)^q \Big(\E \Big| \Big(\sum_{k=1}^{n}\|A_k (D_k)\|^q_H\Big)^{1/q}\Big|^p\Big)^{q/p} 
& \le 
(p^\star-1)^q \Big(\E \Big\vert\sum_{k=1}^{n}\|A_k (D_k)\|^q_H\Big\vert^{p/q}\Big)^{q/p} 
\\& \le (p^\star-1)^q \sum_{k=1}^{n} \snorm{A_k}^q _{\infty}( \E\|D_k\|^p_{H})^{q/p}
\\&  \le (p^\star-1)^q\sum_{k=1}^{n}  \snorm{A_k}^q _{\infty}\|D_k\|^q_{\hi,p}.
\end{align*}
\end{proof}

\begin{lemma} \label{lem:lineq}
For  $t=1,\ldots, n$, let $\{X_t\}$ be a zero-mean stationary ergodic process in $\op^{p}_H$ and $\{A_t\} \in S_\infty(H)$. Then, 
\begin{align*}
\mathrm{(i)}\, &\bignorm{ \sum_{t=1}^{n} A_t (X_t)}^q_{\hi,p} \le K^q_{p}\snorm{A_n}^q_{\ell_q}\boldsymbol{\Delta}^q_{p,1,0}%
,\quad \mathrm{(ii)}\, \bignorm{ \sum_{t=1}^{n} A_t X^{(m)}_t}^q_{\hi,p} \le  K^q_{p}\snorm{A_n}^q_{\ell_q}%
\boldsymbol{\Delta}^q_{p,1,0} ,\\
\mathrm{(iii)} &\,\bignorm{ \sum_{t=1}^{n} A_t (X_t-X^{(m)}_t)}^q_{\hi,p} \le  K^q_{p}\snorm{A_n}^q_{\ell_q}
\boldsymbol{\Delta}^q_{p,1,m+1}. 
\end{align*}
where $\boldsymbol{\Delta}_{p,q^\prime,m} =\sum^{\infty}_{j=m}  \nu^{q^\prime}_{\hi,p,j}$ and $\snorm{A_n}^q_{\ell_q}=\sum_{t=1}^{n} \snorm{A_t}_{\infty}^q $.
\end{lemma}
\begin{proof}
Using \eqref{eq:ProjX} and \autoref{lem:Burkh} $\mathrm{(i)}$ directly follows. For $\mathrm{(ii)}$, by stationarity
\[
\|P_j(X^{(m)}_t)\|_{\hi,p} 
= \|\E[X_{t-j}-X_{t-j,\{0\}}|\G_{t,t-m}]\|_{\hi,p} \le \nu_{\hi,p}(X_{t-j})
\]
where we abbreviated $X_{t-j,\{0\}}=E[X_{t-j}| \G_{t-j,\{0\}}]$ and therefore $\mathrm{(ii)}$ follows from $\mathrm{(i)}$. Finally, if we write $X_t-X_t^{(m)} = \sum_{j=1+m}^{\infty}\E[X_t|\G_{t,t-j}]-\E[X_t|\G_{t,t-j+1}] 
$ then $D_{t,j}:=\E[X_t|\G_{t,t-j}]-\E[X_t|\G_{t,t-j+1}]$ for $t=n,\ldots, 1$ defines a martingale difference with respect to the backward filtration $\G(\epsilon_t,\ldots, \epsilon_i)$, $i=0,-1,\ldots$. $\mathrm{(iii)}$ now follows from noting by the contraction property and stationarity
\begin{align}
\|D_{t,j}\|_{\hi,p} = \|\E[ (X_t-X_{t,\{t-j\}})|\G_{t,t-j}] \|_{\hi,p} & \le  \| (X_t-X_{t,\{t-j\}}) \|_{\hi,p} 
\\&=  \| (X_j-X_{j,\{0\}}) \|_{\hi,p} = \nu_{\hi,p}(X_j). \label{eq:Xcoup}
\end{align}
\end{proof}

\begin{proof}[Proof of Proposition \autoref{prop:propertiesDM}]
(i) Since the process $\{X^{(m)}_{t}\}$ is $m$-dependent it is immediate that the $D_{m,k}$ are also $m$-dependent. Hence, we may write $D^{\lambda}_{m,0}=\sum_{t=0}^{m} P_{0}(X^{(m)}_t) e^{-\im \lambda t}$. By orthogonality, $\E\|D_{m,k}\|^2_H \le \sum_{t=0}^{\infty} \E\|P_{0}(X_t) \|^2_H < \infty$. Next, observe that \[\E [D^{(\lambda)}_{m,k}|\G_{k-1}]=  \frac{1}{\sqrt{2\pi}} \sum_{t=0}^{\infty} \E \big[\E[X^{(m)}_{t+k}|\G_k] -\E[X^{(m)}_{t+k}|\G_{k-1} ]|\G_{k-1}\big]e^{-\im t \lambda}= 0\]
by the properties of the conditional expectation.\\
 (ii) Under \autoref{as:depstruc} with $p=4$, we obtain from \autoref{lem:Burkh}
\begin{align*}
\E \snorm{D^{(\lambda)}_{m,k} \otimes D^{(\lambda)}_{m,k}}^2_2 = \E \|D^{(\lambda)}_{m,k} \|^4_H & \le\big( \sum_{t=0}^{\infty}\|P_0(X^{(m)}_t)\|^2_{\hi,4})^2 
\\& \le \big( \sum_{t=0}^{\infty}\|P_0(X_t)\|^2_{\hi,4})^2 \le \big( \sum_{t=0}^{\infty} \nu^2_{\hi,4}(X_t)\big)^2<\infty.
\end{align*}
Secondly, observe that for all $n_1,n_2 \in \mathbb{N}$ such that $n_2 \ge n_1$, we have using \autoref{lem:Burkh}
\begin{align*}
\E \|D^{(\lambda)}_{m,k,n_2} -D^{(\lambda)}_{m,k,n_2}\|^4_H  =\big( (\E \|D^{(\lambda)}_{m,k,n_2} -D^{(\lambda)}_{m,k,n_1}\|^4_H)^{1/2} \big)^2 = \big( \|D^{(\lambda)}_{m,k,n_2} -D^{(\lambda)}_{m,k,n_1}\|^2_{\hi,4} \big)^2
\le (\sum_{t=n_1+1}^{n_2}\|P_0(X_t)\|^2_{\hi,4})^2
\end{align*}
from which it is clear that $\{D^{(\lambda)}_{\infty,k,T}\}_{\{T \ge 1\}}$ is Cauchy in $\op^4_H$. Trivially, $\{D^{(\lambda)}_{m,k,T}\}_{\{T \ge 1\}}$ is therefore Cauchy in $\op^4_H$, uniformly in $m$. To ease notation, let $Y_{n}:=D^{(\lambda)}_{m,k,n}$. Now observe that for all $n_1,n_2\in \mathbb{N}$
\begin{align*}
\E \snorm{Y_{n_2} \otimes Y_{n_2} -Y_{n_1} \otimes Y_{n_1}}^2_2 & \le  2\E \snorm{ (Y_{n_2}-Y_{n_1}) \otimes Y_{n_2}}^2_2 +2 \E\snorm{Y_{n_1}  \otimes  (Y_{n_2}-Y_{n_1})}^2_2
\\& \le 
 2\E \| (Y_{n_2}-Y_{n_1})\|^2_H \|Y_{n_2}\|_H^2 +2\E \|Y_{n_1}\|_H^2 \| (Y_{n_2}-Y_{n_1})\|^2_H
 \\& \le 
 2 (\E \| (Y_{n_2}-Y_{n_1})\|^4_H \E\|Y_{n_2}\|_H^4)^{1/2} +2 (\E \|Y_{n_1}\|_H^4 \E\ (Y_{n_2}-Y_{n_1})\|^4_H)^{1/2}
\\& \le 4 (\epsilon N)^{1/2},  
\end{align*}
where we used that $\{Y_{n}\}$ is Cauchy in $\op^4_H$ from which it follows that for all $\epsilon>0$ there exists an $N$ such that for $n_1,n_2\ge N$, $\E \|(Y_{n_2} -Y_{n_1})\|^4_H < \epsilon$ and $\E\|Y_{n}\|^4_H <N$. 
Next we prove (iii). 
First we need to prove that
\[
\lim_{m \to \infty}\lim_{T \to \infty}\Tr\big(\var(D^{(\lambda)}_{m,0,T})\big)  =\Tr( \F^{(\lambda)})< \infty. \tageq \label{eq:firstthing}\]
Recall that $\Tr\big(\var(D^{(\lambda)}_{m,0,T})\big) =\E\|D_{m,0,T}^{(\lambda)}\|_H^2$ where the latter is finite uniformly in $m$ and $T$ because the limit satisfies $\E\|D^{(\lambda)}_0\|^2_H <\infty$ by property (ii). We shall therefore proceed similar to \citep{PeWu10,CevHor17}.
By stationarity and by the fact that the integral of the complex exponential yields the constraint $t-s=h$
\begin{align*}
\int_{-\pi}^{\pi} \E\|D_{m,0,T}^{(\omega)}\|^2 e^{\im h \omega} d\omega
&=\frac{1}{2\pi} \int_{-\pi}^{\pi} \E \inprod{\sum_{t=0}^{T} P_{0}(X^{(m)}_{t})}{\sum_{s=0}^{T} P_{0}(X^{(m)}_{s})} e^{-\im (t-s-h) \omega} d\omega
\\& =\E \sum_{t=h}^{T}\inprod{ P_{0}(X^{(m)}_{t})}{P_{0}(X^{(m)}_{t-h})}.
\end{align*}
Since $\G_{-t} \subseteq \G_{-h} \forall t \ge h$, the properties of the conditional expectation show that, for any $m \ge 1$, we have $\E\big[\E [ X^{(m)}_0 |\G_{-h}]| \G_{-t}\big] \overset{\op^2_H}{=} \E[X^{(m)}_0| \G_{-t}], \forall t \ge h$. Morevover,  $X^{(m)}_{-h}$ is $\G_{-h}$-measurable. Therefore, we obtain by orthogonality of the projection operators and stationarity that
\begin{align*}
\E \sum_{t=h}^{T}\inprod{ P_{0}(X^{(m)}_{t})}{P_{0}(X^{(m)}_{t-h})}
& =\E \inprod{\sum_{t=h}^{T} P_{0}(X^{(m)}_{t})}{\sum_{s=h}^{T}P_{0}(X^{(m)}_{s-h})} 
\\&=\E \inprod{\sum_{t=h}^{T} P_{-t}(\E [X^{(m)}_{0}|\G_{-h}])}{\sum_{s=h}^{T}P_{-s}(\E[X^{(m)}_{-h}|\G_{-h}])}.
\end{align*}
By ergodicity and from (ii) $\{D^{(\lambda)}_{m,0,T}\}_{\{T \ge 1\}}$ is Cauchy in $\op^2_H$. Thus, $\lim_{T \to \infty} \sum_{t=h}^{T} P_{-t}(\E [X^{(m)}_{0}|\G_{-h}]) \overset{\op^2_H}{=}  \E [X^{(m)}_{0}|\G_{-h}]$ and $\lim_{T \to \infty} \sum_{s=h}^{T}P_{-s}\E[X^{(m)}_{-h}|\G_{-h}] \overset{\op^2_H}{=}X^{(m)}_{-h}$. Therefore, continuity of the inner product yields
\begin{align*}
&\lim_{T \to \infty}\E \inprod{\sum_{t=h}^{T} P_{-t}(\E [X^{(m)}_{0}|\G_{-h}])}{\sum_{s=h}^{T}P_{-s}(\E[X^{(m)}_{-h}|\G_{-h}])}
\\& =\E \inprod{\E [X^{(m)}_{0}|\G_{-h}]}{X^{(m)}_{-h}} =  \E \inprod{ X^{(m)}_0}{X^{(m)}_{-h}} = \Tr(C^{m}_{h}),
\end{align*}
where we used the tower property.  Hence, $\lim_{T \to \infty}\frac{1}{2\pi} \int_{-\pi}^{\pi} \E\|D_{m,0,T}^{(\lambda)}\|_H^2 e^{\im h \lambda} d\lambda= \Tr(C^{(m)}_{h})$. But this holds in particular for $m=\infty$, i.e., for the process $\lim_{m \to \infty}X^{m}_{t} =X_t$.  Now observe that the conditions of the  classical F{\'e}jer-Lebesgue theorem are satisfied and therefore
\begin{align*}
\lim_T \Tr(\var(\mathcal{D}^{\lambda}_{T}) )& = \lim_T \sum_{h \le T} (1-\frac{h}{T})  \E \inprod{ X_h}{X_{0}}e^{-\im h \omega}=\E\|D_{0}^{(\lambda)}\|_H^2 =\Tr(\F^{(\lambda)}) < \infty, \tageq \label{eq:trFlim}\end{align*}
where we used again property (ii) in order to obtain the finite trace. Let $\mathcal{D}^{\omega}_{m,T} $ denotes the functional DFT of $X^{(m)}_t$. Clearly, we have immediately from the above as well that 
\[ \lim_{m \to \infty} \lim_{T \to \infty} \Tr\big(\var(\mathcal{D}^{\omega}_{m,T})\big) = \lim_{m \to \infty} \E\|D_{m,0}^{(\lambda)}\|_H^2=\lim_{m \to \infty}\Tr(\F^{(\lambda)}_m)=\Tr(\F^{(\lambda)}) \tageq \label{eq:trFlim2}\] 
where $2\pi \F^{(\lambda)}_{m} = \sum_{|h| \le m}  \E (X^{(m)}_h \otimes X^{(m)}_{0}) e^{-\im h \lambda}$ and where we applied the dominated convergence theorem which is justified by \eqref{eq:trFlim}. This proves \eqref{eq:firstthing}. Consequently, non-negative definiteness allows us to conclude that $\F^{(\lambda)}_m \in S_1(H)$ for all $m \ge 1$ and any  $\lambda \in (-\pi,\pi]$. Then, using the permutation operator is a unitary operator, Holders' inequality for operators yields
\[\bigsnorm{\Pi_{ijkl}\big(\F^{(\lambda)} {\otimes}{\F^{(\lambda)}})}_1 \le \snorm{\Pi_{ikjl}}_{\infty}\snorm{\F^{(\lambda)} \widetilde{\otimes}{\F^{(\lambda)}}}_1 \le \snorm{\F^{(\lambda)}}^2_{1} = (\E\|D^{(\lambda)}_0\|^2_H)^2 \le \E\|D^{(\lambda)}_0\|^4_H< \infty, \tageq \label{eq:PermTF}\]
where we applied \eqref{eq:trFlim} in the equality and Jensen's inequality together with property (ii) in the last inequality. From continuity of $\widetilde{\otimes}$, $\Pi$ and the dominated convergence theorem together with \eqref{eq:trFlim2}, we obtain 
\[
\lim_{m \to \infty} \lim_{T \to \infty} \Tr\big(\Pi_{ijkl}\var(D^{(\lambda)}_{m,0,T}) \widetilde{\otimes} \var(D^{(\lambda)}_{m,0,T})\big) =
\Tr(\Pi_{ijkl} \F^{(\lambda)} \widetilde{\otimes}\F^{(\lambda)})< \infty.\]
\end{proof}
\section{Joint convergence of finite-dimensional distributions of ${\xi}^{\lambda}_{T,m}$} \label{sec:Ap_fdds}

\begin{proof}[Proof of \autoref{thm:fddsmain}]
We recall that \[
{\xi}^{\lambda}_{T,m}:={(\snorm{\Phi_T}^2_{F})}^{-1/2} ({\ldm{}} +{\ldm{\dagger}}).
\]
We want to show that $\{{\xi}^{\lambda_1}_{T,m},\ldots, {\xi}^{\lambda_d}_{T,m} \}$ are converging jointly to complex Gaussian elements of $S_2(H)$. From Proposition \autoref{prop:Msq}, we know that ${\xi}^{\lambda_j}_{T,m}$ define martingales in $\op^2_{S_2(H)}(\Omega, \mathcal{A},\mathbb{P})$ with respect to the filtration $\{\G_T\}$. Below we shall prove convergence of the finite-dimensional distributions via a martingale central limit theorem on the linear combinations. To make this precise, let $U=\{u_1,\ldots, u_d, v_1,\ldots, v_d \in H\}$. For any $u, v \in H$ note that  we can define the natural filtration of the process $\{\inprod{X_t}{u}\}_t$ over $(\Omega, \mathcal{A},\mathbb{P})$ by $\{\G_t(u)\}$. In the following, we let $\{\G_t(u_j,v_j)\}=\sigma(\{\inprod{X_{t}}{u_j},\inprod{X_{s}}{v_j}\}_{t,s: t\ge s})$ to be the natural filtration over $(\Omega, \mathcal{A},\mathbb{P})$ of the projected process process $\{\inprod{X_t \otimes X_s}{u_j \otimes v_j}_{S}\}_{t,s: t\ge s}$. Correspondingly, we denote the projection operator $P^{(u_j,v_j)}_0 := \E[\cdot|\G_{0}(u_j,v_j)]- \E[\cdot|\G_{-1}(u_j,v_j)]$. More generally, let $\G_t(U)= \sigma(\inprod{X_{t_1}}{u_1},\ldots, \inprod{X_{t_d}}{u_d}, \ldots,\inprod{X_{t_{d-1}}}{v_1},$ $ \inprod{X_{t_{2d}}}{u})$  for all $t=t_1 \ge \ldots \ge t_{2d}$ and $P^{U}_0$ the corresponding projection operator. Observe then that
\begin{align*}
\frac{1}{\snorm{\Phi_T}_{F}}\Big( \inprod{\ldmi{}{j}(u_j)}{v_j} +\inprod{\ldmi{\dagger}{j}(u_j)}{v_j} \Big)
\end{align*}
defines a well-defined martingale process in $\op^2_\cnum(\Omega, \mathcal{A},\mathbb{P})$ with respect to the filtration $\{\G_T(u_j,v_j)\}$. 
In order to derive joint convergence of the finite-dimensional distributions it suffices to show that, for any $a_1,\ldots, a_d \in \rnum$ and $\lambda_i \pm \lambda_j \ne 0 \mod 2\pi$, the process
\begin{align*}
\frac{1}{\snorm{\Phi_T}_{F}}\sum_{j=1}^{d}a_j\Big( \inprod{\ldmi{}{j}(u_j)}{v_j} +\inprod{\ldmi{\dagger}{j}(u_j)}{v_j} \Big)
\end{align*}
converges to a zero-mean complex normal random variable with covariance 
\[\sum_{j=1}^{d}a_j\inprod{\Gamma_m (u_j)}{v_j}= { 4\pi^2} \sum_{j=1}^{d}a_j\Big(\inprod{\F^{(\lambda_j)}_m(v_j)}{v_j}  \inprod{u_j}{\F^{(\lambda_j)}_m(u_j)} +\mathrm{1}_{\{0,\pi\}}(\inprod{\F^{(\lambda_j)}_m(u_j)}{v_j}  \inprod{\F^{(\lambda_j)}_m(u_j)}{v}  ) \Big)\] and pseudocovariance 
\[ \sum_{j=1}^{d}a_j\inprod{\Sigma_m (u_j)}{v_j} = { 4\pi^2}\sum_{j=1}^{d} a_j\Big(\mathrm{1}_{\{0,\pi\}}\big(\inprod{\F^{(\lambda_j)}_m(u_j)}{u_j}\inprod{\F^{(\lambda_j)}_m(v_j)}{v_j} \big)+ \inprod{\F^{(\lambda_j)}_m(u_j)}{v_j}  \inprod{\F^{(\lambda_j)}_m(u_j)}{v_j}  \Big).\]
Note that this process is adapted to the filtration $\G_{T}(U)$. We shall do this by means of the Cram{\'e}r-Wold device. We first decompose the functional processes $\ldmi{}{j}$ as  
\begin{align*}
\ldm{} = \sum_{t=2}^{T}  \dm{t} \otimes \Big(\sum_{s=1}^{t-4m} \phi^{(\lambda)}_{T,s-t} \dm{s}+\sum_{s=t-4m+1 \vee 1}^{t-1} \phi^{(\lambda)}_{T,s-t} \dm{s} \Big).
\end{align*}
The following lemma shows the second sum is of lower order in norm. 
\begin{lemma}\label{lem:secM}
Under the conditions of \autoref{thm:funcdist}
\begin{align*}
&\bignorm{\sum_{t=2}^{T} \dm{t}  \otimes \big(\sum_{s=t-4m+1 \vee 1}^{t-1} \phi^{(\lambda)}_{T,s-t} \dm{s}\big)}_{S_2,2} = o(\snorm{\Phi_T}_{F}). 
\end{align*}
\end{lemma}
This implies in turn that we can focus on the distributional properties of the projections of the operators
\begin{align}
\sum_{t=4m+1}^{T}  \dm{t} \otimes N^{(\lambda)}_{m,t} \quad \text{ and } \quad \Big(\sum_{t=4m+1}^{T}  \dm{t} \otimes N^{(\lambda)}_{m,t}\Big)^{\dagger}, \label{eq:CLTterm}
\end{align}
where 
\begin{align}N^{(\lambda)}_{m,t}:=\sum_{s=1}^{t-4m} \phi^{(\lambda)}_{T,s-t} \dm{s}. \label{eq:N4m}
\end{align} 
From Proposition \autoref{prop:Msq}, it is immediate that both terms in \eqref{eq:CLTterm} constitute well-defined martingales in $\op^2_{S_2(H)}(\Omega,\mathcal{A},\G_{T},\mathbb{P})$. Consequently, projecting these on fixed $u,v \in H$, we obtain the following two martingale processes with paths in $\cnum$
\begin{align}
\biginprod{\sum_{t=4m+1}^{T}  \dm{t} \otimes N^{(\lambda)}_{m,t}}{v \otimes u}_S &= 
\sum_{t=4m+1}^{T} \inprod{ \dm{t}}{v} \overline{\inprod{N^{(\lambda)}_{m,t}}{u}}, \label{eq:CLTpr1}
\\ \biginprod{ \Big(\sum_{t=4m+1}^{T}  \dm{t} \otimes N^{(\lambda)}_{m,t}\Big)^{\dagger}}{v \otimes u}_S &= 
\sum_{t=4m+1}^{T} \overline{\inprod{ \dm{t}}{u}} \inprod{N^{(\lambda)}_{m,t}}{v}.  \label{eq:CLTpr2}
\end{align}
In order to apply a martingale central limit theorem on the sum of \eqref{eq:CLTpr1} and \eqref{eq:CLTpr2} and over $j=1,\ldots, d$, we must verify Lindeberg's condition is satisfied. Without loss of generality we do this for  \eqref{eq:CLTpr1} and for fixed $u,v$ as the result is immediate to carry over to a finite sum over $j$. To ease notation in the following, we set $\inprod{ \dm{t}}{x}:= \dm{t}(x)$ and $\inprod{ N^{(\lambda)}_{m,t}}{x}:=N^{(\lambda)}_t(x)$ for any $x \in H$. Recall the inequality $\E\{\|Y\|^2_{H} {\mathrm{1}_{\|Y\|_{H}}>\epsilon}\} \le \frac{1}{\epsilon^2} \E\|Y\|^4_{H}$ which holds for any $Y \in \op^4_H$. Hence applying this for $H=\cnum$, we have 
\begin{align*}
\sum_{t=1+4m}^{T}\E\Big\{ \big|{\dm{t}}(v)\overline{{N^{(\lambda)}_t}(u)}\big|^2 \mathrm{1}_{|{\dm{t}}(v)\overline{{N^{(\lambda)}_t}(u)}|>\epsilon}\Big\} \le \frac{1}{\epsilon^2}\sum_{t=1+4m}^{T}\E |{\dm{t}}(v)\overline{{N^{(\lambda)}_t}(u)}|^4.
\end{align*}
Lindeberg's condition is therefore satisfied if we can show that the term on the right hand side is of order $o(\|\phi_T\|^4_{F})$. Since the $\{\dm{t}\}$ are $m$-dependent and by definition of \eqref{eq:N4m} $|t-s|\ge4m$, $\dm{t}$ and $N_{t}$ are independent. Therefore, using \autoref{lem:Burkh} with $H =\cnum$ yields
\begin{align*}
\sum_{t=1+4m}^{T}\E |{\dm{t}}(v)\overline{{N^{(\lambda)}_t}(u)}|^4 
& \le \|{D_0}\|^4_{{\hi},4}  \|v\|^4_{{H}}\sum_{t=1+4m}^{T} ( \|\overline{{N^{(\lambda)}_t}(u)}\|^2_{\cnum,4} )^2
\\& \le \|{D_0}\|^4_{{\hi},4} \|v\|^4_{{H}}\sum_{t=1+4m}^{T}  (K^2_4 \snorm{\Phi_T}^2_{\ell_2}\| \|u\|^2_{H} \|D_0\|^2_{\hi,4} )^2 
\\& = O(T  \varrho_T^4) = o(\snorm{\Phi_T}^4_{F}) \tageq \label{eq:DU4}
\end{align*}
and similarly for \eqref{eq:CLTpr2}, showing that Lindeberg's condition is satisfied. 
It therefore remains to verify the that the conditional variance satisfies
\begin{align*}
\frac{1}{\snorm{\Phi_T}^2_{F}} \sum_{t=1+4m}^{T} \E\Big(\Big\vert\sum_{j=1}^d a_j\big( {\dmi{t}{j}}(v_j)\overline{{N^{(\lambda_j)}_t}(u_j)}+\overline{{\dmi{t}{j}}(u_j)} {N^{(\lambda_j)}_t}(v_j)\big)\Big\vert^2 \Bigg\vert \G_{t-1}^{(U)}\Big) \overset{p}{\to} \sum_{j=1}^{d}a_j \Gamma^{\lambda_j}_m  \tageq \label{eq:Mar3}
\end{align*}
and that the conditional pseudocovariance satisfies
\begin{align*}
\frac{1}{\snorm{\Phi_T}^2_{F}} \sum_{t=1+4m}^{T} \E\Big(\Big(\sum_{j=1}^d a_j\big({\dmi{t}{j}}(v_j)\overline{{N^{(\lambda_j)}_t}(u_j)}+\overline{{\dmi{t}{j}}(u_j)} {N^{(\lambda_j)}_t}(v_j)\big)\Big)^2 \Bigg\vert \G_{t-1}^{(U)}\Big) \overset{p}{\to} \sum_{j=1}^{d}a_j \Sigma^{\lambda_j}_m   \tageq \label{eq:Mar2}.
\end{align*}
Moreover, observe that we can write $\E(\cdot|\G_{t-1}^{(U)}) = \sum_{k=1}^{m} P^{(U)}_{t-k}(\cdot)+\E(\cdot|\G_{t-m-1}^{(U)})$. We will show that the sum of projections are of lower order. For \eqref{eq:Mar2}, orthogonality of the $P^{(U)}_j(\cdot)$ and the contraction property of the expectation give
\begin{align*}
&\bignorm{\sum_{t=1+4m}^{T}  \sum_{k=1}^{m} P^{(U)}_{t-k}\Big( \big[\sum_{j=1}^d a_j\big({\dmi{t}{j}}{(v_j)}\overline{{N^{(\lambda_j)}_t}{(u_j)}}+\overline{{\dmi{t}{j}}{(u_j)}} {N^{(\lambda_j)}_t}{(v_j)}\big)\big]^2\Big)}^2_{\mathbb{C},2}
\\&
\le
\Big(\sum_{k=1}^{m}\bignorm{  \sum_{t=1+4m}^{T} P^{(U)}_{t-k}\Big( \big[\sum_{j=1}^d a_j\big({\dmi{t}{j}}{(v_j)}\overline{{N^{(\lambda_j)}_t}{(u_j)}}+\overline{{\dmi{t}{j}}{(u_j)}} {N^{(\lambda_j)}_t}{(v_j)}\big)\big]^2\Big)}_{\mathbb{C},2}\Big)^2
\\& \le
\Big(\sum_{k=1}^{m} \Big( \sum_{t=1+4m}^{T}\bignorm{\Big( \sum_{j=1}^d a_j\big({\dmi{t}{j}}{(v_j)}\overline{{N^{(\lambda_j)}_t}{(u_j)}}+\overline{{\dmi{t}{j}}{(u_j)}} {N^{(\lambda_j)}_t}{(v_j)}\big)\Big)}^4_{\mathbb{C},4} \Big)^{1/2}\Big)^2
    \\& = O\big(m^2 \sum_{t=1+4m}^{T}    \max_j  \|{\dmi{t}{j}}{(v_j)}\|^4_{\mathbb{C},4}\| {{N^{(\lambda_j)}_t}{(u_j)}}\|^4_{\mathbb{C},4} \big)
    \\&=o(\snorm{\Phi_T}^4_{F}),
\end{align*}
where we used again that $\dm{t}$ and $N^{(\lambda)}_t$ are independent for any $\lambda$ and where the order follows in a similar manner to \eqref{eq:DU4}. Furthermore, observe that, for any $x \in U$ and any $\lambda$, $N^{(\lambda)}_t(x)$ is $\G^{(U)}_{t-4m}$ and $\dm{t}(x)$ is $\G^{(U)}_{t,t-m}$ measurable. The left-hand side of \eqref{eq:Mar3} therefore equals
\begin{align*}
 &\frac{1}{\snorm{\Phi_T}^2_{F}}\sum_{t=1+4m}^{T} \E\Bigg(\Big\vert\sum_{j=1}^{d}a_j \big(\dmi{t}{j} (v_j)\overline{{N^{(\lambda_j)}_t}(u_j)}+\overline{{\dmi{t}{j}}(u_j)} {N^{(\lambda_j)}_t}(v_j)\big)\Big\vert^2 \big\vert\G^{(U)}_{t-m-1}\Bigg) +o_p(1)
\\&= \frac{1}{\snorm{\Phi_T}^2_{F}}\sum_{t=1+4m}^{T} \sum_{j=1}^d a^2_j \Big( |\overline{{N^{(\lambda_j)}_t}{(u_j)}}|^2\E\big|{\dmi{t}{j}}{(v_j)}|^2 +|{N^{(\lambda_j)}_t}{(v_j)}|^2 \E|\overline{{\dmi{t}{j}}{(u_j)}}|^2  
\\&\phantom{\frac{1}{\snorm{\Phi_T}^2_{F}}\sum_{t=1+4m}^{T}} + \overline{{N^{(\lambda_j)}_t}{(u_j)}} \overline{{N^{(\lambda_j)}_t}{(v_j)}} \E {\dmi{t}{j}}{(v_j)} {{\dmi{t}{j}}{(u_j)}}
+{N^{(\lambda_j)}_t}{(v_j)}{{N^{(\lambda_j)}_t}{(u_j)}} \E \overline{{\dmi{t}{j}}{(u_j)}}  \overline{{\dmi{t}{j}}{(v_j)}}\Big)
\\&  
 \phantom{\frac{1}{\snorm{\Phi_T}^2_{F}}\sum_{t=1+4m}^{T}} + \sum_{i \ne j} a_i a_j \big(\overline{{N^{(\lambda_i)}_t}(u_i)}{{N^{(\lambda_j)}_t}(u_j)}\E[\dmi{t}{i}(v_i)\overline{\dmi{t}{j}(v_j)}]+{N^{(\lambda_i)}_t}(v_i){{N^{(\lambda_j)}_t}(u_j)}\E[\overline{{\dmi{t}{i}}(u_i)}\overline{\dmi{t}{j}(v_j)}] 
 \\&
 \phantom{\frac{1}{\snorm{\Phi_T}^2_{F}}\sum_{t=1+4m}^{T}} 
 +\overline{{N^{(\lambda_i)}_t}(u_i)}\overline{N^{(\lambda_j)}_t}(v_j)\E[ \dmi{t}{i} (v_i){{\dmi{t}{j}}(u_j)}]+ {N^{(\lambda_i)}_t}(v_i) \overline{N^{(\lambda_j)}_t(v_j)}\E[\overline{{\dmi{t}{i}(u_i)}}{{\dmi{t}{j}}(u_j)}] \big) +o_p(1),
\end{align*}
while \eqref{eq:Mar2} becomes
\begin{align*}
 &\frac{1}{\snorm{\Phi_T}^2_{F}}\sum_{t=1+4m}^{T} \E\Bigg(\Big(\sum_{j=1}^{d}a_j \big(\dmi{t}{j} (v_j)\overline{{N^{(\lambda_j)}_t}(u_j)}+\overline{{\dmi{t}{j}(u_j)}} {N^{(\lambda_j)}_t}(v_j)\big)\Big)^2 \big\vert\G^{(U)}_{t-m-1}\Bigg) +o_p(1)
\\&
=\frac{1}{\snorm{\Phi_T}^2_{F}}\sum_{t=1+4m}^{T} \sum_{j=1}^d a^2_j\big( (\overline{{N^{(\lambda_j)}_t}{(u_j)}})^2 \E({\dmi{t}{j}}{(v_j)})^2+({N^{(\lambda_j)}_t}{(v_j)})^2 \E(\overline{{\dmi{t}{j}}{(u_j)}})^2 \\&\phantom{\frac{1}{\snorm{\Phi_T}^2_{F}}\sum_{t=1+4m}^{T} \sum_{j=1}^d} +2\overline{{N^{(\lambda_j)}_t}{(u_j)}} {N^{(\lambda_j)}_t}{(v_j)}\E{\dmi{t}{j}}{(v_j)}\overline{{\dmi{t}{j}}{(u_j)}}\big)
\\&  
 \phantom{\frac{1}{\snorm{\Phi_T}^2_{F}}} + \sum_{i \ne j} a_i a_j \big(\overline{{N^{(\lambda_i)}_t}(u_i)}\overline{{N^{(\lambda_j)}_t}(u_j)}\E[\dmi{t}{i}(v_i){\dmi{t}{j}}(v_j)]+{N^{(\lambda_i)}_t}(v_i)\overline{{N^{(\lambda_j)}_t}(u_j)}\E[\overline{{\dmi{t}{i}}(u_i)}{\dmi{t}{j}}(v_j)] 
 \\&
 \phantom{\frac{1}{\snorm{\Phi_T}^2_{F}}} 
 +\overline{{N^{(\lambda_i)}_t}(u_i)}{N^{(\lambda_j)}_t}(v_j)\E[ \dmi{t}{i} (v_i)\overline{{\dmi{t}{j}}(u_j)}]+ {N^{(\lambda_i)}_t}(v_i){N^{(\lambda_j)}_t}(v_j)\E[\overline{{\dmi{t}{i}}(u_i)}\overline{{\dmi{t}{j}}(u_j)}] \big)+o_p(1).
\end{align*}
We require the following lemma.
\begin{lemma}\label{lem:Zt}
Let $\{\dm{t}\} \in \op^4_H$ be a $H$-valued martingale difference process. Then, provided that conditions (i) and (iv) of \autoref{as:phi} are satisfied, we have
\begin{align*}
\bignorm{\frac{1}{\snorm{\Phi_T}^2_{F}}\sum_{t=2}^{T} M^{(\lambda_1)}_{t}\otimes {M}^{(\lambda_2)}_{t} - \E M^{(\lambda_1)}_0 \otimes \overline{M}^{(\lambda_2)}_0}_{S_2,2} =
o(1),
\end{align*}
where $M^{(\lambda)}_{t}:=\sum_{s=1}^{t-1} \phi^{(\lambda)}_{T,s-t} \dm{s}$.
\end{lemma}
Since norm convergence implies convergence in the weak operator topology, we obtain for any $u, v \in H$
\begin{align*}
 \bignorm{\frac{1}{\snorm{\Phi_T}^2_{F}}\sum_{t=1+4m}^{T} {{N^{(\lambda_1)}_t}(u) {N^{(\lambda_2)}_t}(v)}- \E {{N^{(\lambda_1)}_t}(u) {N^{(\lambda_2)}_t}(v)}}_{\mathbb{C},2} =o(1).
\end{align*}
Therefore, we may replace them with their expectation in \eqref{eq:Mar3} and \eqref{eq:Mar2} in order to obtain, respectively, for \eqref{eq:Mar3}
\begin{align*}
& \frac{1}{\snorm{\Phi_T}^2_{F}}\sum_{t=1+4m}^{T} \sum_{j=1}^d a^2_j \Big( \E|\overline{{N^{(\lambda_j)}_t}{(u_j)}}|^2\E\big|{\dmi{t}{j}}{(v_j)}|^2 +\E|{N^{(\lambda_j)}_t}{(v_j)}|^2 \E|\overline{{\dmi{t}{j}}{(u_j)}}|^2  
\\&\phantom{\frac{1}{\snorm{\Phi_T}^2_{F}}\sum_{t=1+4m}^{T}} +\E[ \overline{{N^{(\lambda_j)}_t}{(u_j)}} \overline{{N^{(\lambda_j)}_t}{(v_j)}}] \E[ {\dmi{t}{j}}{(v_j)} {{\dmi{t}{j}}{(u_j)}}]
+\E[{N^{(\lambda_j)}_t}{(v_j)}{{N^{(\lambda_j)}_t}{(u_j)}}] \E \overline{{\dmi{t}{j}}{(u_j)}}  \overline{{\dmi{t}{j}}{(v_j)}}\Big)
\\&  
 \phantom{\frac{1}{\snorm{\Phi_T}^2_{F}}} + \sum_{i \ne j} a_i a_j \big(\E[\overline{{N^{(\lambda_i)}_t}(u_i)}{{N^{(\lambda_j)}_t}(u_j)}]\E[\dmi{t}{i}(v_i)\overline{\dmi{t}{j}(v_j)}]+\E[{N^{(\lambda_i)}_t}(v_i){{N^{(\lambda_j)}_t}(u_j)}]\E[\overline{{\dmi{t}{i}}(u_i)}\overline{\dmi{t}{j}(v_j)}] 
 \\&
 \phantom{\frac{1}{\snorm{\Phi_T}^2_{F}}\sum_{t=1+4m}^{T}} 
 +\E[\overline{{N^{(\lambda_i)}_t}(u_i)}\overline{N^{(\lambda_j)}_t}(v_j)]\E[ \dmi{t}{i} (v_i){{\dmi{t}{j}}(u_j)}]+ \E[{N^{(\lambda_i)}_t}(v_i) \overline{N^{(\lambda_j)}_t(v_j)}]\E[\overline{{\dmi{t}{i}(u_i)}}{{\dmi{t}{j}}(u_j)}] \big)
\tageq \label{eq:Cov2}
\end{align*}
and for \eqref{eq:Mar2}
\begin{align*}
&\frac{1}{\snorm{\Phi_T}^2_{F}}\sum_{t=1+4m}^{T} \sum_{j=1}^d a^2_j\Big( \E(\overline{{N^{(\lambda_j)}_t}{(u_j)}})^2 \E({\dmi{t}{j}}{(v_j)})^2+\E({N^{(\lambda_j)}_t}{(v_j)})^2 \E(\overline{{\dmi{t}{j}}{(u_j)}})^2 \\&\phantom{\frac{1}{\snorm{\Phi_T}^2_{F}}\sum_{t=1+4m}^{T} \sum_{j=1}^d} +2\E\overline{{N^{(\lambda_j)}_t}{(u_j)}} {N^{(\lambda_j)}_t}{(v_j)}\E{\dmi{t}{j}}{(v_j)}\overline{{\dmi{t}{j}}{(u_j)}}\Big)
\\&  
 \phantom{\frac{1}{\snorm{\Phi_T}^2_{F}}} + \sum_{i \ne j} a_i a_j \big(\E[\overline{{N^{(\lambda_i)}_t}(u_i)}\overline{{N^{(\lambda_j)}_t}(u_j)}]\E[\dmi{t}{i}(v_i){\dmi{t}{j}}(v_j)]+\E[{N^{(\lambda_i)}_t}(v_i)\overline{{N^{(\lambda_j)}_t}(u_j)}]\E[\overline{{\dmi{t}{i}}(u_i)}{\dmi{t}{j}}(v_j)] 
 \\&
 \phantom{\frac{1}{\snorm{\Phi_T}^2_{F}}} 
 +\E[\overline{{N^{(\lambda_i)}_t}(u_i)}{N^{(\lambda_j)}_t}(v_j)]\E[ \dmi{t}{i} (v_i)\overline{{\dmi{t}{j}}(u_j)}]+ \E[{N^{(\lambda_i)}_t}(v_i){N^{(\lambda_j)}_t}(v_j)]\E[\overline{{\dmi{t}{i}}(u_i)}\overline{{\dmi{t}{j}}(u_j)}]\big).\tageq \label{eq:PseudCov2}
\end{align*}
Next, we make use of the following auxiliary result.
\begin{lemma}\label{lem:inpsZt}
Let $\{\dm{t}\} \in \op^2_H$ be a $H$-valued martingale difference process and let conditions $\mathrm{(ii)}$ and $\mathrm{(iii)}$ in \autoref{as:phi} be satisfied. Furthermore, assume $\lambda_1 \pm\lambda_2 \ne 0 \mod 2\pi$. Then for any $u, v \in H$,
\[\sum_{t=1+4m}^{T} |\E N^{(\lambda_1)}_t(u) N^{(\lambda_2)}_t(v)| =o(\snorm{\Phi_T}^2_{F}),\]
where $N^{(\lambda)}_t$ is as defined in \eqref{eq:N4m}.
\end{lemma}
Suppose first that $d=1$. It follows from this \autoref{lem:inpsZt} that the third and fourth term of \eqref{eq:Cov2} and the first two terms of \eqref{eq:PseudCov2}) will be of lower order if $\lambda \ne 0, \pi$. Hence, from Proposition \autoref{prop:propertiesDM}
\[\E \inprod{\dm{t}}{u}\overline{\inprod{\dm{t}}{v}} = \sum_{|k|\le m}  \inprod{C^{(m)}_{k}(v)}{u}e^{-\im \lambda k} =2\pi \inprod{\F^{(\lambda)}_m(v)}{u}.\]
If $\lambda =0 \mod \pi$, we also have $\E \inprod{\dm{t}}{u}{\inprod{\dm{t}}{v}} =2\pi \inprod{\F^{(\lambda)}_m(v)}{(u)}$. Note that the latter is real for $v=u$. Hence, we obtain for \eqref{eq:Cov2} and \eqref{eq:PseudCov2} 
\begin{align*}
 \frac{1}{\snorm{\Phi_T}^2_{F}}\sum_{t=1+4m}^{T}&\E|\overline{{N^{(\lambda)}_t}{(u)}}|^2\E\big|{\dm{t}}{(v)}|^2 +\E|{N^{(\lambda)}_t}{(v)}|^2 \E|\overline{{\dm{t}}{(u)}}|^2  
\\& + \E \overline{{N^{(\lambda)}_t}{(u)}} \overline{{N^{(\lambda)}_t}{(v)}} \E {\dm{t}}{(v)} {{\dm{t}}{(u)}}+\E{N^{(\lambda)}_t}{(v)}{{N^{(\lambda)}_t}{(u)}} \E \overline{{\dm{t}}{(u)}}  \overline{{\dm{t}}{(v)}}
\\& = \frac{ 8\pi^2\sum_{t=1+4m}^{T} \sum_{s=1}^{t-4m}  w^2_{s-t}}{\snorm{\Phi_T}^2_{F}} \Big(\inprod{\F^{(\lambda)}_m(v)}{v}  \inprod{u}{\F^{(\lambda)}_m(u)} +\mathrm{1}_{\{0,\pi\}}(\inprod{\F^{(\lambda)}_m(v)}{u}  \inprod{\F^{(\lambda)}_m(v)}{u}  ) \Big)
\\& \overset{}{\to}{ 4\pi^2} \Big(\inprod{\F^{(\lambda)}_m(v)}{v}  \inprod{u}{\F^{(\lambda)}_m(u)} +\mathrm{1}_{\{0,\pi\}}( \inprod{\F^{(\lambda)}_m(u)}{v}  \inprod{\F^{(\lambda)}_m(u)}{v}  ) \Big)
\end{align*}
and
\begin{align*}
\frac{1}{\snorm{\Phi_T}^2_{F}}\sum_{t=1+4m}^{T}& \E(\overline{{N^{(\lambda)}_t}{(u)}})^2 \E({\dm{t}}{(v)})^2+\E({N^{(\lambda)}_t}{(v)})^2 \E(\overline{{\dm{t}}{(u)}})^2 +2\E \overline{{N^{(\lambda)}_t}{(u)}} {N^{(\lambda)}_t}{(v)}\E{\dm{t}}{(v)}\overline{{\dm{t}}{(u)}}
\\& \overset{}{\to}  { 4\pi^2} \Big(\mathrm{1}_{\{0,\pi\}}\big(\inprod{\F^{(\lambda)}_m(u)}{u}\inprod{\F^{(\lambda)}_m(v)}{v} \big)+ \inprod{\F^{(\lambda)}_m(u)}{v}  \inprod{\F^{(\lambda)}_m(u)}{v} \Big),
\end{align*}
respectively, as $T\to \infty$. This completes the proof for $d=1$. Next, suppose that $d>1$. If $\lambda_i \pm \lambda_j \ne 0 \mod 2\pi$, then by \autoref{lem:inpsZt} the cross terms are of lower order. Together this establishes the stated convergence in \eqref{eq:Mar3} and \eqref{eq:Mar2}, respectively.
\end{proof}

\subsection{Proofs auxiliary statements} \label{sec:secAux}

\begin{proof}[Proof of \autoref{lem:Zt}]
Let $Y_{T-1} =  \sum_{t=2}^{T}M^{(\lambda_1)}_{t-1} \otimes {M}^{(\lambda_2)}_{t-1} - \snorm{\Phi_T}^2_{F}\E M^{(\lambda_1)}_0 \otimes {M}^{(\lambda_2)}_0 $. 
Observe that from the properties of $\{M^{(\lambda)}_{t-1}\}$, the process $Y_{T-1}$ is $\G_{T-1}$ measurable, stationary and ergodic. 
Ergodicity of the underlying process and a telescoping argument allows us therefore to write 
\begin{align*}
\E\snorm{Y_{T-1}}^2_{2} &= \E\bigsnorm{\sum_{j=-\infty}^{T-1} P_j \Big(\sum_{t=2}^{T}M^{(\lambda_1)}_{t-1} \otimes {M}^{(\lambda_2)}_{t-1}\Big)}^2_{2} 
\\&=\sum_{j=-\infty}^{0}  \E\bigsnorm{ P_j \Big(\sum_{t=2}^{T}M^{(\lambda_1)}_{t-1} \otimes {M}^{(\lambda_2)}_{t-1}\Big)}^2_{2} +\E\sum_{j=1}^{T-1}  \bigsnorm{ P_j \Big(\sum_{t=2}^{T}M^{(\lambda_1)}_{t-1} \otimes {M}^{(\lambda_2)}_{t-1}\Big)}^2_{2}, \tageq \label{eq:Ynorm}
\end{align*}
where we used orthogonality of the projection operators $\{P_{j}\}$. We first consider the first term on the right hand side for which we have $j < 1$. Since $\{\dm{s}\}$ has uncorrelated increments
\begin{align*}
P_j (M^{(\lambda_1)}_{t-1} \otimes {M}^{(\lambda_2)}_{t-1}) &=\sum_{s,s^\prime=1}^{t-1} P_j (\phi^{(\lambda_1)}_{T,s,t} \dmi{s}{1} \otimes  { \phi^{(\lambda_2)}_{T,s^{\prime},t} \dmi{s^\prime}{2}}) = \sum_{s=1}^{t-1} P_j (\phi^{(\lambda_1)}_{T,s,t}\dmi{s}{1}\otimes \phi^{(\lambda_2)}_{T,s,t}{\dmi{s}{2}})
\\& = \sum_{s=1}^{t-1} \big(\phi^{(\lambda_1)}_{T,s,t}\widetilde{\otimes}  \phi^{(\lambda_2)}_{T,s,t}\big) P_j (\dmi{s}{1}\otimes{\dmi{s}{2}}),
\end{align*}
where we used that linear operators and expecation operators commute, i.e., $\E (A X \otimes A X) =\E ( (A \widetilde{\otimes} A) (X \otimes X))=(A \widetilde{\otimes} A)\E (  (X \otimes X))$ for $A \in S_{\infty}(H), X \in \op^4_H$.
Consequently, linearity, orthogonality of the projections, Minkowsk's inequality and stationarity of $\{\dm{s}\}$ yield
\begin{align*}
\sum_{j=-\infty}^{0} \E \bigsnorm{ P_j \Big(\sum_{t=2}^{T}M^{(\lambda_1)}_{t-1} \otimes {M}^{(\lambda_2)}_{t-1}\Big)}^2_{2} & \le \sum_{j=-\infty}^{0} \Big(\sum_{t=2}^{T} \bignorm{ P_j \Big(M^{(\lambda_1)}_{t-1} \otimes {M}^{(\lambda_2)}_{t-1}\Big)}_{S_2,2} \Big)^2
\\&=
\sum_{j=-\infty}^{0}\Big( \sum_{t=2}^{T} \Big( \bignorm{\sum_{s=1}^{t-1} \big(\phi^{(\lambda_1)}_{T,s,t}\widetilde{\otimes}  \phi^{(\lambda_2)}_{T,s,t}\big) P_0 (\dmi{s-j}{1} \otimes  {\dmi{s-j}{2}})}_{S_2,2} \Big)^{2} 
\\& =\sum_{j=-\infty}^{0} \Big(\sum_{t=2}^{T} \sum_{s=1}^{t-1} \bigsnorm{\phi^{(\lambda_1)}_{T,s,t}\widetilde{\otimes}  \phi^{(\lambda_2)}_{T,s,t}}_{\infty}  \bignorm{P_0 (\dmi{s-j}{1} \otimes  {\dmi{s-j}{2}})}_{S_2,2} \Big)^{2} 
\\& \le \sum_{j=-\infty}^{0} \Big(\sum_{s=1}^{T-1} \sum_{t=s+1}^{T}  \snorm{\phi^{(0)}_{T,s,t}}^2_{\infty} \bignorm{P_0 (\dmi{s-j}{1} \otimes  {\dmi{s-j}{2}})}_{S_2,2} \Big)^{2}. 
\end{align*}
the Cauchy-Schwarz inequality implies we obtain under \autoref{as:depstruc} 
\begin{align*}
\sum_{j=-\infty}^{0} &\Big(\sum_{s=1}^{T-1}\bignorm{P_0 (\dmi{s-j}{1} \otimes  {\dmi{s-j}{2}})}_{S_2,2}\sum_{t=s+1}^{T}  \snorm{\phi^{(0)}_{T,s,t}}^2_{\infty} \Big)^{2}
\\&  \le \sum_{j=-\infty}^{0} \Bigg( \Big[\sum_{s=1}^{T-1}\bignorm{P_0 (\dmi{s-j}{1} \otimes {\dmi{s-j}{2}})}^2_{S_2,2} \Big]^{1/2} \Big[\sum_{s=1}^{T-1} \Big( \sum_{t=s+1}^{T} \snorm{\phi^{(0)}_{T,s,t}}^2_{\infty} \Big)^2 \Big]^{1/2}  \Bigg)^{2} 
\\&  \le \sum_{s=1}^{T-1} \Big( \sum_{t=s+1}^{T}  \snorm{\phi^{(0)}_{T,s,t}}^2_{\infty} \Big)^2  \sum_{j=-\infty}^{0}  \sum_{s=1}^{T-1}\bignorm{P_0 (\dmi{s-j}{1} \otimes {\dmi{s-j}{2}})}^2_{S_2,2}
\\&= o(\snorm{\Phi_T}^4_{F}). 
\end{align*}
 For the second term of \eqref{eq:Ynorm}, i.e.,
\begin{align*}
\sum_{j=1}^{T-1} \E \bigsnorm{ P_j \Big(\sum_{t=2}^{T}M^{(\lambda_1)}_{t-1} \otimes {M}^{(\lambda_2)}_{t-1}\Big)}^2_2,
\end{align*} 
we have to distinguish cases since $1\le j \le T-1$. Firstly observe that if $1 \le t \le j$, then $P_j (M^{(\lambda_1)}_{t-1} \otimes {M}^{(\lambda_2)}_{t-1} )= O_H$ since $M^{(\lambda_1)}_{t-1} \otimes {M}^{(\lambda_2)}_{t-1}$ is $\G_{t-1}$ measurable and hence $\E [M^{(\lambda_1)}_{t-1} \otimes {M}^{(\lambda_2)}_{t-1}|G_j]= M^{(\lambda_1)}_{t-1} \otimes {M}^{(\lambda_2)}_{t-1}$. 
We can thus focus on $t >j$. To ease notation, denote $\dm{s} := D_{s}$. Since expectation and tensor operator commute, we obtain for the various cases:
\begin{itemize}\itemsep-0.2ex
\item if $s_1\le j-1$:
\begin{itemize}\itemsep-0.2ex
\item $s_2 > j$, $\E [D_{s_1} \otimes {D}_{s_2}|\G_j] = D_{s_1}\otimes\E [  {D}_{s_2}|\G_j]  = O_H$ and similarly $\E [D_{s_1} \otimes {D}_{s_2}|\G_{j-1}] =O_H$ and therefore $P_j(D_{s_1} \otimes {D}_{s_2})=O_H$.
\item $s_2 = j$: We have 
$\E [D_{s_1} \otimes {D}_{s_2}|\G_j] =D_{s_1} \otimes \E [ {D}_{s_2}|\G_j]  =D_{s_1} \otimes {D}_{s_2}$ while $\E [D_{s_1} \otimes {D}_{s_2}|\G_{j-1}] =O_H$. Hence, $P_j(D_{s_1} \otimes {D}_{s_2})=D_{s_1} \otimes {D}_{s_2}$
\item $s_2 > j-1$: We have again $P_j(D_{s_1} \otimes {D}_{s_2})=O_H$.
\end{itemize}
\item  if $s_1 >s_2 \ge j$: using the tower property, we have 
\[ \E [D_{s_1} \otimes {D}_{s_2}|\G_j] =\E [ \E [D_{s_1} \otimes {D}_{s_2}|\G_{s_2} ]|\G_j] =\E [ \E [D_{s_1} |\G_{s_2} ]\otimes {D}_{s_2}|\G_j]  =O_H. \]
\end{itemize}
Hence,
\begin{align*}
P_j (M_{t-1} \otimes {M}_{t-1}) &
=\sum_{s_1 =1}^{j-1} \big(\phi_{T,s_1,t} \widetilde{\otimes }\phi_{T,j,t}\big) ( D_{s_1} \otimes {D}_{j}) + \sum_{s_2 =1}^{j-1}\big(\phi_{T,j,t}\widetilde{\otimes }\phi_{T,s_2,t}\big)  ( D_{j} \otimes  {D}_{s_2}) 
\\&+ \sum_{s=j+1}^{t-1} \big(\phi_{T,s,t}  \widetilde{\otimes } \phi_{T,s,t}\big)  P_j (\dm{s} \otimes  {\dm{s}})  =: U_j + U^{\dagger}_j + V_j
\end{align*}
and therefore
\begin{align*}
\sum_{j=1}^{T-1} \E \bigsnorm{ P_j \Big(\sum_{t=2}^{T}M_{t-1} \otimes {M}_{t-1}\Big)}^2_2 = \sum_{j=1}^{T-1} \E \bigsnorm{ P_j \Big(\sum_{t=j+1}^{T}M_{t-1} \otimes {M}_{t-1}\Big)}^2_2
\le \sum_{j=1}^{T-1} \E \bigsnorm{\sum_{t=j+1}^{T}  U_j + U^{\dagger}_j + V_j\Big)}^2_2.
\end{align*}
For the first term, stationarity of $\{\dm{s}\}$, the properties of $\widetilde{\otimes}$ and \autoref{lem:Burkh} yield
\begin{align*}
 \sum_{j=1}^{T-1} \E \bigsnorm{\sum_{t=j+1}^{T}  U_j}^2_{2}& = \sum_{j=1}^{T-1} \E \bigsnorm{\sum_{t=j+1}^{T}  \sum_{s_1 =1}^{j-1} \phi_{T,s_1,t}\widetilde{\otimes }\phi_{T,j,t} ( D_{s_1} \otimes  \overline{D}_{j}) \Big)}^2_{2}
 \\&=  \sum_{j=1}^{T-1}  \E\bigsnorm{ \sum_{s_1 =1}^{j-1}\sum_{t=j+1}^{T} \big( \phi_{T,s_1,t}\widetilde{\otimes }\phi_{T,j,t}\big) ( D_{s_1} \otimes  \overline{D}_{j}) \Big)}^2_{2}
  \\&=  \sum_{j=1}^{T-1}  \E\bigsnorm{ \sum_{s_1 =1}^{j-1}\sum_{t=j+1}^{T}  \phi_{T,s_1,t}( D_{s_1}) \otimes \phi_{T,j,t}{ (D_{j}}) \Big)}^2_{2}
            \\&=  \sum_{j=1}^{T-1}  K^2_2  \sum_{s_1 =1}^{j-1} \bignorm{\sum_{t=j+1}^{T}  \phi_{T,s_1,t}}^2_{\infty} \| D_{s_1}\|_{\hi, 4}^2  \snorm{{\phi_{T,j,t}}}^2_\infty \| {D}_{j}\|^2_{\hi,4}
                \\&= K^2_2 \sum_{j=1}^{T-1}    \sum_{s_1 =1}^{j-1} \bignorm{\sum_{t=j+1}^{T}  \phi_{T,s_1,t} \widetilde{\otimes}{\phi_{T,j,t}}}^2_{\infty}  \| D_{0}\|^2_{\hi,4}  \|{D}_{0}\|^2_{\hi,4} 
                \\& = o(\snorm{\Phi_T}^4_{F})
\end{align*}
which follows from \autoref{as:phi}$\mathrm{(iv)}$. The same order applies to the second term. \textcolor{black}{For the third term we find
\begin{align*}
 \sum_{j=1}^{T-1} \bignorm{\sum_{t=j+1}^{T}  V_j}^2_{S_2,2}& = \sum_{j=1}^{T-1} \bignorm{ \sum_{s=j+1}^{T-1}\sum_{t=s+1}^{T} \big( \phi_{T,s,t} \widetilde{\otimes} \phi_{T,s,t}\big)P_0 (\dmi{s-j}{1} \otimes  {\dmi{s-j}{2}})}^2_{S_2,2}
  \\& \le  \sum_{j=1}^{T-1} \sum_{s=j+1}^{T-1}\Big( \bignorm{ \sum_{t=s+1}^{T} \big( \phi_{T,s,t} \widetilde{\otimes} \phi_{T,s,t}\big)}_{\infty} \bignorm{  P_0 (\dmi{s-j}{1} \otimes  {\dmi{s-j}{2}})}_{S_2,2}\Big)^2
\end{align*}
For convenience denote $A_s=\|\sum_{t=s+1}^{T} \big( \phi_{T,s,t} \widetilde{\otimes} \phi_{T,s,t}\big)\|_{\infty}$ and $\vartheta_{s-j}=\| P_0 (\dmi{s-j}{1} \otimes  {\dmi{s-j}{2}})\|_{S_2,2}$. Then we split the sum over $j$ in a sum with terms $1, \ldots, T-1-k_T$ and with terms $T-k_T,\ldots, T-1$ where $k_T=\lfloor T^{\alpha}\rfloor$, $\alpha \in (0,1)$. Additionally, we split the inner sum of the first. We then find via tedious calculations and the Cauchy-Schwarz inequality
\begin{align*}
 \sum_{j=1}^{T-1} & \Big( \sum_{s=j+1}^{T-1}A_s \vartheta_{s-j} \Big)^2  \le \sum_{j=T-k_T}^{T-1} \Big( \sum_{s=j+1}^{T-1}A_s \vartheta_{s-j} \Big)^2  + 2\sum_{j=1}^{T-1-k_T} \Big( \sum_{s=j+1}^{j+k_T-1}A_s \vartheta_{s-j} \Big)^2  +2\sum_{j=1}^{T-1-k_T} \Big( \sum_{s=j+k_T}^{T-1}A_s \vartheta_{s-j}\Big)^2
\\&  \le \sum_{j=T-k_T}^{T-1}  \Big(\sum_{s=j+1}^{T-1}A^2_s \Big) \Big(\sum_{s=j+1}^{T-1} \vartheta^2_{s-j} \Big)+ 2 \sum_{j=1}^{T-1-k_T}\Big( \sum_{s=j+1}^{j+k_T-1}A^2_s \Big) \Big( \sum_{s=j+1}^{j+k_T-1} \vartheta^2_{s-j} \Big)  +2\sum_{j=1}^{T-1-k_T} \Big( \sum_{s=j+k_T}^{T-1}A^2_s\Big) \Big( \sum_{s=j+k_T}^{T-1}\vartheta^2_{s-j}\Big)
\\&  \le \Big(\sum_{s=T-k_T+1}^{T-1}A^2_s \sum_{j=T-k_T}^{s-1} \Big)    \Big(\sum_{s=1}^{T-1} \vartheta^2_{s} \Big)
+2 C \Big(T k_T \varrho^4_T \Big)  +2\sum_{j=1}^{T-1-k_T} \Big( \sum_{s=j+k_T}^{T-1}A^2_s\Big) \Big( \sum_{s=j+k_T}^{\infty}\vartheta^2_{s-j}\Big)
\\&  \le \Big(\sum_{s=T-k_T+1}^{T-1}A^2_s k_T \Big)    \Big(\sum_{s=1}^{T-1} \vartheta^2_{s} \Big) 
+2 C \Big(T k_T \varrho^4_T \Big) +2\sum_{j=1}^{T-1-k_T} \Big( \sum_{s=j+k_T}^{T-1}A^2_s\Big) \Big( \sum_{s=j+k_T}^{\infty}\vartheta^2_{s-j}\Big)
\\&  \le C\Big(\sum_{s=j+1}^{T-1}\big(\sum_{t=s+1}^{T}\| ( \phi_{T,s,t} \widetilde{\otimes} \phi_{T,s,t})\|_{\infty} \big)^2 \Big) k_T  +2 C \Big( T k_T \varrho^4_T \Big)
+2\sum_{j=1}^{T-1-k_T}\sum_{s=j+k_T}^{T-1} \big(\sum_{t=s+1}^{T}\| \phi_{T,s,t}\|^2_{\infty} \big)^2\Big( \sum_{s=k_T}^{\infty}\vartheta^2_{s}\Big)
\\& \le C \Big(T \varrho^4_{T} k_T +  T  \varrho_T^4  \sum_{s=k_T}^{\infty}\vartheta^2_{s}\Big) =o(\snorm{\Phi_T}^4_F),
\end{align*}
for some generic bounded constant $C$, which follows from \autoref{as:depstruc} and \autoref{as:phi}$\mathrm{(i)}$ for any $k_T \to \infty$ such that $\frac{k_T}{T} \to 0$ as $T \to \infty$.}
\end{proof}

\begin{proof}[Proof of \autoref{lem:secM}]
By \autoref{lem:Burkh} and Jensen's inequality, we obtain for fixed $m$,
\begin{align*}
&\bignorm{\sum_{t=2}^{T} \dm{t}  \otimes \big(\sum_{s=t-4m+1 \vee 1}^{t-1} \phi^{(\lambda)}_{T,s-t} \dm{s}\big)}^2_{S_2,2}
  \le  K^2_2  \|\dm{0}\|^2_{\hi,4} \sum_{t=2}^{T} \sum_{s=t-4m+1 \vee 1}^{t-1} \snorm{\phi^{(\lambda)}_{T,s-t}}_{\infty}^2 \norm{ \dm{s} }^2_{\hi,4} 
  \\& =2  K^2_4  \|\dm{0}\|^4_{\hi,4} \Big( \sum_{t=2}^{4m} \sum_{s=1}^{t-1 }\snorm{\phi^{(\lambda)}_{T,s-t}}_{\infty}^2  + \sum_{t=4m+1}^{T} \sum_{s=t-4m+1 }^{t-1}\snorm{ \phi^{(\lambda)}_{T,s-t}}_{\infty}^2 \Big)
 \\& \le  2  K^2_4  \|\dm{0}\|^4_{\hi,4} O \Big( (4m) \varrho_{4m}^2 + T  4m \max_{t} \snorm{A_{T,t}}_{\infty}^2\Big) =o(\snorm{\Phi_T}^2_{F}) + To(\varrho_{T}^2)  = o(\snorm{\Phi_T}^2_{F}). 
\end{align*}
\end{proof}

\begin{proof}[Proof of \autoref{lem:inpsZt}]
Denote $\lambda = \lambda_1 \pm \lambda_2$ and recall that $N^{(\lambda)}_{m,t}=\sum_{s=1}^{t-4m} \phi^{(\lambda)}_{T,s-t} \dm{s}$. Since the increments of $\{\dm{s}\}$ are uncorrelated, we have
\begin{align*}
\sum_{t=1+4m}^{T}\Big|\E \inprod{N^{(\lambda_1)}_{m,t}}{u} \inprod{N^{(\lambda_2)}_{m,t}}{v}\Big|&=\sum_{t=1+4m}^{T}\Big|\sum_{s=1}^{t-4m}e^{\im \lambda(s-t)} \E \inprod{A_{T,s-t}(\dmi{s}{1})}{u}  \inprod{A_{T,s-t}( \dmi{s}{2})}{v}\Big|
\\& =\sum_{t=1+4m}^{T}\Big|\sum_{s=1}^{t-4m}e^{\im \lambda (s-t)} \E \biginprod{A_{T,s-t}(\dmi{s}{1}) \otimes \overline{A_{T,s-t}(\dmi{s}{2})}}{u \otimes \overline{v}}_S\Big|
\\&=\sum_{t=1+4m}^{T}\Big| e^{-\im \lambda t}\sum_{s=1}^{t-4m}e^{\im \lambda s} \biginprod{ (A_{T,s-t} \widetilde{\otimes} \overline{A}_{T,s-t})  \E(\dmi{0}{1} \otimes \overline{\dmi{0}{2}})}{u \otimes \overline{v}}_S\Big|
\\& \le \sum_{t=1+4m}^{T}\Big|\sum_{s=1}^{t-4m}e^{\im \lambda s} \biginprod{ A_{T,s-t} \E(\dmi{0}{1} \otimes \overline{\dmi{0}{2}})\overline{A}^{\dagger}_{T,s-t}  }{u \otimes \overline{v}}_S\Big|.
\end{align*}
To ease notation, set $W_{s-t} =\biginprod{ (A_{T,s-t} \widetilde{\otimes} \overline{A}_{T,s-t})  \E(\dmi{0}{1} \otimes \overline{\dmi{0}{2}})}{u \otimes \overline{v}}_S$ and write $B_j = \sum_{k=1}^{j} e^{\im \lambda j}$. Summation by parts, and Holder's inequality for operators yield
\begin{align*}
& \sum_{t=1+4m}^{T}\Big|\sum_{s=1}^{t-4m}  W_{s-t} (B_s-B_{s-1})\Big|=
 \sum_{t=1+4m}^{T} \Big|W_{4m} B_{t-4m} +\sum_{t=1+4m}^{T} \sum_{s=1}^{t-4m-1} B_s\big(W_{s-t}-W_{s+1-t} \big)\Big|
 \\&\le \sum_{t=1+4m}^{T} |W_{4m} B_{t-4m}| +\sum_{t=1+4m}^{T} \sum_{s=1}^{t-4m-1} |B_s| \bigsnorm{ \big(A_{T,s-t} \widetilde{\otimes}( \overline{A}_{T,s-t}- \overline{A}_{T,s-t+1}) \E(\dmi{0}{1}) \otimes \overline{\dmi{0}{2}})}_2 \snorm{u \otimes \overline{v}}_2
 \\&+\sum_{t=1+4m}^{T}  \sum_{s=1}^{t-4m-1} |B_s| \bigsnorm{\Big((A_{T,s-t} - {A}_{T,s-t+1} )\widetilde{\otimes} \overline{A}_{T,s-t+1}\big)   \E(\dmi{0}{1}) \otimes \overline{\dmi{0}{2}})}_2 \snorm{u \otimes \overline{v}}_2
  \\&\le \sum_{t=1+4m}^{T} |W_{4m} B_{t-4m}| +\sum_{t=1+4m}^{T} \sum_{s=1}^{t-4m-1} |B_s| \snorm{ A_{T,s-t}}_{\infty} \snorm{\overline{A}_{T,s-t}- \overline{A}_{T,s-t+1}}_{\infty} \snorm{\E(\dmi{0}{1}) \otimes \overline{\dmi{0}{2}})}_2  \norm{u}_H \norm{\overline{v}}_H
 \\&+\sum_{t=1+4m}^{T}  \sum_{s=1}^{t-4m-1} |B_s| \snorm{A_{T,s-t} - {A}_{T,s-t+1}}_{\infty} \snorm{\overline{A}_{T,s-t+1}}_{\infty} \snorm{\E(\dmi{0}{1}) \otimes \overline{\dmi{0}{2}})}_2 \norm{u}_H \norm{\overline{v}}_H.
 \end{align*}
Then, using Jensen's inequality and the Cauchy-Schwarz inequality twice, we obtain
\begin{align*}
 &\le C_1  \norm{\dmi{0}{1}}_{\hi,2}  \norm{\overline{\dmi{0}{2}}}_{\hi,2} | 1/\sin (\lambda/2)|  \Big( \sum_{t=1+4m}^{T} \snorm{A_{4m}}^2_{\infty} + \sum_{t=1+4m}^{T} \sum_{s=1}^{t-4m-1} \snorm{ A_{T,s-t}}_{\infty} \snorm{\overline{A}_{T,s-t}- \overline{A}_{T,s-t+1}}_{\infty}
\\&  + \sum_{t=1+4m}^{T} \sum_{s=1}^{t-4m-1}\snorm{A_{T,s-t} - {A}_{T,s-t+1}}_{\infty} \snorm{\overline{A}_{T,s-t+1}}_{\infty} \Big)
\\ &\le C_1 C_2 | 1/\sin (\lambda/2)|  \Big( \sum_{t=1+4m}^{T} \snorm{A_{4m}}^2_{\infty} + \sum_{t=1+4m}^{T} \Big(\sum_{s=1}^{t-4m-1} \snorm{ A_{T,s-t}}^2_{\infty} \sum_{s=1}^{t-4m-1}\snorm{\overline{A}_{T,s-t}- \overline{A}_{T,s-t+1}}^2_{\infty}\Big)^{1/2}\\&  + \sum_{t=1+4m}^{T}\big( \sum_{s=1}^{t-4m-1}\snorm{A_{T,s-t} - {A}_{T,s-t+1}}^2_{\infty} \sum_{s=1}^{t-4m-1} \snorm{\overline{A}_{T,s-t+1}}^2_{\infty} \Big)^{1/2}\Big)
\\& \le O(| 1/\sin (\lambda/2)|) \Big( O(T)+O(T) o(\varrho_T) O(\varrho_T)+O(T)  O(\varrho_T) o(\varrho_T)\Big)= o(\snorm{\Phi_T}^2_{F}),
\end{align*}
where we used that $\max_{1\le t\le T}|B_t| \le |1/(\sin(\lambda/2))|$. 
\end{proof}

\section{Operator approximations} \label{sec:A_opAprox}

\begin{proof}[Proof of \autoref{lem:bilaprox}]
We can decompose the quadratic form 
\begin{align*}
 \hat{\mathcal{Q}}^{\lambda}_{T}&  = \ldx{}+\ldx{\dagger} +\sum_{1 \le t \le T} \phi_{T,t,t}  (X_t \otimes X_t) 
\end{align*}
Set $\hat{\mathcal{C}}_T=\sum_{t=1}^{T} X_t \otimes  X_t$. Then, using linearity of the operator $ \Phi_{T,t,t}$
\begin{align*}
\bignorm{\hat{\mathcal{Q}}^{\lambda}_{T}-\E\hat{\mathcal{Q}}^{\lambda}}_{S_2,2}&\le \bignorm{\ldx{}+\ldx{\dagger}  -\E\ldx{}-\E \ldx{\dagger} }_{S_2,2}+\bignorm{ \Phi_{T,t,t} ( \hat{\mathcal{C}}_T-\E \hat{\mathcal{C}}_T )}_{S_2,2}.
\end{align*}
For the last term, Holder's inequality for operators and \autoref{eq:nullpart} yield
\begin{align*}
 \frac{1}{\snorm{\Phi_T}_{F}}\bignorm{\Phi_{T,t,t} ( \hat{\mathcal{C}}_T-\E \hat{\mathcal{C}}_T )}_{S_2,2}\le  \frac{1}{\snorm{\Phi_T}_{F}}\norm{\Phi_{T,t,t}}_{\infty} \bignorm{ \hat{\mathcal{C}}_T-\E \hat{\mathcal{C}}_T}_{S_2,2}=o(\varrho_T)O\big(\frac{\sqrt{T}}{\snorm{\Phi_T}_{F}} \big) =o(1),
\end{align*}
For the first term, we find using \autoref{lem:mdapprox} and \autoref{lem:marappr}, respectively
\begin{align*}
\frac{1}{\snorm{\Phi_T}_{F} }\bignorm{\ldx{}-\E\ldx{} -\ldm{}}_{S_2,2} &\le \frac{1}{\snorm{\Phi_T}_{F}  }\bignorm{\ldx{}-\E\ldx{}-(\ldd{}-\E \ldd{})}_{S_2,2} 
\\& +\frac{1}{\snorm{\Phi_T}_{F}  }\bignorm{\ldd{}-\E \ldd{}-\ldm{}}_{S_2,2}
\\&\le  K_4 \Upsilon_{4,m}\sum_{t=0}^{\infty}\nu_{\hi,4}(X_t)
\\& + \norm{X_0}_{\hi,4}^2 \frac{m^2 \sqrt{T}}{\snorm{\Phi_T}_{F} } \big( \max_{t}\snorm{A_{T,t}}_{\infty}^2+m\sum_{s=1}^{T}\snorm{A_{T,s-t}-A_{T,s-(t+1)}}_{\infty}^2 \big)^{1/2}. \end{align*}
The result therefore follows from \autoref{as:phi}.
\end{proof}

\begin{lemma}\label{eq:nullpart} Let $X_t$ satisfy \autoref{as:depstruc} with $p=4$. Then 
\[\bignorm{\sum_{1 \le t \le T} (X_t \otimes X_t)- T\E(X_0 \otimes X_0)}_{S_2,2} = O(\sqrt{T} \sum_{j=0}^{\infty}\nu_{\hi, 4}(X_j) )=O(\sqrt{T}). \]
\end{lemma}

\begin{proof}[Proof of \autoref{eq:nullpart}]
By stationarity, ergodicity and orthogonality of the projection operators
\begin{align*}
\bignorm{\sum_{1 \le t \le T} (X_t \otimes X_t)- T\E(X_0 \otimes X_0)}^2_{S_2,2} &= \bignorm{ \sum_{j=-\infty}^{T}\sum_{t=1}^{T}P_j ( X_{t} \otimes X_t)}^2_{S_2,2} \le \sum_{j=-\infty}^{T}\bignorm{  \sum_{t=1}^{T} P_j ( X_{t} \otimes X_t)}^2_{S_2,2}.
\end{align*}
Minkowski's inequality, the Cauchy-Schwarz inequality and stationarity imply
\begin{align*}
\bignorm{  \sum_{t=1}^{T} P_j ( X_{t} \otimes X_t)}_{S_2,2} 
& = \sum_{j=-\infty}^{T} \sum_{t=1}^{T}\bignorm{P_0 ( X_{t-j} \otimes X_{t-j})}_{S_2,2} 
\\& 
\le  \sum_{t=1}^{T}\bignorm{X_{t-j} \otimes (X_{t-j}-X_{t-j,\{0\}}) }_{S_2,2} +
 \bignorm{(X_{t-j}-X_{t-j,\{0\}})  \otimes X_{t-j,\{0\}} }^2_{S_2,2} 
\\& \le  \sum_{t=1}^{T} (\E \|X_{t-j}\|^4_H)^{1/4}( \E\| (X_{t-j}-X_{t-j,\{0\}}) \|^4_{H})^{1/4} + (\E \|X_{t-j,\{0\}}\|^4_H)^{1/4}( \E\| (X_{t-j}-X_{t-j,\{0\}}) \|^4_{H})^{1/4}
\\&
\le 2 \|X_0\|_{\hi,4} \sum_{t=1}^{T} {\nu}_{\hi,4}(X_{t-j}).
\end{align*}
Consequently, 
\[\sum_{j=-\infty}^{T}\bignorm{  \sum_{t=1}^{T} P_j ( X_{t} \otimes X_t)}^2_{S_2,2} \le 
4 \|X_0\|^2_{\hi,4} \sum_{j=-\infty}^{T} \Big(\sum_{t=1}^{T} {\nu}_{\hi,4}(X_{t-j})\big)^2
\le 4 \|X_0\|^2_{\hi,4} T  (\sum_{j=0}^{\infty} {\nu}_{\hi,4}(X_j))^2.\]

The result follows by taking the square root.
\end{proof}

\begin{lemma}[M-dependence approximation]\label{lem:mdapprox}
Suppose \eqref{eq:depstruc2} with $2p$ is satisfied for some $p \ge 2$. Then 
\begin{align*}
\frac{\|\ldx{} - \E \ldx{} - (\ldd{}- \E \ldd{})\|_{S_2,p}}{\sqrt{T}\snorm{\phi_T}_{\ell_{2}}\sum_{t=0}^{\infty}\nu_{\hi,2p}(X_t)} \le K_p \Upsilon_{2p,m}
\end{align*}
where $\Upsilon_{2p,m} = 2\sum_{t=0}^{\infty}\min(\nu_{\hi, 2p}(X_t),\boldsymbol{\Delta}^{1/2}_{2p,2,m+1})$ and
\begin{align*}
\ldx{}:=\sum_{s=2}^T  \sum_{t=1}^{s-1} \Phi^{(\lambda)}_{T,s,t}( X_s \otimes X_t)
\quad \text{and} \quad \ldd{}:=\sum_{s=2}^T  \sum_{t=1}^{s-1} \Phi^{(\lambda)}_{T,s,t}( X^{(m)}_s \otimes X^{(m)}_t).
\end{align*}
\end{lemma}
\begin{proof}[Proof of \autoref{lem:mdapprox}]
Let $N^{(m)}_{T,s} =  \sum_{t=1}^{s-1} \phi^{(\lambda)}_{T,s-t} (X^{(m)}_{t})$ and $N_{T,s} =  \sum_{t=1}^{s-1} \phi^{(\lambda)}_{T,s-t} (X_{t})$ and observe these are $\G_s$-measurable. 
By orthogonality of the projections and Minkowski's inequality
\begin{align*}
\|\ldx{} - \E \ldx{} - (\ldd{} - \E \ldd{})\|^2_{S_2,p}& \le  2
  \sum_{j=-\infty}^{T} \|P_j(\ldx{}  - \tilde{\mathcal{V}}^{(m),\lambda}_T)\|^2_{S_2,p} +
\|P_j(\tilde{\mathcal{V}}^{(m),\lambda}_T - \ldd{})\|^2_{S_2,p},
\end{align*}
where $\tilde{\mathcal{V}}^{(m),\lambda}_T =\sum_{s=2}^T X_s \otimes N^{(m)}_{T,s}$. We shall focus on bounding the first term as the second is similar and has the same upper bound. A similar trick as in \autoref{lem:lineq} shows
\[
\E[\ldx{}-\tilde{\mathcal{V}}^{(m),\lambda}_T| \G_{{j-1}}] = \E[\ldx{}-\tilde{\mathcal{V}}^{(m),\lambda}_T| \G_{{j},\{j\}}] =  \E[\mathcal{V}_{T,\{j\}}-\tilde{\mathcal{V}}^{(m)}_{T,\{j\}}| \G_j].
\]
By the contraction property of the conditional expectation
\begin{align*}
\bignorm{P_j(\ldx{}  - \tilde{\mathcal{V}}^{(m),\lambda}_T)}_{S_2,p} &
\le \bignorm{\ldx{}-\tilde{\mathcal{V}}^{(m),\lambda}_T-(\mathcal{V}_{T,\{j\}}-\tilde{\mathcal{V}}^{(m)}_{T,\{j\}})}_{S_2,p}
\\&  \le \bignorm{\sum_{s=2}^{T} (X_s -X_{s,\{j\}}) \otimes ( N_{T,s}- N^{(m)}_{T,s} )}_{S_2,p}
\\&+ \bignorm{\sum_{s=2}^{T} X_{s,\{j\}} \otimes ( N_{T,s}- N^{(m)}_{T,s} -N_{T,s,\{j\}}+N^{(m)}_{T,s,\{j\}} ) }_{S_2,p}=: J_1 +J_2,
\end{align*}
where we added and subtracted $X_{s,\{j\}} \otimes ( N_{T,s}- N^{(m)}_{T,s} )$ and applied Minkowski's inequality. From \autoref{lem:lineq}(iii)
\[
\| N_{T,s}- N^{(m)}_{T,s}\|_{\hi,2p} \le  (K^q_{p}\snorm{\phi_T}^q_{\ell_q}\boldsymbol{\Delta}^q_{2p,1,m+1})^{1/q}
\]
and $\|X_s -X_{s,\{j\}}\|_{\hi,2p} \le \nu_{\hi,2p}(X_{s-j})$. Then, by the Cauchy-schwarz inequality and recalling that  $q=\min(2,2p)$ 
\begin{align*}
\sum_{j=-\infty}^{T} J^2_1 &
 \le \sum_{j=-\infty}^{T}  \Big(\sum_{s=2}^{T} \nu_{\hi,2p}(X_{s-j})\Big)^2 \Big(\big(K^2_{p}\snorm{\phi_T}^2_{\ell_2}\boldsymbol{\Delta}^2_{2p,1,m+1})^{1/2} \big)^2
\\& \le K^2_{p}\snorm{\phi_T}^2_{\ell_2}\boldsymbol{\Delta}^2_{2p,1,m+1}
 \sum_{s=2}^{T} \sum_{j=-\infty}^{T} \nu_{\hi,2p}(X_{s-j})   \sum_{s=2}^{T} \nu_{\hi,2p}(X_{s-j})  \\&  
 \le T  K^2_{p}\snorm{\phi_T}^2_{\ell_2}\boldsymbol{\Delta}^2_{2p,1,m+1}\boldsymbol{\Delta}_{2p,1,0}^2, 
\end{align*}
where we used that $ \sum_{s=2}^{T}  \nu_{\hi,2p}(X_{s-j}) \le \boldsymbol{\Delta}_{2p,1,0} $. Secondly, from \autoref{lem:lineq}, \eqref{eq:Xcoup} and Minkowski's inequality, we obtain
\begin{align*}
& \|X_s- X^{(m)}_s+X^{(m)}_{s,\{j\}} -X_{s,\{j\}}\|_{\hi,2p} 
\\& \le \min(\|X_s- X^{(m)}_s\|_{\hi,2p}+\|X^{(m)}_{s,\{j\}} -X_{s,\{j\}}\|_{\hi,2p},\|X_s-X_{s,\{j\}} \|_{\hi,2p}
 +\|X^{(m)}_{s,\{j\}} -X^{(m)}_s\|_{\hi,2p})
 \\&\le 2 \min\big( (\sum_{j=m+1}^{\infty} \|D_{t,j}\|^2_{\hi,2p})^{1/2}, \nu_{\hi,2p}(X_{s-j})\big).
\end{align*}
Hence, changing the order of summation and using property \eqref{eq:Phiprop} of ${\Phi}^{(\lambda)}_{T,t,s}$ yields
\begin{align*}
\sum_{j=-\infty}^{T} J^2_2 \le & 
 \sum_{j=-\infty}^{T} \bignorm{\sum_{s=2}^{T}\sum_{t=1}^{s-1} X_{s,\{j\}} \otimes\phi_{T,t-s}^{(\lambda)} (X_t- X^{(m)}_t+X^{(m)}_{t,\{j\}} -X_{t,\{j\}}) }^2_{S_2,p}
\\& =\sum_{j=-\infty}^{T}(\sum_{t=1}^{T-1} \bignorm{\sum_{s=t+1}^{T}\phi_{T,s-t}^{(-\lambda)} (X_{s,\{j\}}) \otimes (X_t- X^{(m)}_t+X^{(m)}_{t,\{j\}} -X_{t,\{j\}}) }_{S_2,p}\Big)^2
\\& \le K^2_{p}\snorm{\phi_T}^2_{\ell_2}  \boldsymbol{\Delta}_{2p,1,0}   2 \sum_{j=-\infty}^{T}  \boldsymbol{\Delta}_{2p,1,0}\Big( \sum_{t=1}^{T-1}\min\big(  \boldsymbol{\Delta}^{1/2}_{2p,2,m+1},  \nu_{\hi,2p}(X_{t-j}) \big)\Big)^2
\\&  \le K^2_{p}\snorm{\phi_T}^2_{\ell_2} \boldsymbol{\Delta}_{2p,1,0}   2 T   \boldsymbol{\Delta}_{2p,1,0} \Upsilon^2_{2p,m}.
\end{align*}
Noting that $\Upsilon_{2p,m} \ge \boldsymbol{\Delta}_{2p,1,m+1}$, we obtain
\begin{align*}
  \sum_{j=-\infty}^{T} &\|P_j(\ldx{}  - \tilde{\mathcal{V}}^{(m),\lambda}_T)\|^2_{S_2,p} +
\|P_j(\tilde{\mathcal{V}}^{(m),\lambda}_T - \ldd{})\|^2_{S_2,p}
\\& 
\le 2 K^2_{p}\snorm{\phi_T}^2_{\ell_2}   T  \boldsymbol{\Delta}^2_{2p,1,0}  \Upsilon^2_{2p,m}.
\end{align*}

\end{proof}

\begin{lemma}[martingale approximation to the $m$-dependent process]\label{lem:marappr}
Let $\ldm{}$ as defined in \eqref{eq:ldm} and $\ldd{}$ as in \autoref{lem:mdapprox}. Under \autoref{as:depstruc} with $p=4$, we have
\begin{align*}
\frac{\|\ldd{}- \E \ldd{}-\ldm{}\|_{S_2,2}}{m^{3/2} \sqrt{T} \|X_0\|^2_{\hi,4}} \le C\big( \max_{t}\snorm{A_{T,t}}_{\infty}+\big(m\sum_{s=1}^{T} \snorm{A_{T,s}-A_{T,s-1)}}_{\infty}^2 \big)^{1/2}\Big).
\end{align*}
\end{lemma}
\begin{proof}[Proof of \autoref{lem:marappr}]
By construction $\dm{k}\in \op^{p}_H$ defines an $m$-dependent martingale difference and therefore we can write $\dm{k} = \sum_{t=0}^{m} P_{k}(X^{(m)}_{t+k})e^{-\im t \lambda}$ since the terms $t>m$ are zero. We decompose the difference $\ldd{}-\ldm{}$ as follows
\begin{align*}
\sum_{t=2}^{T} & X^{(m)}_t \otimes \sum_{s=1}^{t-1} \phi^{(\lambda)}_{T,s-t} (X^{(m)}_s -\dm{s}) + \sum_{t=2}^{T} (X^{(m)}_t-\dm{t}) \otimes \sum_{s=1}^{t-1} \phi^{(\lambda)}_{T,s-t} \dm{s}  
\\& =\sum_{t=2}^{T} X^{(m)}_t \otimes \Big( \sum_{s=1}^{t-4m} \phi^{(\lambda)}_{T,s-t} (X^{(m)}_s -\dm{s}) + \sum_{s=t-4m+1}^{t-1} \phi^{(\lambda)}_{T,s-t} (X^{(m)}_s -\dm{s}) \Big) 
\\&+ \sum_{t=2}^{T} (X^{(m)}_t-\dm{t}) \otimes \sum_{s=1}^{t-1} \phi^{(\lambda)}_{T,s-t} \dm{s}:=  \sum_{t} M^{\star}_t+ Y_{t}+Z_{t}.  \tageq \label{eq:decMarap}
\end{align*}
Note that $\ldd{}-\E\ldd{}-\ldm{} = \sum_{t}M^{\star}_t +Y_t -\E Y_t +Z_t -\E Z_t$.
We treat the above terms separately. Firstly, we consider $M^{\star}_t :=X^{(m)}_t \otimes \sum_{s=1}^{t-4m} \phi^{(\lambda)}_{T,s-t} (X^{(m)}_s -\dm{s})$. The process $\big\{M^{\star}_{t+4mk}\big\}_{k \in \mathbb{N}}$ is then a martingale difference sequence in $\op^2_{S_2}$.
Let $W_k= \sum_{t=0}^{m} \E[ X^{(m)}_{t+k} |\G_k] e^{-\im t \lambda}$ and observe that 
\begin{align*}
X^{(m)}_k & = \sum_{t=0}^{m} \E[ X^{(m)}_{t+k} |\G_k] e^{-\im t \lambda} -  \sum_{t=1}^{m} \E[ X^{(m)}_{t+k} |\G_k] e^{-\im t \lambda}
\\&=  \sum_{t=0}^{m} \E[ X^{(m)}_{t+k} |\G_k] e^{-\im t \lambda} -  \sum_{t=0}^{m-1} \E[ X^{(m)}_{t+k+1} |\G_k] e^{-\im (t+1) \lambda} =W_k - \E[W_{k+1}| \G_k]e^{-\im \lambda}
\end{align*}
and that $\dm{k} = W_k - \E[W_{k}| \G_{k-1}]$. Therefore, 
\begin{align*}
\bignorm{\sum_{s=1}^{t-4m}  \phi^{(\lambda)}_{T,s-t} (X^{(m)}_s- \dm{s})}_{\hi,2} &= \bignorm{\sum_{s=1}^{t-4m}  A_{T,s-t}e^{-\im (s-t)\lambda} (\E[W_{s}| \G_{s-1}] - \E[W_{s+1}| \G_s]e^{-\im \lambda})}_{\hi,2} 
\\&  \bignorm{\sum_{s=1}^{t-4m}  A_{T,s-t} \Big(e^{-\im (s-t)\lambda} \E[W_{s}| \G_{s-1}] -e^{\im (t-(s+1))\lambda} \E[W_{s+1}| \G_s]\Big)}_{\hi,2}. 
\end{align*}
Set $V_s = e^{\im (t-s)\lambda} \E[W_{s}| \G_{s-1}] $. Summation by parts, Holder's inequality for operators and \autoref{lem:Burkh} together yield
\begin{align*}
\bignorm{\sum_{s=1}^{t-4m}  A_{T,s-t} \Big(V_s-V_{s+1}\Big)}_{\hi,2}
& \le \max_{t}\snorm{A_{T,t}}_{\infty}\|V_{t-4m}\|_{\hi,2} + \bignorm{\sum_{s=1}^{t-4m} (A_{T,s-t}-A_{T,s-t-1}) (\sum_{l=1}^{m}P_{s-l}V_s)}_{\hi,2}
\\& \le \max_{t}\snorm{A_{T,t}}_{\infty}\|V_{t-4m}\|_{\hi,2} + \sqrt{\bignorm{\sum_{l=1}^{m}\sum_{s=1}^{t-4m} (A_{T,s-t}-A_{T,s-t-1}) (P_{s-l}V_s)}^2_{\hi,2}}
\\& \le  2m \|X_0\|_{\hi,2}  \max_{t}\snorm{A_{T,t}}_{\infty}+\big(\sum_{s=1}^{t-4m} \snorm{A_{T,s-t}-A_{T,s-t-1}}_{\infty}^2 \big)^{1/2} C m^{3/2}\|X_0\|_{\hi,2}. 
\end{align*}
Hence, for some finite constant $C$,
\[\| M^{\star}_t\|_{\hi,2} \le C m \|X_0\|^2_{\hi,2}\big(  \max_{t}\snorm{A_{T,t}}_{\infty}+(\sum_{s=1}^{t-4m} \snorm{A_{T,s-t}-A_{T,s-t-1}}_{\infty}^2 m )^{1/2}\big), 
\]
where we used the contraction property and stationarity to find $\|V_{t-4m}\|_{\hi,2} \le \|W_k\|_{\hi,2} \le 2m\|X_0\|_{\hi,2}$, and where we used that $( (\sum_{l=1}^{m}\|P_{-l}V_0\|_{\hi,2})^2)^{1/2} \le ( \sum_{l=1}^{m}\|P_{-l}V_0\|^2_{\hi,2})^{1/2} \le \sqrt{m} 2m \|X_0\|_{\hi,2} $ since the $P_j(\cdot)$ form martingale differences. Consequently, a martingale decomposition of the sum gives
\begin{align*}
\bignorm{\sum_{t=1}^{T} M^{\star}_t }_{S_2,2}&\le \sum_{t=1}^{4m} \bignorm{\sum_{s=0}^{\lfloor (T-t)/(4m) \rfloor} M^{\star}_{t+4m s} }_{S_2,2} 
\\& \le 4 m^{3/2} T^{1/2} C \|X_0\|^2_{\hi,2} \big(   \max_{t}\snorm{A_{T,t}}_{\infty}+(\sum_{s=1}^{T} \snorm{A_{T,s}-A_{T,s-1}}_{\infty}^2 m )^{1/2} \big).
\end{align*}
For $Y_{t}$, we note that $\sum_{s=t-4m+1}^{t-1} \phi^{(\lambda)}_{T,s-t} (X^{(m)}_s -\dm{s})$ is $\G_{t-1}$ measurable and that $Y_t$ is $5m$-dependent. Therefore, via Minkowski's inequality, the Cauchy-Schwarz inequality and a similar decomposition as above shows that
\begin{align*}
\bignorm{\sum_t Y_{t} - \E Y_t }_{S_2,2}=C  \sqrt{T}  m^{3/2} \|X_0\|^2_{\hi,4} (  \max_{t}\snorm{ A_{T,t}}_{\infty}+\sum_{s=1}^{T} \snorm{A_{T,s}-A_{T,s-1}}_{\infty}^2 m^{1/2})\end{align*}
and similarly for $\|\sum_t Z_{t} - \E Z_t \|_{S_2,2}$.
\end{proof}
\section{Proofs of \autoref{sec:sec4}} \label{sec:ApF}
\begin{proof}[Proof of \autoref{thm:consF}]
We first prove part (i) and consider the following bias variance decomposition
\begin{align*}
\snorm{\hat{\F}_T^{(\lambda)}-\F^{(\lambda)}}_{S_2,p}&  \le  \snorm{\hat{\F}_T^{(\lambda)}-\E\hat{\F}_T^{(\lambda)}}_{S_2,p}+\snorm{\,\E\hat{\F}_T^{(\lambda)}-\F^{(\lambda)}}_{S_2,p}.
\end{align*}
We start with the first term. We decompose the error as
\begin{align}
\hat{\F}^{(\lambda)}_{T}- \E \hat{\F}^{(\lambda)}_{T} &=\frac{1}{T} \Bigg(\sum_{s=2}^T X_s \otimes N^{(\lambda)}_{T,s} - \E \sum_{s=2}^T X_s \otimes N^{(\lambda)}_{T,s} +(\sum_{s=2}^T X_s \otimes N^{(\lambda)}_{T,s})^{\dagger}- \E(\sum_{s=2}^T X_s \otimes N^{(\lambda)}_{T,s})^{\dagger}\Bigg) \label{eq:Fvar1}
\\& +\frac{1}{T}\bigg(\sum_{1 \le t \le T}A_{T,0}  (X_t \otimes X_t)  -\E \sum_{1 \le t \le T}A_{T,0} (X_t \otimes X_t)  \Bigg) \label{eq:Fvar2},
\end{align}
where $N^{(\lambda)}_{T,s} =  \sum_{t=1}^{s-1} \phi^{(\lambda)}_{T,t,s} X_t$. We first derive the order of  \eqref{eq:Fvar1} in $\op^2_{S_2(H)}$. Using ergodicity and othogonality of the projection operator, we can write
\begin{align*}
&\bignorm{ \sum_{t=2}^T X_t \otimes N^{(\lambda)}_{T,t-1} - \E \sum_{t=2}^T X_t \otimes N^{(\lambda)}_{T,t-1}}^2_{S_2,p}
 = \sum_{j=-\infty}^{T}\bignorm{P_j \Big(\sum_{t=2}^T\sum_{s=1}^{t-1} (X_t \otimes  \phi^{(\lambda)}_{T,s,t} X_s ) \Big)}^2_{S_2,p}.
\end{align*}
The contraction property of the conditional expectation and the Cauchy-Schwarz inequality imply
\begin{align*}
&\bignorm{P_j \Big(\sum_{t=2}^T\sum_{s=1}^{t-1}\phi^{(\lambda)}_{T,t,s} (X_t \otimes   X_s ) \Big)}_{S_2,p} 
\le \bignorm{\sum_{t=2}^T\sum_{s=1}^{t-1}\phi^{(\lambda)}_{T,t,s} \Big( (X_t \otimes   X_s ) -(X_{t,\{j\}} \otimes   X_{s,\{j\}} ) \Big)}_{S_2,p} 
\\& =
 \bignorm{\sum_{t=2}^T\sum_{s=1}^{t-1}\phi^{(\lambda)}_{T,t,s} \Big( X_t \otimes (X_s - X_{s,\{j\}}) + (X_t - X_{t,\{j\}}) \otimes   X_{s,\{j\}}  \Big)}_{S_2,p} 
 \\ &
 \le  \bignorm{\sum_{t=2}^T\sum_{s=1}^{t-1}\phi^{(\lambda)}_{T,t,s} \Big( X_t \otimes (X_s - X_{s,\{j\}})\Big)}_{S_2,p}+ \bignorm{\sum_{t=2}^T\sum_{s=1}^{t-1}\phi^{(\lambda)}_{T,t,s} \Big(  (X_t - X_{t,\{j\}}) \otimes   X_{s,\{j\}}  \Big)}_{S_2,p} 
 \\& 
  \le  \sum_{s=1}^{T-1}  \bignorm{ \sum_{t=s+1}^{T} \phi^{(\lambda)}_{T,t,s} X_t \otimes (X_s - X_{s,\{j\}})}_{S_2,p} + \sum_{t=2}^T\bignorm{\sum_{s=1}^{t-1}\phi^{(\lambda)}_{T,t,s} \Big(  (X_t - X_{t,\{j\}}) \otimes   X_{s,\{j\}}  \Big)}_{S_2,p}
  \\&  \le \sum_{s=1}^{T-1} \bigg( \bignorm{ \sum_{t=s+1}^{T}  \phi^{(\lambda)}_{T,t,s} X_t}^2_{\hi,2p} \bignorm{X_s - X_{s,\{j\}}}^2_{\hi,2p}\bigg)^{1/2} +\sum_{t=2}^T\bignorm{ \Big(  (X_t - X_{t,\{j\}}) \otimes \sum_{s=1}^{t-1}\phi^{(\lambda)}_{T,s,t}  X_{s,\{j\}}  \Big)}_{S_2,p}
    \\&  \le \sum_{s=1}^{T-1} \bigg( \bignorm{ \sum_{t=s+1}^{T} \phi^{(\lambda)}_{T,t,s} X_t}^2_{\hi,2p} \bignorm{X_s - X_{s,\{j\}}}^2_{\hi,2p} \bigg)^{1/2}+\sum_{t=2}^T \bigg(\norm{X_t - X_{t,\{j\}}}^2_{\hi,2p} \bignorm{ \sum_{s=1}^{t-1}\phi^{(\lambda)}_{T,s,t}  X_{s,\{j\}}}^2_{\hi,2p}\big)^{1/2}.
\end{align*}
Hence, from \autoref{lem:lineq} we obtain 
\begin{align*}
&\frac{1}{T^2}\sum_{j=-\infty}^{T} \bignorm{P_j \Big(\sum_{t=2}^T\sum_{s=1}^{t-1}\phi^{(\lambda)}_{T,t,s} (X_t \otimes   X_s ) \Big)}^2_{S_2,2} 
\\& \le \frac{1}{T^2} \sum_{j=-\infty}^{T} \Big(K_{2p}  \max_{1\le s \le T-1} \snorm{\phi_{s,T}}_{\ell_p} \boldsymbol{\Delta}_{2p,1,0}\sum_{s=1}^{T-1} \nu_{\hi,2p}(X_{s-j})   + \sum_{t=2}^{T} \nu_{\hi,2p}(X_{t-j})  K_{2p}  \max_{2\le t \le T} \snorm{\phi_{t,T}}_{\ell_p}  \boldsymbol{\Delta}_{2p,1,0} \Big)^2
\\& \le \frac{1}{T^2} K^2_{2p} ( \max_{1\le t \le T} \snorm{\phi_{t,T}}_{\ell_p})^2 \boldsymbol{\Delta}^2_{2p,1,0} \sum_{j=-\infty}^{T} \Big(\sum_{s=2}^{T-1} \nu_{\hi,2p}(X_{s-j})   + \sum_{t=2}^{T} \nu_{\hi,2p}(X_{t-j})  \Big)^2
\\& \le \frac{1}{T^2} 4 K^2_{2p} ( \max_{1\le t \le T} \snorm{\phi_{t,T}}_{\ell_p})^2 T \boldsymbol{\Delta}^4_{2p,1,0}  =  O((b_T T)^{-1}).
\end{align*}
From \autoref{eq:nullpart}, it is immediate that \eqref{eq:Fvar1} is of order $O(T^{-1})$ in $\op^p_{S_2}$. It therefore follows by Minkowski's inequality that
\[
 \norm{\hat{\F}_T^{(\lambda)}-\E\hat{\F}_T^{(\lambda)}}^2_{S_2,p} = O((b_T T)^{-1}).
\]
Let us then consider the bias. Observe that from \eqref{eq:relFQ}, stationarity yields
\begin{align*}
\E \hat{\F}_T^{\lambda}& =\frac{1}{2\pi T}\sum_{s,t=1}^{T} \E(X_s \otimes X_t)  w(b_T(t-s)) e^{\im \lambda (t-s)} 
 =\frac{1}{2\pi} \sum_{|h| \le T} w(b_Th)\frac{1}{T}\sum_{t =\max(1,1-h)}^{\min(T,T-h)}\cov(X_{t+h} \otimes X_t) e^{-\im \lambda h}
\\&=\frac{1}{2\pi} \sum_{|h| < T} w(bh)(1-\frac{|h|}{T})C_{h}e^{-\im \lambda  h}.
\end{align*}
Hence, using Minkowski's inequality we can bound the error by
\begin{align*}
\snorm{\E\hat{\F}_T^{(\lambda)}-\F^{\lambda}}_2& \le \bigsnorm{\underbrace{\frac{1}{2\pi} \sum_{|h| < T} (w(b_Th)-1) C_{h}e^{-\im\lambda h}}_{R_{o,\lambda}}}_{2}  +\bigsnorm{\underbrace{\frac{1}{2\pi} \sum_{|h| \ge T} C_{h}e^{-\im \lambda h}}_{R_{1,\lambda}}}_2 
\\& + \bigsnorm{\underbrace{\frac{1}{2\pi} \sum_{|h| <T} w(b_Th)\frac{|h|}{T}C_{h}e^{-\im \lambda h}}_{R_{2,\lambda}}}_2.  \tageq \label{eq:Fbias}
\end{align*}
Since $\sum_{h} \snorm{C_h}<\infty$, it is immediate that
\[
\sup_{\lambda}\snorm{R_{1,\lambda}}_2 \le \sum_{|h| \ge T} \snorm{C_{h}}_2 \to 0 \text{ as } T \to \infty.
\]
Since $w(\cdot)$ is bounded, the final term satisfies $\sum_{h \in \znum} w(bh) \snorm{C_h}_2 \le \sup_x|w(x)| \sum_{h \in \znum}\snorm{C_h}_2 < \infty$. Hence, by Kronecker's lemma $\sup_{\lambda}\snorm{R_{2,\lambda}}_2 \to 0$ as $T \to \infty$. Finally, provided that $b_T \to 0$ as $T \to \infty$ and since $\lim_{x\to 0}w(x) = w(0)$, we obtain $w(b_T h) \to w(0)=1$ as $T \to \infty$. Thus, $\sup_{\lambda}\snorm{R_{0,\lambda}}_2 \to 0$ from which asymptotic unbiasedness follows. Next, we prove part (ii) for which it remains to derive the order of the bias under the additional assumptions. Firstly, from Proposition \autoref{prop:sumCh}, the assumption $\sum_{h \in \znum} h \|P_0(X_h)\|_{\hi,2} < \infty$ implies that $\sum_{h \in \znum} h \snorm{C_h}_{\hi,2} <\infty$. Observe then that we can decompose the bias as 
\begin{align*}
\snorm{\E\hat{\F}_T^{(\lambda)}-\F^{\lambda}}_2& \le \bigsnorm{\underbrace{\frac{1}{2\pi} \sum_{|h| < 1/b_T} (w(b_Th)-1) C_{h}e^{-\im\lambda h}}_{R_{o,\lambda}}}_{2}  +\bigsnorm{\underbrace{\frac{1}{2\pi} \sum_{|h| \ge 1/b_T} C_{h}e^{-\im \lambda h}}_{R_{1,\lambda}}}_2
\\&  + \bigsnorm{\underbrace{\frac{1}{2\pi} \sum_{|h| <1/b_T} w(b_Th)\frac{|h|}{T}C_{h}e^{-\im \lambda h}}_{R_{2,\lambda}}}_2.  
\end{align*}
For the final term, we use that $\sup_x |w(x)| =O(1)$ and that $\sum_{h} |h|\snorm{C_h}<\infty$ in order to obtain
\begin{align*}
\sup_{\lambda}\snorm{R_{2,\lambda}}_2=  \bigsnorm{ \frac{1}{2\pi} \sum_{|h| <1/b_T} |w(b_Th)|\frac{|h|}{T}C_{h}}_2  \le \sup_{x}|w(x)|  \frac{1}{2\pi} \sum_{|h| <1/b_T}\frac{|h|}{T}  \snorm{C_{h}}_2 =O(\frac{1}{T}).
\end{align*}
Moreover, $\sup_{\lambda}\snorm{R_{1,\lambda}}_2 = O(b_T)$.
 Finally, since $w(x)-1  =O(x)$ as $x \to 0$ and $\sum_{h} |h|\snorm{C_h}<\infty$, we find for the first term
\begin{align*}
\sup_{\lambda}\snorm{R_{0,\lambda}}_2=  \bigsnorm{\frac{1}{2\pi} \sum_{|h| < 1/b_T} (w(b_T h)-1) C_{h}e^{-\im \lambda h}}_2   \le \sum_{h} O(b_T h) \snorm{C_h}_2 = O(b_T).
\end{align*}
\end{proof}

\begin{proof}[Proof of \autoref{thm:ANF}]
From \eqref{eq:relFQ}, we have $\hat{\F}^{\lambda} = (2\pi T)^{-1} \hat{\mathcal{Q}^{\lambda}}$ with $\phi_{T,(t-s)}^{\omega}=w(b_T(t-s)) e^{\im \omega (t-s)}$. Note that in this case $\snorm{\Phi_T}^2_{ F}:=\sum_{t=1}^{T}\sum_{s=1}^{T}w^2(b_T(t-s))$ and that $\varrho^2_T =\sum_{s=1}^{T} w^2(b_T s)$. For weight functions satisfying \autoref{as:Weights}, a change of variables and symmetry of the weight function in zero yield 
\begin{align*}
\sum_{t=1}^T \sum_{s=1}^{T} |\phi_{T,(t-s)}|^2
&=\sum_{t=1}^T \sum_{s=1}^{T} \Big| w(b_T(t-s)) e^{-\im \omega (t-s)}\Big|^2 = \sum_{t=1}^T \sum_{s=1}^{T}w^2(b_T(t-s)) 
\sim  \frac{T}{b_T} \int w^2(x) d x =\frac{T}{b_T} \kappa. \tageq \label{eq:Phibig}
\end{align*}
It therefore suffices to verify the conditions of \autoref{thm:funcdist}, which is given here for completeness but follows a standard argument \citep[see e.g.,][]{LiuWu10}. From \autoref{as:Weights}, it is obvious that $\mathrm{(ii)}$ of \autoref{as:phi} holds. Additionally, from \eqref{eq:Phibig}, it is immediate that $\varrho^2_T =\kappa O( b_T^{-1})$ and $\snorm{\Phi_T}^2_{ F} = O(T \varrho^2_T)$ so that $\mathrm{(i)}$ is also satisfied. For $\mathrm{(iii)}$, observe that we can write
\begin{align*}
\sum_{t=1}^{M/b_T} |w(b_T t )-w(b_T (t-1))|^2+ \sum_{t = M/b_T}^{T} |w(b_T t )-w(b_T (t-1))|^2 
\end{align*}
By assumption, the function $w$ is bounded and continuous except on a set of measure zero. Observe that there must exist $\epsilon>0$ such that $|w(b_T t )-w(b_T (t-1))| < \epsilon $ uniformly for $|t| \le M/b_T$, except for a finite number of points, $\varepsilon/b_T$, for some constant $\varepsilon>0$ . Because $w(\cdot)$ is bounded, the first term will thus be of order $o(1/b_T)$. 
For the second term, the length of the interval, $I_{M}$, converges to zero for fixed $b_T, T$ as $M \to \infty$. Since the summand is at most of order $b_T^{-1}$ under the stated conditions, the second term is of order $O(I_M/b_T) =o(\varrho^2_T)$. To verify $\mathrm{(iv)}$, we decompose the term as
\begin{align*} 
\sum_{j=1}^{T-1} \sum_{s =1}^{j-M/b_T} \big(\sum_{t=j+1}^{T}  w(b_T(s-t))w(b_T(j-t))\big)^2 +\sum_{j=1}^{T-1} \sum_{s = 1 \vee j-M/b_T}^{j-1} \big(\sum_{t=j+1}^{T} w(b_T(s-t))w(b_T(j-t) )\big)^2
\end{align*}
For the first term, the Cauchy-Schwarz inequality and the condition $\sup_{0 \le b \le 1} b \sum_{h \ge M/b} w^2(b h) \to 0$ as $M \to \infty$ yield
\begin{align*} 
& \sum_{j=1}^{T-1} \sum_{s =1}^{j-M/b_T} \sum_{t=j+1}^{T}  w^2(b_T(s-t))  \sum_{t=j+1}^{T} w^2(b_T(j-t)) 
 =\sum_{j=1}^{T-1} \sum_{t=j+1}^{T} w^2(b_T(j-t))  \sum_{s =1}^{j-M/b_T}  \sum_{t=M /b_T}^{T}  w^2(b_T(s-t)) 
\\& = O(T b^{-1}_T I_M b^{-1}_T) 
=o(\snorm{\Phi_T}^4_{ F}). 
\end{align*}
 Secondly,
\begin{align*} 
\sum_{j=1}^{T-1} \sum_{s = 1 \vee j-M/b_T}^{j-1} \sum_{t=j+1}^{T}  w^2(b_T(s-t))  \sum_{t=j+1}^{T} w^2(b_T(j-t)) =O(T b_T^{-1} \varrho^4_T) = O(T b_T^{-3}) =o(\snorm{\Phi_T}^4_{ F}) 
\end{align*}
where we used that $1/b_T = o(T)$.
\end{proof}

\begin{proof}[Proof of \autoref{lem:sampmean}]
Note that we can write the spectral density operator as
\begin{align*}
2\pi T \hat{\ddot{\F}}^{(\lambda)} &=\sum_{s,t=1}^{T}\Phi^{\lambda}_{t,s,T}\big( (X_s-\mu +\mu-\hat{\mu}) \otimes (X_t-\mu +\mu-\hat{\mu}) \big) 
\\& =2\pi T \hat{{\F}}^{(\lambda)} +  \sum_{s,t=1}^{T}\Phi^{\lambda}_{t,s,T}\big((\mu-\hat{\mu}) \otimes (X_t-\mu)\big) +  \sum_{s,t=1}^{T}\Phi^{\lambda}_{t,s,T}\big((X_s-\mu) \otimes (\mu-\hat{\mu})\big) \tageq \label{eq:fdot23}\\&  +  \sum_{s,t=1}^{T}\Phi^{\lambda}_{t,s,T}\big(( \mu-\hat{\mu}) \otimes (\mu-\hat{\mu})\big). \tageq \label{eq:fdot4}
\end{align*}
We therefore will show that the last two terms in \eqref{eq:fdot23} and the term in \eqref{eq:fdot4} are of lower order. For the second term of \eqref{eq:fdot23}, a change of variables, the properties of the tensor product and the Cauchy-Schwarz inequality yield
\begin{align*}
&\bignorm{\sum_{|h|< T}\sum_{t =\max(1,1-h)}^{\min(T,T-h)}w(b_T(h))e^{\im (h) \lambda}(\mu-\hat{\mu}) \otimes (X_t-\mu) }_{S_2,2} \\&=\bignorm{\sum_{|h|< T}w(b_T(h))e^{\im (h) \lambda}(\mu-\hat{\mu}) \otimes \sum_{t =\max(1,1-h)}^{\min(T,T-h)}(X_t-\mu) }_{S_2,2} 
\\& \le 
\sum_{|h|< T} w(b_T(h))\Big(\norm{\mu-\hat{\mu}}^2_{\hi,4} \bignorm{\sum_{t =\max(1,1-h)}^{\min(T,T-h)}(X_t-\mu) }^2_{\hi,4}  \Big)^{1/2}
\\& \le 
\sum_{|h|< T} w(b_T(h))\Big(O(T^{-1}) O(T)  \Big)^{1/2} = O(\frac{1}{b_T}),
\end{align*}
where we used \autoref{lem:lineq}(i) in order to obtain $ \norm{\mu-\hat{\mu}}^2_{\hi,4}= \norm{\frac{1}{T} \sum_{t=1}^{T}(X_t - \E X_T)}^2_{\hi,4} \le \frac{1}{T^2} T \boldsymbol{\Delta}_{4,2,0}  =O(\frac{1}{T})$
and to obtain $ \bignorm{\sum_{t=1}^{T-h}(X_t-\mu) }^2_{\hi,4}= O(T)$. The third term in \eqref{eq:fdot23} is similar. For \eqref{eq:fdot4}, Holder's inequality and  \autoref{lem:lineq}(i) yield
\begin{align*}
\bignorm{\sum_{s,t=1}^{T}\Phi^{\lambda}_{t,s,T}\big(( \mu-\hat{\mu}) \otimes (\mu-\hat{\mu})\big)}_{S_2,2}&  \le
\sum_{s,t=1}^{T}\snorm{\Phi^{\lambda}_{t,s,T}}_{\infty}\Big( \E\snorm{( \mu-\hat{\mu}) \otimes (\mu-\hat{\mu})}^2_{2}\big)^{1/2}
\\& \le \sum_{s,t=1}^{T}\snorm{\phi^{\lambda}_{t,s,T}}_{\infty} \norm{\mu-\hat{\mu}}^2_{\hi,4}= O(T b^{-1}_T \frac{1}{T}) = O(\frac{1}{b_T})
 \end{align*}
where we used that $\sum_{s,t=1}^{T}\snorm{\phi^{\lambda}_{t,s,T}}_{\infty} =\sum_{s,t=1}^{T}w(b_T(t-s)) =O(T b^{-1}_T)$. Hence, $\norm{\hat{\ddot{\F}}^{(\lambda)} -\hat{{\F}}^{(\lambda)}}_{S_2,2} =O(T^{-1} b^{-1}_T)$ and the exact argument shows that $\norm{\hat{\ddot{\F}}^{(\lambda)} -\hat{{\F}}^{(\lambda)}}_{S_2,\frac{p}{2}} =O(T^{-1} b^{-1}_T)$ for $p\ge 2$ provided the process is in $\op^{p}_H$. 
\end{proof}

{\begin{proof}[Proof of Proposition \ref{prop:eigcon}]
\blu{Consistency of the empirical eigenvalues follows immediately from \autoref{thm:consF}{(i)} together with the fact that $\sup_{j }|{\beta}_j^{\lambda}- \hat{\beta}_j^{\lambda}| \le  \snorm{\hat{\F}_T^{\lambda} -{\F}^{\lambda}}_\infty \le  \snorm{\hat{\F}_T^{\lambda} -{\F}^{\lambda}}_2$ \citep[see e.g.,][]{GGK90}.
To prove the consistency of the eigenprojectors, denote $\delta_j =\inf_{l\ne k}\big\vert\lamx{\omega}{k}-\lamx{\omega}{l}\big\vert$ and let $\mathrm{1}_{A_{j,T}}$ equals one if the event $A_{j,T}= \{\snorm{\hat{\F}_T^{\lambda} -{\F}^{\lambda}}_2 < \delta_j/4\}$ occurs and zero otherwise. From Proposition 3.1 of \cite{MM03}
\[
\hat{\Pi}^{\lambda}_j-{\Pi}^{\lambda}_j = \mathcal{T}_j(\hat{\F}_T^{\lambda} -{\F}^{\lambda}) + R_{j,T}
\]
where $\mathcal{T}_j$ is a bounded linear operator for each $j$ and where $R_{j,T}$ is a random $S_2(H)$-valued element which satisfies 
$\snorm{R_{j,T}}_2 \mathrm{1}_{A_{j,T}} \le 8 \delta^{-2}_j\snorm{\hat{\F}_T^{\lambda} -{\F}^{\lambda}}^2_2$. The result now follows again from \autoref{thm:consF}{(i)} since $P(\mathrm{1}_{A^\complement_{j,T}}) \to 0$.}
\end{proof}

\section{\blu{Proofs of \autoref{sec:sec5}}}
\begin{lemma}\label{lem:itproj}
Let $X, Y \in \op^{p}_{H}$ and set $Z = \otimes^{k}_{i=1} Z_i$ with $Z_i  \in \op^{p}_{H}$ for $i=1,\ldots, k$ for $k \in \mathbb{N}$. If $Z$ is $\G_j$-measurable, then
\begin{align*}
\E\big[ P_j(X) \otimes Y \otimes Z\big]&= \E \big[P_j(X) \otimes P_j(Y) \otimes Z \big].\tageq \label{eq:itproj}
\end{align*}
\end{lemma}
\begin{proof}
By linearity, $\G_j$-measurability of $P_j(X)$ and $Z$, respectively and the tower property
\begin{align*}
\E \big[P_j(X) \otimes P_j(Y) \otimes Z \big] 
 & = \E[\E[ P_j(X)  \otimes Y \otimes Z| \G_j]] - \E [ \E[P_j(X) \otimes \E[Y|\G_{j-1}] \otimes Z |\G_{j-1}]]
 \\&=   \E[ P_j(X) \otimes Y \otimes Z]  - \E[\E[P_j(X)| \G_{j-1}]\otimes  \E[Y|\G_{j-1}] \otimes Z ]. 
\end{align*}
The result follows then from noting that $P_j(\cdot)$ is a martingale difference sequence with respect to $\G_j$.
\end{proof}
\begin{proof}[Proof of \autoref{thm:cumrel4}]
It suffices to consider $t_{k-1} \ge t_{k-2} \ge \ldots t_1 \ge 0$ since
\[
\sum_{t_1,\ldots, t_{k-1} \in \znum}\snorm{\,\mathrm{cum}(X_{t_1},\ldots, X_{t_{k-1}},X_0)}_{2} = k! \sum_{t_{k-1} \ge t_{k-2} \ge \ldots t_1 \ge 0}\snorm{\,\mathrm{cum}(X_{t_1},\ldots, X_{t_{k-1}},X_0)}_{2}.
\]
We use \eqref{eq:cumtens} and 
first consider the set $|v_1|=k$. We may write
\begin{align*}
\E \big[ X_{t_1} \otimes \cdots \otimes X_{t_{k-1}}\otimes X_0] 
&=  \sum_{j_0 = -\infty}^{0} \sum_{j_1=-\infty}^{t_1}\cdots \sum_{j_{k-1}=-\infty}^{t_{k-1}}\E \big[ P_{j_1}(X_{t_1}) \otimes \cdots \otimes P_{j_{k-2}}(X_{t_{k-1}}) \otimes  P_{j_0}(X_0)\big], \tageq \label{eq:cum1}
\end{align*}
where the equality holds in Hilbert-Schmidt norm. Observe that 
the term will be zero unless the two largest elements of the set $J:=\{j_0,\ldots, j_{k-1}\}$ are equal; otherwise we can apply to the tower property and condition on a filtration for which all but one element are measurable. 
Without loss of generality, we consider the case $j_{k-2}= j_{k-3} \ge \ldots, j_1 \ge j_0$. In this case, \eqref{eq:cum1}  becomes
\begin{align*}
&  \sum_{j_0 = -\infty}^{0} \sum_{j_1=-\infty}^{t_1}\cdots \sum_{j_{k-3}=-\infty}^{t_{k-3}} \sum_{j_{k-2}=-\infty}^{t_{k-2}}\E \big[ P_{j_1}(X_{t_1}) \otimes \cdots \otimes P_{j_{k-2}}(X_{t_{k-2}})  \otimes P_{j_{k-2}}(X_{t_{k-1}}) \otimes  P_{j_0}(X_0)\big].
\end{align*}
Appplying \autoref{lem:itproj} iteratively, we can write,
\begin{align*}
& \snorm{\,\E \big[ P_{j_1}(X_{t_1}) \otimes \cdots \otimes P_{j_{k-2}}(X_{t_{k-2}})  \otimes P_{j_{k-2}}(X_{t_{k-1}}) \otimes  P_{j_0}(X_0)\big] }_2
\\& =\snorm{\, \E \big[  P_{j_0}\Bigg(P_{j_1}(X_{t_1}) \otimes P_{j_1}\Big(P_{j_{2}}(X_{t_{2}})\otimes P_{j_2}\big(  \cdots \otimes P_{j_{k-3}}\big(P_{j_{k-2}}(X_{t_{k-2}})  \otimes P_{j_{k-2}}(X_{t_{k-1}})\big) \big) \Big)  \Bigg)\otimes  P_{j_0}(X_0)\big]}_2. 
\end{align*}
To ease the exposition, we focus on $k=4$. Denote $P_{j,\{l\}}(\cdot) = \E[\cdot|\G_{j,\{l\}}]- \E[\cdot|\G_{j-1,\{l\}}]$. By Jensen's inequality, the last term is then bounded by
\begin{align*}
 \le \bignorm{ P_{j_0}\Big(P_{j_1}(X_{t_1}) \otimes P_{j_1}\big(P_{j_2}(X_{t_2}) \otimes P_{j_{2}}(X_{t_{3}}) \big)  \Big)}_{\hi^{\otimes 3},\frac{4}{3}}\|P_{j_0}(X_0)\|_{\hi,4} = A \cdot \|P_{j_0}(X_0)\|_{\hi,4}, 
\end{align*}
where we used the notation  $\|\cdot\|_{ \hi^{\otimes n},\ell} = (\E\|\cdot\|^\ell_{\otimes^n H})^{1/\ell}, \ell \ge 0$. By Jensen's inequality, $A$ is bounded by
\begin{align*}
& \le \bignorm{ P_{j_1}(X_{t_1}) \otimes P_{j_1}\big(P_{j_2}(X_{t_2}) \otimes P_{j_{2}}(X_{t_{3}}) \big) -P_{j_1,\{j_0\}}(X_{t_1}) \otimes P_{j_1,\{j_0\}}\big(P_{j_2}(X_{t_2}) \otimes P_{j_{2}}(X_{t_{3}}) \big)}_{\otimes^{3} \hi,\frac{4}{3}}
\\& 
\le \bignorm{  P_{j_1}(X_{t_1})- P_{j_1,\{j_0\}}(X_{t_1})}_{\hi,4} \bignorm{P_{j_1}\big(P_{j_2}(X_{t_2}) \otimes P_{j_{2}}(X_{t_{3}}) \big) }_{S_2,2}
\\& + \bignorm{P_{j_1,\{j_0\}}(X_{t_1})}_{\hi,4} \bignorm{P_{j_1}\big(P_{j_2}(X_{t_2}) \otimes P_{j_{2}}(X_{t_{3}}) \big)- P_{j_1,\{j_0\}}\big(P_{j_2}(X_{t_2}) \otimes P_{j_{2}}(X_{t_{3}}) \big) }_{S_2,2}
\\& 
\le   \bignorm{P_{j_1,\{j_0\}}(X_{t_1})}_{\hi,4} \cdot J + \bignorm{  P_{j_1}(X_{t_1})- P_{j_1,\{j_0\}}(X_{t_1})}_{\hi,4}\Big\{ \bignorm{\big(P_{j_2}(X_{t_2})- P_{j_2,\{j_1\}}(X_{t_2}) \big)}_{\hi,4} \bignorm{P_{j_{2}}(X_{t_{3}})}_{\hi,4}
\\& \phantom{\bignorm{P_{j_1,\{j_0\}}(X_{t_1})}_{H,k} \cdot J + \bignorm{  P_{j_1}(X_{t_1})- P_{j_1,\{j_0\}}(X_{t_1})}_{\hi,4}\Big\{} + \bignorm{P_{j_2,\{j_1\}}(X_{t_2})}_{\hi,4}  \bignorm{P_{j_{2}}(X_{t_{3}}) - P_{j_{2},\{j_1\}}(X_{t_{3}})}_{\hi,4}\Big\},
 \end{align*}
 where $J$ is given by
\begin{align*}
\bignorm{ P_{j_2}(X_{t_2}) \otimes P_{j_{2}}(X_{t_{3}})-P_{j_2,\{j_1\}}(X_{t_2}) \otimes P_{j_{2}\{j_1\}}(X_{t_{3}})  -P_{j_2,\{j_0\}}(X_{t_2}) \otimes P_{j_2,\{j_0\}}(X_{t_{3}}) \big)+P_{j_2,\{j_0,j_1\}}(X_{t_2}) \otimes P_{j_2,\{j_0,j_1\}}(X_{t_{3}}) \big) }_{S_2,2}.
\end{align*}
Using that 
\begin{align*}x_1 y_1-x_2y_2-x_3 y_3+x_4 y_4 &=(x_1-x_2- x_3+x_4)y_1 + x_4(y_4-y_3-y_2+y_1) 
\\&+ (x_4-x_2)(y_2-y_1)+ (x_3-x_4)(y_1-y_2),
\end{align*}
we can write $J = J_1 +J_2 +J_3 +J_4$ with
\begin{align*}
J_1 & \le \bignorm{P_{j_2}(X_{t_2}) -P_{j_2,\{j_1\}}(X_{t_2})-P_{j_2,\{j_0\}}(X_{t_2}) +P_{j_2,\{j_0,j_1\}}(X_{t_2})}_{\hi,4} \bignorm{P_{j_{2}}(X_{t_{3}})}_{\hi,4}
\\J_2 & \le \bignorm{P_{j_2}(X_{t_3}) -P_{j_2,\{j_1\}}(X_{t_3})-P_{j_2,\{j_0\}}(X_{t_3}) +P_{j_2,\{j_0,j_1\}}(X_{t_3})}_{\hi,4} \bignorm{P_{j_2,\{j_0,j_1\}}(X_{t_{2}})}_{\hi,4}
\\J_3 & \le\bignorm{P_{j_2,\{j_1\}}(X_{t_2})-P_{j_2,\{j_0,j_1\}}(X_{t_2})}_{\hi,4}\bignorm{P_{j_{2}}(X_{t_{3}})- P_{j_{2},\{j_1\}}(X_{t_{3}}) }_{\hi,4}
\\J_4 & \le\bignorm{P_{j_2,\{j_0\}}(X_{t_2}) -P_{j_2,\{j_0,j_1\}}(X_{t_2})}_{\hi,4}\bignorm{P_{j_{2}}(X_{t_{3}})-P_{j_{2},\{j_1\}}(X_{t_{3}}) }_{\hi,4}.
\end{align*}
By the contraction property of the conditional expectation
\begin{align*}
\bignorm{P_{j_k}(X_t)+\sum_{i=1}^{k-1}(-1)^{i} \sum_{1\le l_1 < \ldots <l_{i} \le{k-1}} P_{j_k,\{j_{l_1},\ldots,j_{l_i}\}}(X_t)}_{\hi,p}  \le \bignorm{X_t+\sum_{i=1}^{k}(-1)^{i} \sum_{1\le l_1 < \ldots <l_{i}\le k} \E[X_t |\G_{\{j_{l_1},\ldots,j_{l_i}\}}]}_{\hi,p}, \tageq \label{eq:tightbound}
\end{align*}
and therefore
\begin{align*}
& \bigsnorm{ \, \E \Big[ P_{j_1}(X_{t_1}) \otimes P_{j_2}(X_{t_2}) \otimes P_{j_{2}}(X_{t_{3}})  \otimes  P_{j_0}(X_0)\Big] }_{S_2(S_2(H))}
\\&\le \nu_{\hi,4}(-j_0)  \nu_{\hi,4}(t_1-j_1) \Big\{\nu_{\hi,4}(t_2-j_2,t_2-j_1,t_2-j_0) \nu_{\hi,4}(t_3-j_2) +  \nu_{\hi,4}(t_3-j_2,t_3-j_1,t_3-j_0) \nu_{\hi,4}(t_2-j_2) 
\\&\phantom{\nu_{\hi,4}(-j_0)  \nu_{\hi,4}(t_1-j_1) \cdot}+  \nu_{\hi,4}(t_2-j_0,t_2-j_2)  \nu_{\hi,4}(t_3-j_1,t_3-j_2)  +  \nu_{\hi,4}(t_2-j_1,t_2-j_2)  \nu_{\hi,4}(t_3-j_1,t_3-j_2)  \Big\}
\\&+ \nu_{\hi,4}(-j_0) \nu_{\hi,4}(t_1-j_0,t_1-j_1)\Big\{\nu_{\hi,4}(t_2-j_1,t_2-j_2) \nu_{\hi,4}(t_3-j_2) + \nu_{\hi,4}(t_2-j_2)  \nu_{\hi,4}(t_3-j_1,t_3-j_2)\Big\}.
\end{align*}
Denote  $\boldsymbol{\Delta}_{\hi,4}(k):=\sum_{j_0,\ldots, j_{k-1} =0}^{\infty} \nu_{\hi,4}(j_0,\ldots, j_{k-1})$. Tedious calculations then show that, using \eqref{eq:sufcon}, we obtain
\begin{align*}
& \sum_{t_3 \ge t_2 \ge t_1 \ge 0}\sum_{j_0 = -\infty}^{0} \sum_{j_1=-\infty}^{t_1}\sum_{j_2=-\infty}^{t_2} \sum_{j_{2}=-\infty}^{t_{2}}\bigsnorm{\E \Big[ P_{j_1}(X_{t_1}) \otimes P_{j_2}(X_{t_2}) \otimes P_{j_{2}}(X_{t_{3}})  \otimes  P_{j_0}(X_0)\Big] }_{2}
\\& \le \big(\boldsymbol{\Delta}_{\hi,4}\big)^2 \Big\{2\boldsymbol{\Delta}_{\hi,4}(3) \boldsymbol{\Delta}_{\hi,4} + 2\big(\boldsymbol{\Delta}_{\hi,4}(2)\big)^2\Big\} + \boldsymbol{\Delta}_{\hi,4}(2)  \boldsymbol{\Delta}_{\hi,4} \Big\{2\boldsymbol{\Delta}_{\hi,4}(2) \boldsymbol{\Delta}_{\hi,4}\Big\} < \infty.
\end{align*}
For the second order terms, consider for example the term
\[
\sum_{t_3 \ge t_2 \ge t_1 \ge 0} \bigsnorm{\Pi_{1324}\Big(\E \big[ X_{t_1} \otimes X_{t_3}]  \otimes \E[X_{t_2}\otimes X_0]\Big)}_{2}.\]
By orthogonality of the projections and a similar derivation as above yields that this is bounded by
\begin{align*}
&\le \sum_{t_3 \ge t_2 \ge t_1 \ge 0} \sum_{j_0 = -\infty}^{0} \sum_{j_1=-\infty}^{t_1} \bigsnorm{\Pi_{1324}\Big(\E\big[ P_{j_{0}}(X_{t_{2}}) \otimes  P_{j_0}(X_0)\big] \otimes\E \big[ P_{j_1}(X_{t_1}) \otimes  P_{j_{1}}(X_{t_{3}}) \big]\Big)}_{2}
\\&  \le \sum_{t_3 \ge t_2 \ge t_1 \ge 0} \sum_{j_0 = 0}^{\infty} \sum_{j_1=-t_1}^{\infty} \nu_{\hi,2}(t_2+j_0)\nu_{\hi,2}(j_0)\nu_{\hi,2}(t_1+j_1)\nu_{\hi,2}(t_3+j_1)
\\& \le \sum_{t_3 \ge t_2 \ge t_1 \ge 0} \sum_{j_0 = 0}^{\infty} \sum_{j_1=-t_1}^{\infty} \nu_{\hi,2}(t_2+j_0,t_2+j_1)\nu_{\hi,2}(j_0)\nu_{\hi,2}(t_1+j_1)\nu_{\hi,2}(t_3+j_1)
\\&
\le \boldsymbol{\Delta}^3_{\hi,2} \boldsymbol{\Delta}_{\hi,2}(2) < \infty.
\end{align*}
The other terms are similar and the proof is therefore omitted.
\end{proof}
\begin{proof}[Proof of Proposition \ref{prop:sufcon}]
Observe that for any $k \in \{1,\ldots,n\}$, Minkowski's inequality implies
\begin{align*}
\nu_{X,\hi,p}(j_{1},\ldots,j_{n}): &= \bignorm{X_0+\sum_{i=1}^{n}(-1)^{i} \sum_{1\le l_1 < \ldots <l_{i}\le n} \E[X_0 |\G_{0,\{-j_{l_1},\ldots,-j_{l_i}\}}]}_{\hi,p} 
\\ & \le \bignorm{X_0+\sum_{i=1}^{n-1}(-1)^{i} \sum_{\substack{l_1 < \ldots <l_{i}
\in \{1,\ldots,n\}\setminus \{k\}}} \E[X_0 |\G_{0,\{-j_{l_1},\ldots,-j_{l_i}\}}]}_{\hi,p}  \\&+ \bignorm{\E[X_0|\G_{0,\{-j_{k}\}}]+\sum_{i=1}^{n-1}(-1)^{i} \sum_{\substack{l_1 < \ldots <l_{i}
\in \{1,\ldots,n\}\setminus \{k\}}} \E[X_0 |\G_{0,\{-j_{k}\},\{-j_{l_1},\ldots,-j_{l_i}\}}]}_{\hi,p},
\end{align*}
since the second term is equal to the first, it then easily follows that 
\[\nu_{X,\hi,p}(j_1,\ldots,j_{n}) \le 2 \min_{1\le l_1< \dots <l_{n-1}\le n} \nu_{X,\hi,p}({j_{l_{1}},\ldots, j_{l_{n-1}}}). \tageq \label{eq:minbound}\] 
Applying \eqref{eq:minbound} consecutively gives the result.
\end{proof}

\end{document}